\documentclass[a4paper, 11pt]{article}

\usepackage[utf8]{inputenc}
\usepackage{amsmath}
\usepackage{amsthm}
\usepackage{amsfonts}
\usepackage{amssymb}
\usepackage{amstext}
\usepackage{dsfont}
\usepackage{graphicx}
\usepackage{color}
\usepackage[colorlinks]{hyperref}
\usepackage{epigraph}
\usepackage[left=3cm,top=2cm,bottom=3.5cm,right=2.5cm,nohead,nofoot]{geometry}
\usepackage{tikz}
\usetikzlibrary{arrows,shapes,matrix}

\graphicspath{{Figures/}{.}}%

\title{ On the maximum of the C$\beta$E field }
\author{Reda \textsc{Chhaibi}        \footnote{\texttt{reda.chhaibi@math.univ-toulouse.fr}},
        Thomas \textsc{Madaule}      \footnote{\texttt{thomas.madaule@math.univ-toulouse.fr}},
        Joseph \textsc{Najnudel}     \footnote{\texttt{joseph.najnudel@math.univ-toulouse.fr}}
       } 

\allowdisplaybreaks[4]

\DeclareMathOperator{\eqlaw}{\stackrel{\Lc}{=}}

\DeclareMathOperator{\Var}{Var}

\def\half{\frac{1}{2}}
\def\quart{\frac{1}{4}}

\def\D{\Delta}

\def\e{\mathtt e}
\def\k{\mathtt k}

\def\N{{\mathbb N}}
\def\Z{{\mathbb Z}}
\def\Q{{\mathbb Q}}
\def\R{{\mathbb R}}
\def\C{{\mathbb C}}
\def\U{{\mathbb U}}
\def\D{{\mathbb D}}
\def\G{{\mathbb G}}

\def\P{{\mathbb P}}
\def\E{{\mathbb E}}

\def\Bc{{\mathcal B}}

\def\Dc{{\mathcal D}}

\def\Fc{{\mathcal F}}
\def\Gc{{\mathcal G}}

\def\Lc{{\mathcal L}}
\def\Nc{{\mathcal N}}
\def\Oc{{\mathcal O}}
\def\Pc{{\mathcal P}}

\def\I{\mathfrak I}

\newtheorem{thm}{Theorem}[section]
\newtheorem{proposition}[thm]{Proposition}
\newtheorem{corollary}[thm]{Corollary}

\newtheorem{conjecture}[thm]{Conjecture}

\newtheorem{lemma}[thm]{Lemma}

\newtheorem{rmk}[thm]{Remark}

\numberwithin{equation}{section}

\begin{document}

\maketitle

\begin{abstract}
In this paper, we investigate the extremal values of (the logarithm of) the characteristic polynomial of a random unitary matrix whose spectrum is distributed according the Circular Beta Ensemble (C$\beta$E). More precisely, if $X_n$ is this characteristic polynomial and $\U$ the unit circle, we prove that:
$$\sup_{z \in \U} \Re \log X_n(z) = 
  \sqrt{\frac{2}{\beta}} 
  \left(\log n - \frac{3}{4} \log \log n + \Oc(1) \right)\ ,$$
as well as an analogous statement for the imaginary part. The notation $\Oc(1)$ means that the corresponding family of random variables, indexed by $n$, is tight. This answers a conjecture of Fyodorov, Hiary and Keating, originally formulated for the $\beta=2$ case, which corresponds to the CUE field. 
\end{abstract}

\indent
\hrule
\tableofcontents
\indent
\hrule

\section{Introduction}
\label{section:introduction}
Consider $n$ identically charged particles on the unit circle $\U$ with a logarithmic interaction potential and inverse temperature parameter $\beta$. This gives rise to a probability distribution on $\U^n$ given by:
\begin{align}
\label{eq:beta_ens}
\frac{1}{(2\pi)^n Z_{n,\beta}} \prod_{1 \leq j < k \leq n}
\left|e^{ i \theta_j } - e^{ i \theta_k} \right|^{\beta} d\theta_1 \dots d\theta_n\ .
\end{align}
Such a probability distribution is called the Circular $\beta$ ensemble. In the paper \cite{bib:KN04}, Killip and
Nenciu give matrix models made of CMV matrices whose spectrum is distributed exactly according to the distribution \eqref{eq:beta_ens}.

From the Circular $\beta$ ensemble, one can construct the corresponding characteristic polynomial $X_n$, given by: 
\begin{align}
\label{def:Z_n}
X_n(z) =  \prod_{j=1}^{n} (1 - z e^{i \theta_j} ) =  \prod_{j=1}^n (1 - \lambda_j z), 
\end{align}
where $\lambda_j = e^{i \theta_j}$.

Because of the multiplicative structure of $X_n$, it is natural to consider its logarithm, defined on the simply connected domain
$\C \backslash \left( \{\overline{\lambda_1} , \dots, \overline{ \lambda_n}\} [1, \infty) \right)$, and
given as follows: 
$$\log X_n(z) = \sum_{j=1}^n \log (1 - \lambda_j z),$$
where, in order to avoid ambiguity, the branch of $\log (1 - \lambda_j z)$ is chosen in such a way that $\Im \log (1 - \lambda_j z) \in (-\pi, \pi)$: with this choice, $\log X_n$ is continuous on the domain where it is defined. 

The most classical case corresponding to this setting is the Circular Unitary Ensemble (CUE),
given by the eigenvalues of a Haar-distributed random matrix on the unitary group $U(n)$. In this case, the joint
distribution of the eigenvalues is given by Equation \eqref{eq:beta_ens} specialized to $\beta = 2$. This particular
value of $\beta$ implies that $(\lambda_j)_{1 \leq j \leq n}$ is a {\it determinantal process}: more precisely, for $1 \leq p \leq n$,
and for a measurable  function $f$ from $\U^p$ to $\mathbb{R}_+$, 
$$\E \left[ \sum_{j_1 \neq j_2 \neq \dots \neq j_p } f(\lambda_{j_1}, \dots, \lambda_{j_p}) \right] 
= \int_{\U^n} 
f(z_1, \dots, z_p)  \det
\left( K_n(z_j, z_k) \right)_{1 \leq j, k \leq p} d\mu(z_1) \dots d \mu(z_p),$$
where
$$K_n(z,z') = \sum_{\ell=0}^{n-1} (\bar{z} z')^{\ell},$$
 $\mu$ denoting the uniform probability measure on $\U$. 

This determinantal structure can be used to get exact formulas for moments of linear statistics of the $\lambda_j$'s. A particularly remarkable formula has been proven by  Diaconis and Shahshahani \cite{bib:DS94}, by using some representation theory of the unitary group $U(n)$, namely the combinatorics of Schur functions. It can also be proven by using the determinantal formula above. This result says the following: if the matrix $M_n$ is Haar-distributed on $U(n)$, if $(\Nc_k^{\C})_{k \geq 1}$ is a family of i.i.d. {\it complex} Gaussians, such that
 $$\E [\Nc_k^{\C}] 
 = \E [(\Nc_k^{\C})^2] = 0, \, 
 \E [|\Nc_k^{\C}|^2] = 1,
 $$
 and if $\Nc_{-k}^{\C} 
 := \overline{\Nc_{k}^{\C}} $ then 
 $$\E \left[\prod_{r=1}^m\operatorname{Tr}(M_n^{k_r}/\sqrt{k_r}) \right] 
  = \E \left[\prod_{r=1}^m 
  \Nc_{k_r}^{\C} \right].$$
 for all non-zero integers $k_1, \dots, k_m$ such that $\sum_{j=1}^m |k_m| \leq n$. Notice that moments match exactly those of Gaussians up to a certain order depending on $n$. This matching of moments is not exact for general $\beta$ as shown by Jiang and Matsumoto \cite{bib:JM15} using the combinatorics of Jack functions.
 
 Nevertheless, the following convergence in distribution remains, for the finite-dimensional marginals: 
 $$\left(\operatorname{Tr}(M_n^{k})/\sqrt{k} \right)_{k \geq 1} 
 \underset{n \rightarrow \infty}{\longrightarrow} \sqrt{\frac{2}{\beta}}\left( \Nc_{k}^{\C} \right)_{k \geq 1}.
$$
As suggested by the exact matching of moments, the speed of this convergence is super-exponential in the case of $\beta=2$ 	(\cite{bib:J97}). From the formula: 
$$\log X_n(z) = - \sum_{k = 1}^{\infty} 
\frac{\operatorname{Tr}(M_n^k) z^k }{k},$$
one deduces that the field $(\log X_n(z))_{z \in \mathbb{D}}$ on the open unit disc converges in distribution to a centered complex Gaussian field $(\mathbb{G}(z))_{z \in \mathbb{D}}$, whose correlation structure is given by 
$$\E [ \mathbb{G}(z)\mathbb{G}(z')]
 = 0, \; \E \left[ \overline{\mathbb{G}(z)} \mathbb{G}(z') 
 \right] = - \frac{2}{\beta} \log \left( 1 - \bar{z} z' \right).$$
It is easy to see that such a field has the series representation:
\begin{align}
\label{eq:G_series}
\forall z \in \D, \ \G(z) & := \sqrt{\frac{2}{\beta}} \sum_{k=1}^\infty \frac{\Nc_k^\C}{\sqrt{k}} z^k \ ,
\end{align}
where the $\left( \Nc_k^\C , k \geq 1 \right)$ are i.i.d standard complex Gaussians.

For sake of simplicity, we focus in next paragraphs on the $\beta=2$ case, which has been much more studied in the literature. The variance of $\mathbb{G}$ has a logarithmic singularity when we approach the unit circle, and $\log X_n(z)$ does not converge for $|z| = 1$ to a bona fide function. More precisely, Keating and Snaith \cite{bib:KSn} have proven the convergence in distribution: 
\begin{align}
\label{eq:ks_logX_lim}
\frac{\log X_n(z)}{\sqrt{\log n}}
& \ \underset{n \rightarrow \infty}{\longrightarrow} \Nc_1^{\C}. 
\end{align}
In other words, for $|z| = 1$, $\log X_n(z)$ behaves like a complex Gaussian variable with total variance $\log n$. 
 
On the other hand, Hughes, Keating and O'Connell \cite{bib:HKO} have proven that  one can still get a convergence of $(\log X_n(z))_{z \in \U}$ without  normalization, if we don't ask for the limiting object to be a well-defined function at single points on the unit circle. In fact, for every $\epsilon > 0$, the expression in Eq. \eqref{eq:G_series} gives a well-defined  object in the Sobolev space $H^{-\epsilon}$. It has a meaning only upon convoluting with a sufficiently regular function. In that sense, we have the following convergence in law, for the corresponding random distributions on the unit circle $\U$: 
\begin{align}
\label{eq:log_X_CV}
(\log X_n(z))_{z \in \U}
& \underset{n \rightarrow \infty}{\longrightarrow} \G .
\end{align}

\paragraph{Extremal statistics:} By an explicit computation (\cite{bib:FHK}), it is possible to prove that 
$$ \Var\left( \log \left|X_n(z)\right| \right) \sim \half \log n $$
and that the correlation saturates at the scale $|\theta - \theta'| \sim \frac{1}{n}$. By correlation saturation, we simply mean that the order of magnitude of the correlation remains the same for $\theta - \theta'$ going to zero and for $|\theta - \theta'| \sim 1/n$. Thus, the naive analogy consists in approximating the function $\log \left|X_n(z)\right|$ on the circle by its values at $\Oc(n)$ points. Each point would be assigned an independent copy of a Gaussian with variance $\half \log n$, in accordance with Eq. \eqref{eq:ks_logX_lim}. It is classical that the maximum of such independent Gaussians is of order $\log n$, which intuitively explains the leading order. One hopes to show that the proof of this first order does not depend on the correlation structure. The story is different for the second order term. If not for the correlations, the asymptotic expansion would be $ \log n - \frac{1}{4} \log \log n$ by approximating the field by $\Oc(n)$ independent Gaussians.

From this discussion, one sees that $\left( \log \left|X_n(z)\right| \right)_{z \in \U}$ is a complicated (yet integrable) regularization of the log-correlated Gaussian field $\left( \G(z) \right)_{z \in \U}$. In terms of global features, it is in every way similar to the ``cone construction'' (see Arguin, Zindy \cite[Fig. 1]{bib:AZ14}): correlation is of logarithmic nature and saturates at the scale $\frac{1}{n}$. In that universality class, one expects:
$$ \max_{z \in \U} \log \left|X_n(z)\right| \sim \log n - \frac{3}{4} \log \log n,$$
which is an established result in many cases. In the case of tree models such as branching Brownian motion and branching random walks, the result holds at fairly large level of generality (See \cite{HSh09, AShi10, Aid11}). By ``tree model'', we mean a model where a tree structure is apparent and explicit. Among non-tree models, where one needs to identify an approximate branching structure, the result holds for log-correlated Gaussian fields \cite{bib:Mad13, RDZ15}, discrete GFF (Gaussian Free Fields) as described in \cite{BZe10, BDZ13}, and cover times \cite{BK14}. The constant $\frac{3}{4}$ is strongly related to such an underlying hierarchical structure.

It is also worth mentioning that the field $\left( \Re \G\left( e^{i\theta} \right) \right)_{\theta \in [0, 2\pi)}$ can be regularized into a Gaussian field by evaluating the random field $z \mapsto \Re \G(z)$ in the interior of the unit disk. The existing technology for Gaussian log-correlated fields is applicable to $\theta \mapsto \Re \G\left( e^{-\frac{1}{n}+i\theta} \right)$, with mild modifications. It yield the expected results for this simple regularization where all the Random Matrix Theory is lost. Here, we will be exclusively concerned with $\left( \log X_n(z)\right)_{z \in \U}$.

In two very insightful papers \cite{bib:FyKe, bib:FHK}, Fyodorov, Hiary and Keating formulate the following conjecture.
\begin{conjecture}
\label{conjecture:max}
$$\sup_{z \in \U} \log |X_n(z)|
 - \left( \log n - \frac{3}{4} \log \log n\right) \underset{n \rightarrow \infty}{\longrightarrow} \frac{1}{2} (K_1 + K_2)$$
where $K_1$ and $K_2$ are two independent Gumbel random variables. 
\end{conjecture}
Indeed, in the notations of these papers, $-\left( K_1 + K_2 \right)$ is a random variable with density
$$ p(x) = 2 e^x K_0(2 e^{\frac{x}{2}} ) = e^x \int_\R dy e^{-e^{x/2} \cosh(y) } \ .$$
Here $K_0$ is the modified Bessel function of the second kind. A quick computation of moment generating functions allows us to realize that we are dealing indeed with minus the sum of two independent Gumbel random variables.

It is a very challenging problem to prove (or disprove) such a precise conjecture. However, progress has recently been made in this direction. In a first breakthrough \cite{bib:ABB}, Arguin, Belius and Bourgade have proven that 
$$\frac{\sup_{z \in \U} \log |X_n(z)|}{\log n} \underset{n \rightarrow \infty}{\longrightarrow} 1$$
in probability, and shortly afterwards, using different methods, Paquette and Zeitouni \cite{bib:PZ} have refined this result by showing: 
$$\frac{\sup_{z \in \U} \log |X_n(z)| - \log n}{\log \log n}
\underset{n \rightarrow \infty}{\longrightarrow} - \frac{3}{4}$$
in probability. The refinement given by Paquette and Zeitouni is an important progress as the constant $\frac{3}{4}$ morally confirms the existence of hierarchical structures.

Note that the comparison between $\log X_n$, and the Gaussian field $\G$ can only be accurate in the macroscopic or the mesoscopic scale, i.e. large with respect to $1/n$. In the microscopic scale, the behavior of $\log X_n$ is not Gaussian anymore, and it has been studied by Chhaibi, Najnudel and Nikeghbali in \cite{CNN}. In this paper, the author have proven the following convergence in distribution on the space of holomorphic functions:
$$ \left( \frac{X_n(e^{2 i \pi z}/n)}{X_n(1)}  \right)_{z \in \C} 
\underset{n \rightarrow \infty}{\longrightarrow} \xi_\infty(z),$$
where $\xi_{\infty}$ is a random holomorphic function whose zeros are all real and form a determinantal sine-kernel process. It can be interesting to study the behavior of the large values of $|\xi_{\infty}|$, and to see if their behavior has an influence in the limiting behavior of the maximum of $|X_n|$ on the unit circle. 

\paragraph{Multiplicative chaos point of view:} Once the convergence \eqref{eq:log_X_CV} of $\log |X_n(z)|$ towards the Gaussian field $\G$ is stated, one can ask if it is possible to exponentiate in order to get results of convergence for the field $\left( |X_n(z)| \right)_{z \in \U}$. Such an exponential cannot be done in a classical way, since $\G(z)$ is not well-defined for a single $z \in \U$. 

In the 1980s, Kahane \cite{K} has constructed such an exponential as a random multifractal measure, called the Gaussian Multiplicative Chaos, which has also been used in the mathematical study of the two-dimensional quantum gravity (for example, see \cite{KPZ} and \cite{DS2}). For a survey, we recommend \cite{RV}. This measure can be defined as follows. For a given parameter $\alpha > 0$, and for all integers $L \geq 1$, one defines
 $$\G^L (z) 
   := \sum_{1 \leq k \leq L} \frac{\Nc_{k}^{\C} z^k}{\sqrt{k}}$$
and then the measure $\mu_L^{(\alpha)}$ whose density with respect to the uniform measure $\mu$ on $\U$ is given by 
$$\frac{d \mu_L^{(\alpha)}}{d \mu} (z)  
=  \frac{ e^{\alpha \Re \G^{L} (z) }}
        { \E\left[e^{\alpha \Re \G^{L} (z)} \right] }.$$
By using martingale arguments, one can prove that $\mu_L^{(\alpha)}$ converges almost surely to a finite random measure
$\mu^{(\alpha)}$ on $\U$. This construction gives a phase transition at $\alpha = 2$. For $\alpha \geq 2$, the limiting measure is almost surely equal to zero, whereas it is non-degenerate for $\alpha \in (0,2)$: in this case, it defines the Gaussian Multiplicative Chaos corresponding to the parameter $\alpha$. 
 
It has recently been proven by Webb \cite{W} that for $\alpha < \sqrt{2}$, one has the following result: if 
$\mu_{X_n}^{(\alpha)}$ is defined by 
$$\frac{d \mu_{X_n}^{(\alpha)}}{d \mu} 
:= \frac{|X_n(z)|^{\alpha}}{\E [|X_n(z)|^{\alpha}] },$$
then we have the convergence in distribution
$$\mu_{X_n}^{(\alpha)} \underset{n \rightarrow \infty}{\longrightarrow} 
\mu^{(\alpha)},$$
in the  space of Radon measures on the unit circle, equipped with the topology of weak convergence. 

This result has also been previously conjectured by Fyodorov, Hiary and Keating (see \cite{bib:FHK} and \cite{bib:FyKe}), and it is believed to remain true for all $\alpha < 2$, the restriction $\alpha < \sqrt{2}$ being only technical. In \cite{bib:FyKe}, the authors also study the dependency in  $\alpha$ of the behavior of the measure $\mu_{X_n}^{(\alpha)}$, in particular of their moments, and heuristically, they also find a phase transition at $\alpha = 2$. For $\alpha \geq 2$, the behavior of $\mu_{X_n}^{(\alpha)}$ is dominated by the large values of $|X_n|$ on the unit circle. From these heuristics, Fyodorov, Hiary and Keating motivate their Conjecture \ref{conjecture:max}. 

\paragraph{Number theoretic motivations:} Another interesting point corresponds to the analogy which is conjectured between the behaviors of $X_n$ and the Riemann zeta function. Some conjectures on the moments of $\zeta$, directly related to corresponding results on $X_n$, are given by Keating and Snaith in \cite{bib:KSn}. Moreover, in \cite{bib:FHK}, Fyodorov, Hiary and Keating make the following conjecture: for $U$ uniformly distributed on $[0,1]$, the family 
$$\left(\sup_{h \in [0,1]} \log \left|\zeta \left(\frac{1}{2} + i 
(UT + h)\right)\right|  - \log \log T + \frac{3}{4} \log  \log \log T\right)_{T \geq 3}$$
of random variables is tight.  Such a conjecture is consistent with the analogy between $\zeta$ and $X_n$, with the classical correspondance between $n$ and $\log T$.  Recently, Arguin, Belius and Harper \cite{ABH15} have proven a part of the conjecture by Fyodorov and Keating for a randomized model of the Riemann function. More precisely, they have proven that if $(U_p)_{p \in \Pc}$ is a family of i.i.d. uniform variables on $\U$, indexed by the set $\Pc$ of prime numbers, one has 
$$\frac{\sup_{h \in [0,1]}\left( \sum_{p \in 
\Pc \cap [0,T]}  \frac{\Re (U_p p^{-ih})}{\sqrt{p}}\right) - \log \log T}{\log \log \log T} \underset{T \rightarrow \infty}{\longrightarrow} -\frac{3}{4}.$$
in probability. 

\paragraph{Our result:} The main theorem of the present paper answers Conjecture \ref{conjecture:max} up to the third order, and in the setting of the Circular Beta Ensemble where $\beta > 0$ is not necessarily equal to $2$. For $\beta \neq 2$, the point process of the eigenvalue is not determinantal, and then it is more difficult to get exact formulas for this model. The tool we will use to deal with this problem is the theory of orthogonal polynomials on the unit circle, described for example in the book by Simon \cite{bib:Sim}. In \cite{bib:KN04}, Killip and Nenciu give the construction of an ensemble of random matrices whose eigenvalue distribution follows the C$\beta$E, and prove that the characteristic polynomial can be written as the last term of a sequence of orthogonal polynomials whose parameters, called Verblunsky coefficients, have a distribution which is explicitly given. In the beautiful paper \cite{bib:KSt09}, Killip and Stoiciu use this model in order to deduce the existence of a limiting point 
process for the microscopic behavior of the C$\beta$E. More details are given in the next section, along with the notions we will need. 

The precise statement of our main result is the following.  
\begin{thm} \label{thm:main} 
If $\mathbb{U}' := \mathbb{U} 
\backslash \{\overline{\lambda_1}, \dots, 
\overline{\lambda_n}\}$, the following family of random variables: 
$$ \sqrt{\frac{\beta}{2}} \left( \sup_{z \in \mathbb{U}'} \Re \log X_n(z)  - \left( \log
 n - \frac{3}{4} \log \log n \right) \right)_{n \geq 2}$$ 
 and for $\sigma \in \{-1,1\}$, 
 $$ \sqrt{\frac{\beta}{2}} \left( \sup_{z \in \mathbb{U}'} (\sigma \Im \log X_n(z))  -\left( \log
 n - \frac{3}{4} \log \log n \right) \right)_{n \geq 2}$$
  are tight. 
\end{thm}
It seems reasonable to expect that these families of random variables have a limiting distribution, however, we are not sure about what this distribution should be. It is interesting to state the previous result with the imaginary part of the characteristic polynomial, since this gives some information about the number of points among $(\lambda_j)_{1 \leq j \leq n}$ which lie in a given arc of circle. In particular, we get the following corollary: 
\begin{corollary}
For $z_1, z_2 \in \U$, let $N_n(z_1,z_2)$ be the number of points of the C$\beta$E lying in the arc coming counterclockwise from $z_1$ to $z_2$, and let $N_{n}^{(0)}(z_1,z_2)$ be the expectation of $N_n(z_1,z_2)$, which is equal to the length of the arc multiplied by $n/2 \pi$. Then, the following family of random variables is tight:  
$$
\left( \pi \sqrt{ \frac{\beta}{8} } \sup_{z_1, z_2 \in \U}{|N_n(z_1,z_2) - N_{n}^{(0)} (z_1,z_2)|} 
-\left( \log
n - \frac{3}{4} \log \log n \right)  \right)_{n \geq 2}.  
$$
\end{corollary}
The values of $z_1$ and $z_2$ maximizing $|N_n(z_1,z_2) - N_n^{(0)}(z_1,z_2)|$ correspond to the extreme values of the imaginary part of $\log X_n$ on $\U$. 

\subsection*{Structure of the paper}
In section \ref{section:beta_and_opuc}, we start right away by presenting the general setting which hinges on the realization of the C$\beta$E via orthogonal polynomial techniques, as done by \cite{bib:KN04} and \cite{bib:KSt09}. The two essential players at that point are the Verblunsky coefficients and the Pr\"ufer phases. Then, we show that our main given in Theorem \ref{thm:main} can be reduced to the study of the auxiliary sequence of polynomials $\Phi_n^*$ thanks to Corollary \ref{corollary:phistar}. We go on proving the necessary estimates on Verblunsky coefficients and Pr\"ufer phases.

In section \ref{section:auxiliary_field}, we reduce the problem further to the study of an auxiliary field whose one-point marginals are Gaussian. 

Only then, we are ready to tackle the study of the maximum of the C$\beta$E field.  As customary, it is broken down into two steps: an asymptotic upper bound (Section \ref{section:upper_bound}) and an asymptotic lower bound (Section \ref{section:lower_bound}).

Section \ref{section:appendix} is an appendix containing classical estimates on Gaussian random walks and will only be invoked in the proofs of the upper and lower bounds.

\subsection*{Additional remarks}
In this subsection, we make a few remarks which go beyond the scope of this paper. 

From comparing to the cases which are better understood, e.g Gaussian log-correlated fields, the remainder in Theorem \ref{thm:main} is expected to have a non-universal limiting distribution which depends on the fine features of the model at hand. Moreover, it is understood that such fine properties of a fields's extrema are captured by the associated Multipliticative Chaos, for the critical exponent (see \cite{RV}). Thus, a necessary step would be to analyze the convergence to the Gaussian Multiplicative Chaos in the setting of the C$\beta$E, exactly as in the results of Webb \cite{W}.

A closely related question would be the microscopic landscape of the characteristic polynomial in the vicinity of its extrema. If the typical microscopic landscape has been described in \cite{CNN} as mentioned in the introduction, it is unclear whether this behavior remains typical in the neighborhood of extremal points. For branching Brownian motion, this microscopic panorama has been described by two groups of researchers (\cite{ABBS11}, \cite{ABK11}). For the discrete branching random walk, a similar description has been found by Madaule (\cite{Mad15}).  For the discrete GFF, a description of the large local maxima has been found by Biskup and Louidor \cite{BL16}.

Finally, it is natural to go beyond the characteristic polynomial of the C$\beta$E and examine other random polynomials. {\it In principle}, one can adapt our approach to the case of more general sequences of random orthogonal polynomials on the circle. More precisely, one would require the Verblunsky coefficients to be independent, rotationally invariant and within the critical decay regime of \cite[Theorem 1.7, (ii)]{bib:KSt09}. We chose to restrict ourselves to the C$\beta$E because of its relevance to Random Matrix Theory and because its Verblunsky coefficients are easily compared to Gaussians (Section \ref{section:auxiliary_field}). Developing the fine estimates we require for non-Gaussian random walks would have raised the technicality of the problem.

\subsection*{Acknowledgements}
The authors would like to acknowledge the review paper of N. Kistler \cite{bib:Kistler} which was very helpful in understanding hierarchical models. Also we thank A. Nikeghbali for pointing out the Verblunsky coefficients approach of Killip, Nenciu and Stoiciu for the C$\beta$E.

\subsection*{Notations}
Equality in law between random variables is denoted by $\eqlaw$. We will also make use of the Vinogradov symbol:
$$ f \ll g \Leftrightarrow f = \Oc(g) \ , $$
and in the case the implicit constant depend on parameters such as $\beta$, this dependence will be indicated thanks to subscripts (e.g $\ll_\beta$). Also $ \log_2 := \log \log , $ for shorter notations.

\section{OPUC and preliminary analysis}
\label{section:beta_and_opuc}

\subsection{Setting}
\label{subsection:setting}
Given a measure $\mu$ on the circle, the Gram-Schmidt orthogonalization procedure applied to the sequence $\left\{ 1, z, z^2, \dots \right\}$ gives rise to a sequence of monic orthogonal polynomials $\left( \Phi_k(z), k=0, 1, \dots \right)$. It is well-known in the literature of Orthogonal Polynomials on the Unit Circle (OPUC) that, if the measure is supported on $n$ points, $\Phi_n$ vanishes on exactly these points. The family of OPUCs follows the Szëgo recurrence relation:
$$ \Phi_{k+1}(z) = z \Phi_k(z) - \overline{\alpha}_k \Phi_k^*(z)$$
where $\left( \alpha_k; k=0, 1, \dots, n-1 \right)$ are the so-called Verblunsky coefficients and $\Phi_k^*(z) = z^k \overline{ \Phi_k\left( \bar{z}^{-1} \right) }$. The involution $*$ conjugates and reverses the order of coefficients. Moreover, $\Phi_k^*(z=0) = 1$ because $\Phi_k$ is monic. The Szëgo relationship can be written matrix-wise as:
\begin{align}
\label{def:matrix_szego}
\begin{pmatrix}
\Phi_{k+1}  (z) \\
\Phi_{k+1}^*(z)
\end{pmatrix}
= & 
\begin{pmatrix}
z            & - \overline{\alpha}_k \\
- \alpha_k z & 1
\end{pmatrix}
\begin{pmatrix}
\Phi_{k}  (z) \\
\Phi_{k}^*(z)
\end{pmatrix}
\end{align}

Another standard fact is that for $k \leq n-1$, the zeroes of $\Phi_k$ are inside the open unit disk $\D$, while the zeroes of $\Phi_k^*$ are outside the closed disk. This last property implies that we can define $\log \Phi_k^*$ as the unique version of the logarithm which vanishes at zero and which is continuous on the closed unit disc. 

It has been proven in \cite{bib:KN04} that for the C$\beta$E, the characteristic polynomial $\prod_{j=1}^n\left( z - \lambda_j \right)$ can be obtained as the last element $\Phi_n$ of a system of OPUCs, corresponding to a sequence of random independent Verblunsky coefficients $\alpha^{(n)}_0, \dots, \alpha^{(n)}_{n-1}$, their argument being uniform on $[0, 2\pi)$, $|\alpha^{(n)}_j|^2$ being a Beta random variable of parameters $\left(1, \beta_j := \frac{\beta (j+1)}{2} \right)$ for $0 \leq j \leq n-2$:
\begin{align}
\label{eq:beta_density} 
\P\left( |\alpha^{(n)}_j|^2 \in dx \right) = \beta_j \left( 1-x \right)^{ \beta_j - 1} \mathds{1}_{\left\{ 0 < x < 1 \right\} } dx,
\end{align}
 and $|\alpha^{(n)}_{n-1}|^2 = 1$. We use Proposition B.2 of the article \cite{bib:KN04} in order to reverse the traditional order of the $n-1$ first Verblunsky coefficients. Also, we will shift our interest from the orthogonal polynomials $(\Phi_k)_{0 \leq k \leq n}$ to those traditionally denoted $(\Phi^*_k)_{0 \leq k \leq n}$. Nevertheless, maximum modulii are identical. 

Now, let us couple all the dimensions together by considering the orthogonal polynomials $(\Phi_k, \Phi^*_k)_{k \geq 0}$
associated to an infinite family of independent Verblunsky coefficients $(\alpha_j)_{j \geq 0}$, such that $|\alpha_j|^2$ is a Beta random variable with parameters $(1, \beta_j )$ and the argument of $\alpha_j$ is uniform on $[0, 2\pi)$: from now, the notation $(\Phi_k, \Phi^*_k)$ will always refer to this setting.  The previous system can then be realized by taking $\alpha^{(n)}_j = \alpha_j$ for $0 \leq j \leq n-2$ and $\alpha^{(n)}_{n-1}$ independent, uniform on the unit circle.
With this realization, we have this useful lemma: 
\begin{lemma}
For all $\sigma \in \{1, i, -i\}$, the family of random variables
$$ \left( \sup_{z \in \mathbb{U}'} \Re (\sigma 
\log X_n(z)) - \sup_{z \in \mathbb{U}} \Re (\sigma 
\log \Phi^*_{n-1}(z))  \right)_{n \geq 1}$$
is tight. 
\end{lemma}
\begin{proof}
Using the last step of the Szeg\"o recursion, we get for $z \in \overline{\mathbb{D}}$, 
$$X_n(z) = \Phi_{n-1}^*(z)
\left(1 - \alpha_{n-1}^{(n)} z \frac{\Phi_{n-1}(z)}{\Phi_{n-1}^*(z)} \right),$$
and by continuity, for all $z 
\in \overline{\mathbb{D}} \backslash \{\overline{\lambda_1}, \dots, \overline{\lambda_n}\}$, 
$$\log X_n(z) = \log \Phi_{n-1}^* (z) 
+ \log \left(1 - \alpha_{n-1}^{(n)} z \frac{\Phi_{n-1}(z)}{\Phi_{n-1}^*(z)} \right),$$
if we take a continuous version of the last logarithm. We have 
$$\left| \alpha_{n-1}^{(n)} z \frac{\Phi_{n-1}(z)}{\Phi_{n-1}^*(z)} \right| \leq 1$$
for all $z \in \overline{\mathbb{D}}$, from the maximum principle, the fact that $\Phi_{n-1}^*$ does not vanish on $\overline{\mathbb{D}}$ and that the modulus of the expression is equal to $1$ for $z \in \mathbb{U}$. Hence, for all $z \in 
\mathbb{U}'$,  
$$\left| \Im  \log \left(1 - \alpha_{n-1}^{(n)} z \frac{\Phi_{n-1}(z)}{\Phi_{n-1}^*(z)} \right)\right| \leq \frac{\pi}{2},$$
which implies the lemma for $\sigma \in \{i, -i\}$, since in the second supremum, 
we can replace $\mathbb{U}$ by $\mathbb{U}'$ by continuity. 
We also get, again for $z \in \mathbb{U}'$: 
$$\Re \log \left(1 - \alpha_{n-1}^{(n)} z \frac{\Phi_{n-1}(z)}{\Phi_{n-1}^*(z)} \right) \leq \log 2,$$
which, for $\sigma = 1$, implies the 
tightness of the positive part of the quantity involved in the lemma. 
On the other hand, by continuity, the maximum of $\Re \log \Phi_{n-1}^*$ on $\mathbb{U}$ is attained: let $z_0$ be the corresponding point (say) with smallest argument in $[0, 2 \pi)$.
Since $\alpha_{n-1}^{(n)}$ is uniform on $\mathbb{U}$ and independent of $\Phi_{n-1}^*$, we have 
$$X_n(z_0) = \Phi_{n-1}^* (z_0) (1 - U),$$
where $U$ is uniform on $\mathbb{U}$, independent of $\Phi_{n-1}^* (z_0)$. 
Hence, $z_0 \in \mathbb{U}'$ almost surely, and in this case, we get 
\begin{align*}\sup_{z \in \mathbb{U}'} \Re (
\log X_n(z)) - \sup_{z \in \mathbb{U}} \Re ( 
\log \Phi^*_{n-1}(z)) 
& = \sup_{z \in \mathbb{U}'} \Re (
\log X_n(z)) - \Re \log \Phi^*_{n-1}(z_0)
\\ & \geq \Re \log X_n(z_0) -  \Re \log \Phi^*_{n-1}(z_0) = \Re \log (1-U).
\end{align*}
Hence, for $\sigma = 1$, the negative part of the quantity involved in the lemma is stochastically smaller than $(\Re \log (1-U))_-$, independently of $n$, and then it is tight. 
\end{proof}
We deduce the following: 
\begin{corollary} \label{corollary:phistar}
In order to prove Theorem \ref{thm:main}, it is sufficient to show the same result with $X_n$ replaced by $\Phi^*_n$ and $\mathbb{U}'$ replaced by $\mathbb{U}$. 
\end{corollary}
We will show this result in the sequel of the paper.  
The family of orthogonal polynomials $(\Phi^*_k)_{k \geq 0}$ associated to $(\alpha_j)_{j \geq 0}$ can be described as follows. See \cite{bib:BNR09} for a similar derivation and \cite{bib:KSt09} for a detailed study of the relative Prüfer phases, including the description of their diffusive limit.

\begin{lemma}
\label{lemma:expr}
There are continuous real functions $\Psi_k$ (Prüfer phases) such that for $\theta \in \mathbb{R}$, 
\begin{align}
\label{eq:Phi_product}
\log \Phi^*_k(e^{i\theta}) = & \sum_{j=0}^{k-1} \log \left( 1 - \alpha_{j} e^{i \Psi_{j}(\theta) } \right),
\end{align}
and
\begin{equation} \Psi_k(\theta) = (k+1) \theta - 2 \Im \log \Phi_{k}^*(e^{i \theta})
= (k+1) \theta - 2 \sum_{j=0}^{k-1} 
\Im \log   \left(1 - \alpha_{j}       e^{ i \Psi_{j}(\theta)} \right).
 \label{eq:Psi}
\end{equation}
In particular,  $(\Psi_k(\theta)-(k+1)\theta)_{k \geq 0}$ is a martingale starting at zero and $\E\left( \Psi_k(\theta) \right) = (k+1) \theta$.
\end{lemma}
\begin{proof}
Using the Sz\"ego recursion  \eqref{def:matrix_szego}, for $z = e^{i \theta}$, we get:
$$\Phi^*_k(z) = \prod_{j=0}^{k-1} \frac{\Phi_{j+1}^*(z)}{\Phi_{j}^*(z)}
              = \prod_{j=0}^{k-1} \left( 1 - \alpha_{j} z \frac{\Phi_{j}(z)}{\Phi_{j}^*(z)} \right).$$
               Now, because of the definition of the $*$ involution, 
$$ z \frac{\Phi_{j}(z)}{\Phi_{j}^*(z)} = z^{j+1} \frac{ \overline{\Phi_{j}^*(z)} }{\Phi_{j}^*(z)} = \exp \left(i[(j+1) \theta 
- 2 \Im \log \Phi_j^*(e^{i\theta})] \right),$$
 which shows 
 Equations \eqref{eq:Phi_product} and 
 \eqref{eq:Psi}, by taking the continuous versions of the logarithms. 
\end{proof}
In order to study the extremal values of $(\Phi^*_k)_{k \geq 0}$, the following martingale structure will be crucial. Let $\Fc_j$ be
the $\sigma$-algebra generated by $\alpha_0, \alpha_1, \dots \alpha_{j-1}$, the first $j$ Verblunsky coefficients. Thanks to Lemma \ref{lemma:expr}, the evaluation of the polynomial $\Phi_k^*$ at every point is a multiplicative martingale with respect to this filtration. 

In order to follow the Sz\"ego recursion, it is also useful to consider differences between Pr\"ufer phases at different angles. We get the following lemma: 
\begin{lemma}
Let us define the relative Pr\"ufer phases $(\psi_j)_{j \geq 0}$ and the deformed Verblunsky coefficients $(\gamma_j)_{j \geq 0}$ by 
$$\psi_j(\theta) := \Psi_j(\theta) - \Psi_j(0), \; \gamma_j := \alpha_j e^{i \Psi_j(0)}.$$
Then, the joint law of $(\gamma_j)_{j \geq 0}$ is the same as the law of $(\alpha_j)_{j \geq 0}$,   \begin{align}
\label{eq:Phi_relativeproduct}
\log \Phi^*_k(e^{i\theta}) = & \sum_{j=0}^{k-1} \log \left( 1 - \gamma_{j} e^{i \psi_{j}(\theta) } \right),
\end{align}
and
\begin{align} \psi_k(\theta) & = (k+1) \theta - 2 \left(\Im \log \Phi_{k}^*(e^{i \theta}) - \Im \log \Phi_{k}^*(1) \right) \nonumber
\\ & = (k+1) \theta - 2 \sum_{j=0}^{k-1} 
\left( \Im \log   \left(1 - \gamma_{j} e^{i \psi_j(\theta)} \right)  - \Im \log   \left(1 - \gamma_{j} \right)  \right).
 \label{eq:relativepsi}
\end{align}
In particular,  $(\psi_k(\theta)-(k+1)\theta)_{k \geq 0}$ is a martingale starting at zero and $\E\left( \psi_k(\theta) \right) = (k+1) \theta$. Moreover, for all $k \geq 0$, $\psi_k$ is a.s. increasing, and for all $\theta, \theta' \in \mathbb{R}$, $\psi_k(\theta') - \psi_k( \theta)$ has the same law as $\psi_k(\theta' - \theta)$, and 
$\psi_k(\theta + 2 \pi) = \psi_k(\theta) + 2 (k+1) \pi$. 
\end{lemma}
\begin{proof}
We have $\Psi_0(0) = 0$ and then $\gamma_0 = \alpha_0$. For $j \geq 1$, conditionally on $\mathcal{F}_j$, $e^{i \Psi_j(0)}$ is fixed, with modulus $1$, whereas $\alpha_j$ has the same law as before conditioning (it is independent of $\mathcal{F}_j$). By rotational invariance of this law, the conditional law of $\gamma_j$ given $\mathcal{F}_j$ is equal to the law of $\alpha_j$. This imples the equality in law between $(\alpha_j)_{j \geq 0}$ and $(\gamma_j)_{j \geq 0}$. 
The formula \eqref{eq:Phi_relativeproduct}
is a consequence of \eqref{eq:Phi_product} and the definition of $\gamma_j$, $\psi_j$. 
We then deduce \eqref{eq:relativepsi} from \eqref{eq:Psi}. Hence, the martingale property is clear, and the $2\pi$-periodicity of 
$\theta \mapsto \psi_k(\theta) - (k+1) \theta$ is easily proven from \eqref{eq:relativepsi}, by induction on $k$.  Now, for $\theta, \theta' \in \mathbb{R}$, 
$$ \psi_k(\theta') - \psi_k(\theta)
= (k+1) (\theta' - \theta) - 2 \left(\Im \log \Phi_{k}^*(e^{i \theta'}) - \Im \log \Phi_{k}^*(e^{i \theta}) \right),$$
and then it has the same law as 
$ \psi_k(\theta'- \theta) $, provided that we check that $z \mapsto \Phi_k^*(z e^{i \theta})$ has the same law as $\Phi_k^*$. 
Now, this invariance in law is due to the fact that one can go from one of the two polynomials to the other by multiplying  the Verblunsky coefficients by  deterministic complex numbers of modulus $1$, which does not change their joint distribution. Finally, from \eqref{eq:Phi_relativeproduct} and Szeg\"o 
recursion,  
$$e^{i \psi_k(\theta)} 
= \gamma_k^{-1} \left(1 - \frac{\Phi_{k+1}^* (e^{i \theta})}{\Phi_k^*(e^{i \theta})} \right) =  \alpha_k \gamma_k^{-1} 
\frac{e^{i \theta} \Phi_k(e^{i \theta})}{\Phi^*_k(e^{i \theta})}
= e^{- i \Psi_k(0)} \frac{e^{i (k+1)\theta} 
\overline{\Phi^*_k(e^{i \theta}})}{
\Phi^*_k(e^{i \theta})}.
$$
Now, the last quotient can be written as a finite Blaschke product in $e^{i \theta}$, and then its argument is strictly increasing in $\theta$. 
\end{proof}
\subsection{Subgaussianity estimates for Verblunsky coefficients}
\label{subsection:subgaussianity}
In the sequel, we will often need estimates on the size of the deformed Verblunsky coefficients $\left( \gamma_j \right)_{j \geq 0}$.

\begin{proposition}
\label{proposition:subgaussianity} 
The following inequalities hold for all $j \geq 0$. For all $(t, s) \in \R^2$:
\begin{align*}
     \E\left[ e^{s \Re \log (1- \gamma_j)  + t \Im \log (1- \gamma_j)} \right] & 
\leq \exp\left( \frac{s^2 + t^2}{2} \frac{1}{1+\beta(j+1)}\right) \\
     \E\left[ e^{t \Re \gamma_j  + s \Im \gamma_j} \right] & 
\leq \exp\left( \frac{s^2 + t^2}{2} \frac{1}{1+\beta(j+1)}\right) \ ,
\end{align*}
and for $\frac{8(t^2 + s^2)}{(1 + \beta(j+1))^2} < 1$:
\begin{align*}
     \E\left[ e^{t \Re \gamma_j^2  + s \Im \gamma_j^2} \right] & 
\leq \exp\left(  \frac{8(t^2 + s^2)}{(1 + \beta(j+1))^2} \right) \ .
\end{align*}
\end{proposition}
\begin{proof}
The usual moments of $1 - \gamma_j$ can be computed, as in Lemma 2.3 of \cite{bib:BHNY} and extended by analytic continuation. We get, for all $s, t \in \C$, $\Re(s) > -1$,  
$$\E\left[ e^{s \Re \log (1- \gamma_j)  + t \Im \log (1- \gamma_j)} \right]
= \frac{\Gamma(1 + (\beta (j+1) /2)) \Gamma(1 + s + (\beta (j+1) /2))}{\Gamma(1 + (s+it + \beta (j+1))/2)) \Gamma(1 + (s-it + \beta (j+1))/2)}.   $$
 From the functional equation of the Gamma function and its asymptotics at infinity, we get, for $a > 0$, 
$$
 \frac{\Gamma(a) \Gamma(a+s)}{\Gamma(a + (s + it)/2) \Gamma(a + (s - it)/2) }
 = \prod_{m=0}^{\infty} \frac{(a + m
 + (s + it)/2)(a + m
 + (s - it)/2)}{(a + m) ( a+m+s)}.$$
 If $a > 1/2$, $s >0$, $t \in \mathbb{R}$, we get 
 \begin{align*}
  \frac{\Gamma(a) \Gamma(a+s)}{\Gamma(a + (s + it)/2) \Gamma(a + (s - it)/2) } & = \prod_{m=0}^{\infty}
 \frac{(a+m+(s/2))^2 + (t/2)^2}{(a+m)(a+m+s)}
\\ &  =  \prod_{m=0}^{\infty}
 \left(1 + \frac{s^2 + t^2}{4 (a+m)(a+m+s)} \right)
\\ & \leq \exp \left( \frac{s^2 + t^2}{4} \sum_{m=0}^{\infty} \frac{1}{(a+m)(a+m+s)} \right)
\\ & \leq \exp \left( \frac{s^2+ t^2}{4} \sum_{m=0}^{\infty} \frac{1}{(a+m)^2} \right)
\\ & \leq \exp \left( \frac{s^2+ t^2}{4} \sum_{m=0}^{\infty} \frac{1}{(a+m+(1/2))
(a+m-(1/2))} \right) \\ &  = \exp \left(
\frac{s^2+ t^2}{4 a - 2} \right). 
 \end{align*}
Applying this inequality to $a = 1 + (\beta (j+1)/2)$, $0 \leq j \leq k-1$ yields the first inequality. For the second inequality, we make the following computation for $\left(u, v \right) \in \C$:
\begin{align}
\label{eq:sub2} 
  \E\left[ e^{u \gamma_j  + v \bar{\gamma}_j} \right]
= & \sum_{k,l=0}^\infty \frac{u^k}{k!} \frac{v^l}{l!} \E\left( \gamma_j^k  \bar{\gamma}_l^l \right)
= \sum_{k=0}^\infty \frac{\left(uv\right)^k}{k!^2} \E\left( \left|\gamma_j \right|^{2k} \right) \ .
\end{align}
Putting $u = \frac{t-is}{2}, v=\bar{u}$ and because of the Beta integral $\E\left( \left| \gamma_j \right|^{2k} \right) = \frac{k!}{\prod_{l=1}^k (l + \beta_j)} \leq k! \left( 1+ \beta_j \right)^{-k}$, we turn Eq. \eqref{eq:sub2} into:
$$
       \E\left[ e^{t \Re \gamma_j  + s \Im \gamma_j} \right]
  \leq \sum_{k=0}^\infty \frac{1}{k!} \left( \frac{t^2 + s^2}{4 (1+\beta_j)} \right)^k
  =    \exp\left( \frac{t^2 + s^2}{4 (1+\beta_j)} \right)\ ,
$$
yielding the second inequality. Finally, in order to prove the last inequality, we start with the same computation as Eq. \eqref{eq:sub2}:
\begin{align}
\label{eq:sub3}
       \E\left[ e^{t \Re \gamma_j^2  + s \Im \gamma_j^2} \right]
  = &  \sum_{k=0}^\infty \left( \frac{t^2 + s^2}{4} \right)^k \frac{\E\left( |\gamma_j|^{4k} \right) }{k!^2} \ .
\end{align}
This time, we have from the Beta integral $\E\left( |\gamma_j|^{4k} \right) = \frac{(2k)!}{\prod_{l=1}^{2k} (l + \beta_j)}$, and hence Eq. \eqref{eq:sub3} becomes:
$$
       \E\left[ e^{t \Re \gamma_j^2  + s \Im \gamma_j^2} \right]
  =    \sum_{k=0}^\infty \left( \frac{t^2 + s^2}{4} \right)^k \binom{2k}{k} \frac{1}{\prod_{l=1}^{2k} (l + \beta_j)}
  \leq \sum_{k=0}^\infty \left( t^2 + s^2 \right)^k \frac{1}{\prod_{l=1}^{2k} (l + \beta_j)} \ .
$$
Combining the crude bound $\frac{1}{\prod_{l=1}^{2k} (l + \beta_j)} \leq 4^k \left( 1 + \beta(j+1) \right)^{-2k}$ and the hypothesis on $t^2+s^2$, we conclude:
$$
       \E\left[ e^{t \Re \gamma_j^2  + s \Im \gamma_j^2} \right]
  \leq \frac{1}{ 1-\frac{4(t^2 + s^2)}{(1 + \beta(j+1))^2} }
  \leq    \exp\left( \frac{8(t^2 + s^2)}{(1 + \beta(j+1))^2} \right) \ .
$$

\end{proof}

\subsection{On Prüfer phases and related heuristics}
Eq \eqref{eq:relativepsi} gives the following recursion, established in  
 \cite{bib:KSt09}:
 $$\psi_0(\theta)=\theta,$$
 \begin{align}
\label{eq:prufer_rec}
	\psi_{j+1}(\theta)= \psi_j(\theta) + \theta - 2 \left( \Im  \log \left( 1- \gamma_j e^{i\psi_j(\theta) }\right)
	- \Im  \log \left( 1- \gamma_j  \right)
	\right)
\end{align}
where
\begin{align}
\label{eq:prufer_im_log}
 \Im  \log \left( 1- \gamma_j  \right)
 - \Im  \log \left( 1- \gamma_j e^{i\psi_j(\theta) } \right)
	 & =  \Im \sum_{k=1}^\infty \frac{\gamma_j^k}{k} (e^{ik \psi_j(\theta)}-1).
\end{align}
Also, define the deviation of $\psi_j$ from its mean as $A_j$. We have for any $\theta\in [0,2\pi]$, 
\begin{align}
	A_j(\theta) & := \psi_j(\theta) - (j+1)\theta.
\end{align}

There are two heuristics that come to mind. On the one hand, for fixed $\theta$, the random sequence $\left( \psi_j(\theta) - (j+1) \theta \right)_{j\geq 0}$ should only be slowly varying. We formalise the intuition in the form of Gaussian tail estimates for increments in Proposition \ref{proposition:prufer_deviation}. On the other hand, $\psi_j(\theta)-\psi_j(\theta')$ should be of order $j \left|\theta-\theta'\right|$ with high probability. This is morally the content of Proposition \ref{proposition:prufer_modulus}, which gives a glimpse to the Prüfer phases' modulus of continuity.

\begin{proposition}
\label{proposition:prufer_deviation}
For fixed $\theta$, $k \geq 0$ and all $t > 0$:
$$   \P\left( \sup_{0\leq l\leq k} \left|A_{j+l}(\theta)-A_{j}(\theta)\right| \geq t \Big| \mathcal{F}_j \right)
\leq 4 \exp\left( -\frac{t^2 \beta}{32 \log\left( 1 + \frac{\beta k}{1+\beta j}\right)} \right)$$
\end{proposition}
\begin{proof}
For shorter notations, we will write $\nu(\gamma, \psi) = 2\Im \log(1-\gamma e^{i \psi})$ and $\nu(\gamma) = \nu(\gamma, 0)$. Thanks to the triangle inequality, we have
$$ \left|A_{j+l}(\theta)-A_{j}(\theta)\right| 
   \leq \left| \sum_{m=j}^{j+l-1} \nu(\gamma_m, \psi_m(\theta)) \right|
      + \left| \sum_{m=j}^{j+l-1} \nu(\gamma_m) \right|. $$
       By a union bound and the fact that $(\nu(\gamma_m, \psi_m(\theta))  )_{m\geq j} \eqlaw (\nu(\gamma_m))_{m\geq j}$ 
       conditionally on $\mathcal{F}_j$:
$$ \P\left( \sup_{0\leq l\leq k} \left|A_{j+l}(\theta)-A_{j}(\theta)\right| \geq t  \Big| \mathcal{F}_j \right)
\leq 2 \, \P\left( \sup_{0\leq l\leq k} \left| \sum_{m=j}^{j+l-1} \nu(\gamma_m) \right| \geq \frac{t}{2} \right) . $$
By the symmetry $\nu(\gamma_m) \eqlaw -\nu(\gamma_m)$, we reduce further to:
$$ \P\left( \sup_{0\leq l\leq k} \left|A_{j+l}(\theta)-A_{j}(\theta)\right| \geq t   \Big| \mathcal{F}_j \right)
\leq 4 \, \P\left( \sup_{0\leq l\leq k} \sum_{m=j}^{j+l-1} \nu(\gamma_m) \geq \frac{t}{2} \right) . $$
Now, we implement the martingale version of the Chernoff bound. For $\lambda>0$, consider the sub-martingale $M_l = e^{\lambda \sum_{m=j}^{j+l-1} \nu(\gamma_m) } $. By applying Doob's maximal inequality, we get:
\begin{align*}
       \P\left( \sup_{0\leq l\leq k} \left|A_{j+l}(\theta)-A_{j}(\theta)\right| \geq t   \Big| \mathcal{F}_{j} \right) 
\leq \quad \quad  & 4 \, \P\left( \sup_{0\leq l\leq k} \sum_{m=j}^{j+l-1} \nu(\gamma_m) \geq \frac{t}{2} \right)\\
=    \quad \quad  & 4 \, \P\left( \sup_{0\leq l\leq k} M_l \geq e^{\frac{\lambda t}{2} } \right)\\
\stackrel{\textrm{Doob}}{\leq} \quad 
           & 4 \, e^{-\frac{\lambda t}{2}} \E\left( M_k \right)\\
\stackrel{Prop. \ \ref{proposition:subgaussianity}}{\leq}
           & 4 \, e^{-\frac{\lambda t}{2} + 2 \lambda^2 \sum_{m=j}^{j+k-1} \frac{1}{1+\beta(m+1)} } \\
\leq \quad \quad & \ 4 \, e^{-\frac{\lambda t}{2} + 2 \lambda^2 \frac{1}{\beta}\log\left( 1 + \frac{\beta k}{1+\beta j}\right) } \ ,
\end{align*}
which holds for all $\lambda>0$. Optimizing over this variable concludes the proof.
\end{proof}

\begin{proposition}
\label{proposition:prufer_modulus}
There exists $\rho = \rho(\beta)>1$ such that:
$$ \forall \theta \in \R, \forall j \geq 0, \forall t>0, \ 
   \P\left( \left| \frac{\psi_j(\theta)}{(1+j) \theta} \right| \geq t \right) \ll_\beta \frac{1}{t^\rho}
$$
\end{proposition}
\begin{proof}
Without loss of generality, we can restrict ourselves to $\theta \in[0, 2\pi)$. By Markov's inequality, it is enough to prove that there is a $\rho \in (1,2)$ such that:
$$ \left\| \psi_{j}(\theta) \right\|_{L^\rho} \ll_{\rho, \beta}(1+ j) \theta.$$

\paragraph{Step 1: Controlling the first phases.}

We have, for all $\psi \geq 0$,
$j \geq 0$, 
\begin{align*}
\left|\Im \log \left( \frac{1 - \gamma_j} { 1 - \gamma_j
e^{i \psi}} \right) \right|
&  = \left| \Im \left(\int_0^{\psi} 
\frac{d}{d \theta} \log( 1 -\gamma_j e^{i \theta})d \theta \right)  \right|
\leq \int_0^{\psi} \left|
\frac{d}{d \theta} \log( 1 - \gamma_j e^{i \theta}) \right| d \theta
\\ & \leq 
\left(\int_0^{\psi} \left|
\frac{d}{d \theta} \log( 1 - \gamma_j e^{i \theta}) \right|^{\rho} d \theta
\right)^{1/\rho} \psi^{1 - (1/\rho)}
\\ & \leq \left(\int_0^{\psi} \left(
\frac{|\gamma_j|}{|1 -
\gamma_j e^{i \theta}|}  \right)^{\rho} d \theta
\right)^{1/\rho} \psi^{1 - (1/\rho)}.
\end{align*}
Hence, by rotational invariance of 
the law of $\gamma_j$, 
\begin{align*}\E 
\left[ \left|\Im \log \left( \frac{1 - \gamma_j} { 1 - \gamma_j
e^{i \psi}} \right) \right|^{\rho} 
\,  \big| \, |\gamma_j|\right]
& \leq \psi^{\rho - 1}  \int_{0}^{2 \pi} \frac{d \tau}{2 \pi}
 \int_{0}^{\psi} \left( \frac{|\gamma_j|}{|1 -
|\gamma_j| e^{i (\tau +\theta)}|}
\right)^{\rho} 
d \theta 
\\ & = \psi^{\rho} \int_{-\pi}^{\pi} \frac{d \tau}{2 \pi}\left( \frac{|\gamma_j|}{|1 -
|\gamma_j| e^{i \tau}|} \right)^{\rho}.
\end{align*}
Now, if $r \in [0,1)$, $\tau \in [-\pi, \pi]$, 
$$\frac{r}{|1 - r e^{i\tau}|} 
\leq \frac{r}{1 - r} \leq \frac{1}{1 - r},$$
for $\tau \in [-\pi/2, \pi/2]$,
$$\frac{r}{|1 - r e^{i\tau}|}
\leq \frac{r}{|\Im (r e^{i \tau})|}
\leq \frac{1}{|\sin(\tau)|}
\leq \frac{\pi}{2 |\tau|},$$
and for $\tau \in [- \pi, \pi] \backslash
[- \pi/2, \pi/2]$, 
$$\frac{r}{|1 - r e^{i\tau}|}
\leq \frac{r}{1 - r \cos (\tau)} \leq r \leq 1 \leq \frac{\pi}{|\tau|}.$$
Hence, in any case, 
$$\frac{r}{|1 - r e^{i \tau}|}
\leq \frac{\pi}{\max(1-r, |\tau|)}
\leq \frac{2 \pi}{1 - r + |\tau|}.$$
We then get: 
\begin{align*} \E 
\left[ \left|\Im \log \left( \frac{1 - \gamma_j} { 1 -\gamma_j
e^{i \psi}} \right) \right|^{\rho} 
\,  \big| \, |\gamma_j|\right]
& \leq \psi^{\rho} 
\int_{0}^{\pi} \frac{d \tau}{\pi}
\left( \frac{2 \pi}{1 - |\gamma_j|
+ \tau} \right)^{\rho}
\\ & \leq \psi^{\rho} 2^{\rho} \pi^{\rho - 1} \int_{1 - |\gamma_j|}^{\infty}
\frac{d u}{u^\rho}
  = \Oc_{\rho} \left( \psi^{\rho}
  (1 - |\gamma_j|)^{1 - \rho} 
  \right).
\end{align*}
Now, 
\begin{align*}
\E[ (1 - |\gamma_j|)^{1 - \rho} ] & = \E[ (1 - |\gamma_j|^2)^{1 - \rho}
(1 + |\gamma_j|)^{\rho - 1}]
\\ & \leq 2^{\rho-1} 
 \frac{\beta (j+1)}{2}
 \int_{0}^{1} (1 - x)^{(\beta/2)(j+1) -\rho} dx.
\end{align*}
For $\rho \in (1,\min(2,1+(\beta/2))$), we have $(\beta/2)(j+1) - \rho \geq (\beta/2) - \rho > -1$.
Hence, the last expectation is finite, and depends only on $\rho, \beta, j$:
$$\E [( 1 - |\gamma_j|)^{1 - \rho}] = \Oc_{\rho, \beta,j}(1),$$
 which implies 
$$\E\left[ \left|\Im \log \left( \frac{1 - \gamma_j} { 1 -\gamma_j
e^{i \psi}} \right) \right|^{\rho} \right]
= \Oc_{\rho, \beta,j} (\psi^{\rho})$$
and therefore by an immediate recurrence using Equation \eqref{eq:prufer_rec}, we have for every $j \leq j_0$, $j_0$ fixed:
\begin{align}
\label{eq:prufer_modulus_1}
\|  \psi_{j_0}(\theta) \|_{L^\rho} & \ll_{j_0, \rho, \beta}  |\theta| \ .
\end{align}

\paragraph{Step 2: Conditional second moment estimate.}
Recall that $\Fc_j$ is the $\sigma$-algebra generated by $\alpha_0, \alpha_1, \dots, \alpha_{j-1}$, or equivalently, by $\gamma_0, \dots, \gamma_{j-1}$. In this paragraph, we prove that there exists a constant $C_\beta$ such that for $j \geq j_0$, $j_0$ depending only on $\beta$, conditional second moments do exist and:
\begin{align}
\label{eq:prufer_modulus_2}
\E\left( \psi_{j+1}(\theta)^2 | \Fc_j \right) \leq & \ \left(\theta + \psi_j(\theta) \right)^2 + \frac{C_\beta}{j+1} \psi_{j}(\theta)^2 \ .
\end{align}

Consider the recurrence \eqref{eq:prufer_rec}. From the series expansion \eqref{eq:prufer_im_log}, it is clear that $2 \Im \log\frac{1-\gamma e^{i\psi}}{1-\gamma}$ is centered. Thus the first term in the bound \eqref{eq:prufer_modulus_2} is nothing but $\E\left( \psi_{j+1}(\theta) | \Fc_j \right)^2$. As such, we only have to control the conditional variance:
\begin{align*}
    \Var\left( \psi_{j+1}(\theta)^2 | \Fc_j \right) 
= \quad & 4 \E\left( \left( \Im \log\frac{1-\gamma e^{i\psi_j(\theta)}}{1-\gamma} \right)^2 | \Fc_j \right)\\
\stackrel{Eq. \eqref{eq:prufer_im_log}}{=}
        & 4 \sum_{k=1}^\infty \frac{ \E\left( \left| \Im \gamma_j^k (e^{i k \psi_j(\theta)} - 1 ) \right|^{2} | \Fc_j \right) }{k^2}\\
= \quad & 8 \sum_{k=1}^\infty \E\left( \left| \gamma_j \right|^{2k} \right)
                                                  \frac{\sin\left( \frac{k \psi_j(\theta)}{2} \right)^2}{ k^2 } \\
\leq \quad &  2 \psi_j(\theta)^2 \E\left( \frac{|\gamma_j|^2}{1-|\gamma_j|^2} \right) \ .
\end{align*}
If $j_0$ is the smallest value of $j$ such that $j \geq 2$ and $\beta_j > 1$, we get for $j\geq j_0$, 
$$ \E\left( \frac{|\gamma_j|^2}{1-|\gamma_j|^2} \right) = \beta_j \int_0^1 x (1-x)^{\beta_j-2} = \frac{1}{\beta_j-1}
\ll_{\beta} \frac{1}{1+j}, \ .$$
which gives \eqref{eq:prufer_modulus_2}.
\paragraph{Step 3: Interpolating between first and second moment.}
For $1 < \rho < 2$, H\"older's inequality is written for any random variable $X$:
$$ \left\| X \right\|_{L^\rho} \leq \left\| X \right\|_{L^1}^{\frac{2}{\rho}-1} \left\| X \right\|_{L^2}^{2-\frac{2}{\rho}}, $$
hence in the conditional version, for $j \geq j_0$:
\begin{align*}
          \E\left( \psi_{j+1}(\theta)^\rho | \Fc_j \right)
= \quad & \E\left( \psi_{j+1}(\theta) | \Fc_j \right)^{2-\rho}
          \E\left( \psi_{j+1}(\theta)^2 | \Fc_j \right)^{\rho-1}\\
\stackrel{Eq. \eqref{eq:prufer_modulus_2}}{\leq} & 
    \left( \theta + \psi_{j}(\theta) \right)^{2-\rho}
    \left( \left( \theta + \psi_{j}(\theta) \right)^2 + \frac{C_\beta \psi_j(\theta)^2 }{j+1} \right)^{\rho-1}\\
\leq    & \left( \theta + \psi_{j}(\theta) \right)^{\rho}
          \left( 1 + \frac{C_\beta \psi_j(\theta)^2 }{(j+1) \left( \theta + \psi_{j}(\theta) \right)^2 } \right)^{\rho-1}\\
\leq    & \left( \theta + \psi_{j}(\theta) \right)^{\rho}
          e^{ \frac{C_\beta (\rho-1) }{j+1} } \ .
\end{align*}
Hence, upon taking expectation and using the triangle inequality, for $j \geq j_0$:
\begin{align}
\label{eq:prufer_modulus_3}
\left\| \psi_{j+1}(\theta) \right\|_{L^\rho}
\leq & \, e^{ \frac{C_\beta (\rho-1) }{\rho(j+1)} }
       \left( |\theta| + \left\| \psi_{j}(\theta) \right\|_{L^\rho} \right) \ .
\end{align}

\paragraph{Conclusion:} Applying Gronwall's lemma to Equation \eqref{eq:prufer_modulus_3}, we have, for $j \geq j_0$, 
\begin{align*}
   \left\| \psi_{j}(\theta) \right\|_{L^\rho} 
& \leq e^{\frac{C_\beta (\rho-1) }{\rho} \sum_{k=j_0+1}^j \frac{1}{k}}
       \left( \left\| \psi_{j_0}(\theta) \right\|_{L^\rho} 
            + |\theta| \sum_{l=j_0+1}^{j} e^{-\frac{C_\beta (\rho-1) }{\rho} \sum_{k=j_0+1}^{l-1} \frac{1}{k}}
       \right)\\
& \ll_{\beta, \rho}  (1+j) |\theta| \sum_{l=j_0+1}^j \frac{1}{1+j} e^{\frac{C_\beta (\rho-1) }{\rho} \sum_{k=l}^j \frac{1}{k}}
  \ll (1+j) |\theta| \sum_{l=j_0+1}^j \frac{1}{1+j} \left( \frac{l-1}{j} \right)^{-C_\beta \frac{(\rho-1) }{\rho}}\ .
\end{align*}
We are done upon noticing that for $\rho-1$ small enough, the Riemann sum $\sum_{l=j_0}^j \frac{1}{1+j} \left( \frac{l-1}{j} \right)^{-C_\beta \frac{(\rho-1) }{\rho}}$ is convergent.
\end{proof}
A useful corollary we will often invoke is:
\begin{corollary}
\label{corollary:djo}
With the constant $\rho = \rho(\beta)$ given in the previous proposition, and $D>0$, we have almost surely:
$$ \sup_{0 \leq j \leq k-1} \, \sup_{ \substack{\theta \in [0, 2\pi), \,  \theta'
 \in [0, 2\pi)_{k^{D}}, \\ \theta - \theta' \in [0, 2\pi k^{-D})} }
  \left| \psi_j(\theta) - \psi_j(\theta') \right| = 
  \Oc\left( k^{3-D\left( 1- \frac{1}{\rho} \right)} \right),$$
 where 
 $$[0, 2\pi)_{m} := \left\{ \frac{s}{2 \pi m}, 
 s \in \{0, 1, \dots, m-1\} \right\}.$$
\end{corollary}
\begin{proof}
Because $\psi_j$ is increasing, we have
\begin{align*}
  & \sup_{0 \leq j \leq k-1} \, \sup_{\theta \in [0, 2 \pi)} |\psi_{j}(\theta) - 
\psi_{j} ( 2 \pi k^{-D} \lfloor k^D \theta/ 2 \pi \rfloor ) |\\
= & \sup_{0 \leq j \leq k-1} \, \sup_{0 \leq s \leq k^D - 1}
\psi_j (2 (s+1)\pi k^{-D}) - \psi_j (2 s \pi k^{-D}).
\end{align*}
Now, using the fact that $\psi_j(\theta - \theta')$ and $\psi_j( \theta) - \psi_j(\theta')$ have the same law, we get: 
\begin{align*}
  & \sum_{ \substack{0 \leq j \leq k-1 \\ 0 \leq s \leq k^D - 1} }
    \P\left( \psi_j (2 (s+1)\pi k^{-D}) - \psi_j (2 s \pi k^{-D}) \geq k^{3-D\left( 1- \frac{1}{\rho} \right)} \right)\\
= & k^D \sum_{ 0 \leq j \leq k-1 }
        \P\left( \psi_j (2 \pi k^{-D}) \geq k^{3-D\left( 1- \frac{1}{\rho} \right)} \right)\\
\ll_{\beta} \, & k^D \sum_{ 0 \leq j \leq k-1 } \left( \frac{1+j}{k^{3+\frac{D}{\rho}}} \right)^{\rho}
\ll k^{1-2 \rho} \ .
\end{align*}
As $\rho>1$, these probabilities are summable and almost surely, the events occur finitely many times by the Borel-Cantelli Lemma.
\end{proof}

\section{An auxiliary log-correlated field}
\label{section:auxiliary_field}

Recall that $\beta_j = \frac{\beta}{2}(j+1)$ for all $j\geq 0$. Consider $(E_j)_{j\geq 0}, (\Theta_j)_{j\geq 0}, (\Gamma_j)_{j\geq 0}$ independent with $E_j$ being exponentially distributed, $\Theta_j$ being uniform on $[0,2\pi)$ and $\Gamma_j$ being Gamma distributed with parameter $\beta_j$. Following the setting presented in Subsection \ref{subsection:setting}, the deformed Verblunsky coefficients can be taken as follows:
\begin{align}
\label{eq:verblunsky_gamma}
 \forall  j\geq 0, \ \gamma_j = & \sqrt{\frac{E_j}{E_j+\Gamma_j}} e^{i\Theta_j} \,
\end{align}
by standard equalities in law involving Beta and Gamma distributions.  

Now, notice that 
\begin{align}
\label{eq:complex_normal}
\Nc_j^\C = & - \sqrt{E_j} e^{i \Theta_j} 
\end{align}
is a complex Gaussian of total variance $1$ and that for large $j$, $\Gamma_j$ is close to $\beta (1+j)/2$ with high probability. As such, we feel like approximating $\gamma_j \approx  - \sqrt{\frac{2}{\beta}} \frac{\Nc_j^\C}{\sqrt{j+1}}$. To that endeavor, we define for $\theta \in \mathbb{R}$ the auxiliary field:
\begin{align}
\label{eq:def_Z}
	Z_k(\theta) := & \sum_{j=0}^{k-1} \frac{ \Nc_j^\C e^{i \psi_j(\theta)}}{\sqrt{j+1}} \ .
\end{align}
This is a much more convenient process to study, as $\left( Z_k(\theta) \right)_{\theta \in [0,2\pi)}$ is a field
with Gaussian one-point marginals. It will be particularly convenient in order to establish 
the upper bounds in Section \ref{section:upper_bound}. Moreover, despite the fact that the $Z$ is not globally 
a Gaussian field because of the Pr\"ufer phases, its law is invariant when we shift $\theta$ by a constant. 

Indeed, for $\theta_0 \in \mathbb{R}$, if we change $\Theta_j$ by $\Theta_j + \psi_j( \theta_0)$ modulo $2 \pi$ for all 
$j \geq 0$,
this does not change the distribution of $(E_j)_{j\geq 0}, (\Theta_j)_{j\geq 0}, (\Gamma_j)_{j\geq 0}$
since $\psi_j(\theta_0)$ is measurable with respect to the variables $E_k, \Theta_k, \Gamma_k$ for $k \leq j-1$. 
Hence, the law of the field $Z$ is not changed. On the other hand, 
$ \Nc_j^\C$ and $\gamma_j$ are multiplied by $e^{i \psi_j(\theta_0)}$, and then one 
checks that the new values of $\gamma_j$ correspond to the modified Verblunsky coefficients associated to the 
polynomials $z \mapsto \Phi_k^*(z e^{i \theta})$, whose relative Pr\"ufer phases are given by 
$\theta \mapsto \psi_j(\theta_0 + \theta) - \psi_j(\theta_0)$. We then deduce that the new field $Z$, which 
has the same law as the initial one, is 
given by $(Z_k(\theta + \theta_0))_{k \geq 0, \theta  \in \mathbb{R}}$. 

The main result of this section is:
\begin{proposition}
\label{proposition:comparison_to_Z}
We have almost surely:
$$\sup_{k \geq 0} \sup_{\theta \in [0,2\pi)}
  \left| \log \Phi_k^*(e^{i\theta}) - \sqrt{\frac{2}{\beta}} Z_k(\theta) \right|
  < \infty
$$
In particular, regardless of the coupling obtained by having picked a consistent family of Verblunsky coefficients, the family of random variables
\begin{align*}
	\left( \sup_{\theta \in [0,2\pi)} \left| \log \Phi_k^*(e^{i\theta}) - \sqrt{\frac{2}{\beta}} Z_k(\theta) \right| \right)_{k \geq 0}
\end{align*}
is tight. As such, Theorem \ref{thm:main} is equivalent to proving that for $\sigma \in \{1, i , -i\}$:
$$ \sup_{\theta \in [0; 2\pi)} \Re\left[ \sigma Z_n(\theta) \right]
 = \log n - \frac{3}{4} \log_2 n + \Oc(1) \ ,$$
 $\Oc(1)$ denoting a tight family of random variables. 
\end{proposition}
As a corollary of Proposition \ref{proposition:comparison_to_Z}, we have:
\begin{corollary}
\label{corollary:coupling_bm}
Fix $z \in \U$. One can couple the sequence $\left(\Phi^*_k(z) \right)_{k \geq 0}$ and a complex Brownian motion $\left( W_t^\C = \sqrt{\half}W_t^1 + i \sqrt{\half} W_t^2 \right)_{t \geq 0}$ in such a way that almost surely: 
$$\sup_{k \geq 0} |\log \Phi^*_k(z) -  \sqrt{\frac{2}{\beta}} W_{\tau(k)}^\C | < \infty, $$
where $\tau$ is the time change:
\begin{align}
\label{eq:def_tau}
\tau(p) := & \Var\left( Z_k(\theta=0) \right) = \sum_{k=1}^p \frac{1}{k} \ .
\end{align}
The complex Brownian motion is normalized so that $\E\left( | W_t^\C |^2 \right) = t$, and $\left( W^1, W^2 \right)$ is a pair of independent standard real Brownian motions.
\end{corollary}

\begin{proof}[Strategy of proof for Proposition \ref{proposition:comparison_to_Z}]
	We have the following:
		\begin{align*}
		\log \Phi_k^*(e^{i\theta}) &= \sum_{j=0}^{k-1} \log \left(1 - \gamma_j e^{i\psi_j(\theta)} \right)
		\\
		&=- \sum_{j=0}^{k-1} \gamma_j e^{i\psi_j(\theta)} - \sum_{j=0}^{k-1}  \frac{\gamma_j^2}{2}  e^{2 i\psi_j(\theta)}  + \Oc \left( \sum_{j=1}^{k-1} \frac{|\gamma_j|^3}{1-|\gamma_j|}\right)
	\end{align*}
	from the Taylor expansion of the logarithm. Now, for all $j \geq 0$, by using the expression of the Beta distribution corresponding to $|\gamma_j|^2$,  we get: 
	\begin{align*}\E \left[ \frac{|\gamma_j|^3}{1 - |\gamma_j|} \right] 
		& = \frac{\beta (j+1)}{2} \int_{0}^1 (1 - x)^{\frac{ \beta (j+1)}{2} - 1} \frac{x^{\frac{3}{2}}}{1 - \sqrt{x}} dx
		\\ & = \frac{\beta (j+1)}{2}
		\int_{0}^1 (1 - x)^{\frac{ \beta (j+1)}{2} - 2} x^{\frac{3}{2}}(1+\sqrt{x}) dx
		\\ &  \leq \beta(j+1) 
		\int_{0}^1 (1 - x)^{\frac{ \beta (j+1)}{2} - 2} x^{\frac{3}{2}} dx \ .
	\end{align*}
	If $j$ is larger than $j_0$ (depending only on $\beta$), $\frac{ \beta (j+1)}{2} - 1 > 0$, and then, from the Beta integral: 
	$$\E \left[ \frac{|\gamma_j|^3}{1 - |\gamma_j|} \right] 
	\leq \beta(j+1)  \frac{\Gamma(\frac{ \beta (j+1)}{2} - 1)  \Gamma(\frac{5}{2})}{\Gamma(\frac{ \beta (j+1)}{2} + \frac{3}{2})}
	\underset{j \rightarrow \infty}{\sim}
	(\beta j ) \Gamma(\frac{5}{2}) \left( \frac{\beta j}{2} \right)^{-\frac{5}{2}}
	= \Oc((\beta j)^{-\frac{3}{2}}), 
	$$
	which implies $\E \left[ \sum_{j=j_0}^{\infty}\frac{|\gamma_j|^3}{1 - |\gamma_j|} \right] < \infty$. Hence, almost surely, 
	$$\sum_{j=0}^{\infty}
	\frac{|\gamma_j|^3}{1 - |\gamma_j|} < \infty,$$ 
	and then
	$$\sup_{k \geq 0} \sup_{\theta \in 
	[0, 2\pi)} \left|
	\log \Phi_k^*(e^{i \theta})
	+ \sum_{j=0}^{k-1} \gamma_j e^{i\psi_j(\theta)}
	+ \sum_{j=0}^{k-1} \frac{\gamma_j^2}{2}  e^{2i \psi_j(\theta)}
	\right| < \infty \ .$$
	
	It is now sufficient to prove that almost surely
	\begin{align}
	\label{eq:tightness1}
		\sup_{k \geq 0} \sup_{\theta \in [0,2\pi)}
		\left| \sum_{j=0}^{k-1}  \gamma_j^2  e^{2 i \psi_j(\theta)} \right| < \infty
	\end{align}
	and 
	\begin{align}
	\label{eq:tightness2}
	        \sup_{k \geq 0} \sup_{\theta \in [0,2\pi)}
	        \left| \sum_{j=0}^{k-1} \gamma_j e^{i\psi_j(\theta)}  + \sqrt{\frac{2}{\beta}} Z_k(\theta) \right| < \infty \ .
	\end{align}
	We shall use the same strategy for both processes in the next two subsections.
\end{proof}

\subsection{Tightness of the random functions in Eq. \texorpdfstring{\eqref{eq:tightness1}}{Eq1} }
\label{subsection:tightness1}

Let us define
\begin{align}
\label{eq:def_S}
S_k(\theta) := & \sum_{j=0}^{k-1} (j+1) \gamma_j^2  e^{2 i \psi_j(\theta)} \ .
\end{align}
By Abel summation, for $k \geq 1$, 
$$  \sum_{j=0}^{k-1} \gamma_j^2  e^{2 i \psi_j(\theta)}
  = \sum_{j=0}^{k-1} \frac{S_{j+1}(\theta) -  S_j(\theta)}{j+1}
  = \frac{S_k(\theta)}{k} +  \sum_{j=1}^{k-1} \frac{S_j(\theta)}{j(j+1)}.
$$
We deduce that 
$$
\sup_{k \geq 0} \sup_{\theta \in [0, 2\pi)} \left| \sum_{j=0}^{k-1} \gamma_j^2  e^{2 i \psi_j(\theta)} \right| 
\leq \sup_{k \geq 1}
\sup_{\theta \in [0, 2\pi)} \frac{|S_k(\theta)|}{k}
+ \sum_{j=1}^{\infty} \frac{1}{j(j+1)}
\sup_{\theta \in [0, 2\pi)} |S_j(\theta)|.
$$
It is then sufficient to show that for some $\alpha \in (0,1)$, almost surely:
$$\sup_{\theta \in [0, 2\pi)} |S_k(\theta)|
= \Oc ( k^{\alpha}) \ .$$

\paragraph{One point estimate:} In this paragraph, we prove that for all $\theta \in [0, 2\pi)$, $C>0$, $\alpha \in (1/2,1)$,
\begin{align}
\label{eq:one_point_S}
 \P\left(  |S_k(\theta)| \geq k^{\alpha} \right)
\ll_{\alpha, \beta, C} k^{-C} \ , 
\end{align}
using a Chernoff bound. For a given $\theta \in [0, 2\pi)$, $S_k(\theta)$ has the same law as $S_k(\theta = 0)$ by rotational invariance. Moreover, since the distribution of $S_k(\theta)$ is invariant under multiplication by a complex number of modulus one, we only have to prove
$$\P\left( \Re S_k(\theta=0) \geq k^{\alpha}/2 \right) \ll_{\alpha, \beta, C} k^{-C}.$$

Now, from the third inequality in Proposition \ref{proposition:subgaussianity}, we deduce that for $\lambda$ in a neighborhood of zero depending only of $\beta$:
$$
     \E\left[ e^{(j+1) \lambda \Re(\gamma_j^2)} \right]
\leq \exp\left( \frac{8\lambda^2}{\beta^2} \right) \ .
$$
Hence, using the independence of $(\gamma_j)_{0 \leq j \leq k-1}$, we have $\E [e^{\lambda \Re S_k(\theta = 0)}] \leq  \exp\left(  \frac{8 k \lambda^2}{\beta^2} \right)$. A classical Chernoff bound yields for all $\alpha \in (\half, 1)$,
$$   \P\left(  \Re S_k(\theta=0) \geq k^{\alpha}/2 \right)
\leq \exp\left( - \frac{\lambda k^{\alpha}}{2}
                + \frac{8 \lambda^2 k}{\beta^2} \right),$$
and for $k$ large enough depending on $\alpha$ and $\beta$, we can take $\lambda = \frac{\beta^2 k^{\alpha-1}}{32}$, which gives
$$
\P\left[ \Re S_k(\theta = 0) \geq k^{\alpha}/2 \right]
\leq \exp \left(  - \frac{\beta^2 k^{2 \alpha-1}}{128} \right) \ll_{\alpha, \beta, C} k^{-C},$$
for all $C > 0$. This domination, proven for $k$ large enough, can of course be extended to small values of $k$, since there are finitely many of them. 

\paragraph{Multiple points estimate:}
Using the estimate \eqref{eq:one_point_S}, the union bound and the Borel-Cantelli lemma, we deduce that almost surely, for all integers $D > 0$, 
$$\sup_{\theta \in  [0, 2\pi)_{k^{-D}} }
|S_k(\theta)| < k^{\alpha}$$
occurs for all but finitely many values of $k$, and then 
$$\sup_{\theta \in [0, 2\pi)_{k^{-D}}}  |S_k(\theta)| = \Oc(k^{\alpha}) \ \ a.s \ .$$

Now, it remains to fill the gaps between the points in $[0, 2\pi)_{k^{-D}}$. From the formula \eqref{eq:def_S} giving $S_k(\theta)$, it is enough to have almost surely:
$$\ \sup_{0 \leq j \leq k-1} \, \sup_{ \substack{\theta \in [0, 2\pi), \,  \theta'
 \in [0, 2\pi)_{k^{-D}}, \\ \theta - \theta' \in [0, 2\pi k^{-D})} }
  \left| \gamma_j^2 \left( e^{2 i  \psi_j(\theta)} - e^{2 i  \psi_j(\theta')} \right) \right| = 
  \Oc(k^{-2+\alpha}) ,$$
which is obtained from Corollary \ref{corollary:djo} for $D>0$ large enough.

\subsection{Tightness of the random functions in Eq. \texorpdfstring{\eqref{eq:tightness2}}{Eq2}}
\label{subsection:tightness2}

By putting Equations \eqref{eq:verblunsky_gamma} and \eqref{eq:complex_normal}, we can write 
$$\gamma_j e^{i \psi_j(\theta)} + \sqrt{\frac{2}{\beta}} \frac{\Nc_j^\C}{\sqrt{j+1}} e^{i \psi_j(\theta)}
= \gamma_j e^{i \psi_j(\theta)} \left( 1 - \sqrt{\frac{E_j+\Gamma_j}{\beta_j} } \right) \ .$$
In a similar fashion as in the previous subsection, we define:
\begin{align}
\label{eq:def_T} 
T_k(\theta) := \sum_{j=0}^{k-1} (j+1) 
\gamma_j e^{i \psi_j(\theta)} \left( 1 - \sqrt{\frac{E_j+\Gamma_j}{\beta_j} } \right) .
\end{align}
By Abel summation, for $k \geq 1$, 
$$\sum_{j=0}^{k-1} \gamma_j e^{i \psi_j(\theta)} \left( 1 - \sqrt{\frac{E_j+\Gamma_j}{\beta_j} } \right)
= \sum_{j=0}^{k-1} \frac{T_{j+1}(\theta) -  T_j(\theta)}{j+1} = \frac{T_k(\theta)}{k} +  \sum_{j=1}^{k-1}
\frac{T_j(\theta)}{j(j+1)}.
$$
We deduce that 
$$\sup_{k \geq 0} 
\sup_{\theta \in [0,2\pi)}
\left| \sum_{j=0}^{k-1} \gamma_j e^{i \psi_j(\theta)} \left( 1 - \sqrt{\frac{E_j+\Gamma_j}{\beta_j} } \right)  \right| 
\leq \sup_{k \geq 1} \sup_{\theta \in [0,2\pi)} \frac{|T_k(\theta)|}{k}
+ \sum_{j=1}^{\infty} \frac{1}{j(j+1)}
\sup_{\theta \in [0,2\pi)} |T_j(\theta)|.
$$
As in the previous subsection, it is then sufficient to show that for some $\alpha \in (0,1)$,
$$\sup_{\theta\in [0,2\pi)} |T_k(\theta)|
= \Oc ( k^{\alpha})$$
almost surely and the proof goes along the same lines.

\paragraph{One point estimate:} Just like before, we prove that for all $\theta \in [0, 2\pi)$, $C>0$, $\alpha \in (\frac{3}{4},1)$:
\begin{align}
\label{eq:one_point_A}
 \P\left(  |T_k(\theta)| \geq k^{\alpha} \right)
\ll_{\alpha, \beta, C} k^{-C} \ , 
\end{align}
via a Chernoff bound. Again, by rotational invariance we can assume $\theta=0$ and because the distribution of $T_k(\theta=0)$ is invariant under multiplication by a complex number of modulus one, we only have to prove:
$$\P\left( \Re T_k(\theta=0) \geq k^{\alpha}/2 \right) \ll_{\alpha, \beta, C} k^{-C}.$$
Now, recall that $\gamma_j$ and $E_j+\Gamma_j$ are independent by a classical identity in law due to Lukacs \cite{bib:Luk55}. This fact characterises the Gamma distribution and such identities are refered as ``Beta-Gamma algebra identities''. In any case, by conditioning on $E_j+\Gamma_j$ and applying the second inequality of Proposition \ref{proposition:subgaussianity}, we have, for $\lambda > 0$:
\begin{align*}
       \ \E\left[ e^{(j+1) \lambda \Re \gamma_j \left( 1 - \sqrt{\frac{E_j+\Gamma_j}{\beta_j} } \right) } \right]
\leq & \ \E\left[ e^{\frac{\lambda^2(j+1)}{2\beta} \left( 1 - \sqrt{\frac{E_j+\Gamma_j}{\beta_j} } \right)^2} \right]
= I_1 + I_2, 
\end{align*}
where 
$$I_1 := \E\left[ e^{\frac{\lambda^2(j+1)}{2\beta} \left( 1 - \sqrt{\frac{E_j+\Gamma_j}{\beta_j} } \right)^2} 
\mathds{1}_{|1 - (E_j + \Gamma_j)/\beta_j |
\leq (1+j)^{-1/4}} \right]$$
and 
$$I_2 :=  \E\left[ e^{\frac{\lambda^2(j+1)}{2\beta} \left( 1 - \sqrt{\frac{E_j+\Gamma_j}{\beta_j} } \right)^2} 
\mathds{1}_{|1 - (E_j + \Gamma_j)/\beta_j |
> (1+j)^{-1/4}} \right].$$
For $|1-x| \leq (1+j)^{-1/4}$, we have
$$(1 - \sqrt{x})^2 = (1-x)^2 (1+\sqrt{x})^{-2} \leq (1-x)^2 \leq (1+j)^{-1/2}$$
and then 
$$ I_1 \leq e^{\frac{\lambda^2(j+1)^{1/2}}{2\beta} }.$$
On the other hand, by Cauchy-Schwarz inequality, 
\begin{align*} I_2 & \leq \E\left[ e^{\frac{\lambda^2(j+1)}{\beta} \left( 1 - \sqrt{\frac{E_j+\Gamma_j}{\beta_j} } \right)^2 }\right]^{1/2} \P \left( |1 - (E_j + \Gamma_j)/\beta_j |
> (1+j)^{-1/4} \right)^{1/2}
\\ & \leq \E\left[ e^{\frac{\lambda^2(j+1)}{\beta} \left( 1 +  \frac{E_j+\Gamma_j}{\beta_j}  \right) } \right]^{1/2}
\P \left( |1 - (E_j + \Gamma_j)/\beta_j |
> (1+j)^{-1/4} \right)^{1/2}
\\ &  \leq 
e^{\frac{\lambda^2 (j+1)}{2 \beta}}
\E \left[ e^{\frac{2 \lambda^2}{\beta^2} (E_j + \Gamma_j)} \right]^{1/2}
\P \left( |1 - (E_j + \Gamma_j)/\beta_j |
> (1+j)^{-1/4} \right)^{1/2}
\\ & = e^{\frac{\lambda^2 (j+1)}{2 \beta}}
\left( 1 - \frac{2 \lambda^2}{\beta^2} 
\right)^{-(1 + \beta_j)/2}
\P \left( |1 - (E_j + \Gamma_j)/\beta_j |
> (1+j)^{-1/4} \right)^{1/2}
\\ & \leq e^{\frac{2 \lambda^2}{\beta^2} \left(
1 + \beta(j+1)\right)}\P \left( |1 - (E_j + \Gamma_j)/\beta_j |
> (1+j)^{-1/4} \right)^{1/2},
 \end{align*}
 for $\lambda > 0$ small enough depending on $\beta$. 
 Now, for $\mu \in (0,1/2]$, 
 $\log ( 1 \pm \mu) \geq \pm \mu  - \mu^2$, and then 
 \begin{align*}\P \left( |1 - (E_j + \Gamma_j)/\beta_j |
> (1+j)^{-1/4} \right)
& \leq e^{- \mu \beta_j (1+j)^{-1/4}} 
 \E 
\left[e^{\mu(\beta_j - (E_j + \Gamma_j))} 
+ e^{-\mu(\beta_j - (E_j + \Gamma_j))} \right]
\\ & = e^{- \mu \beta_j (1+j)^{-1/4}}
\left( e^{\mu \beta_j} ( 1 + \mu)^{-1 - \beta_j} + e^{-\mu \beta_j} (1 - \mu)^{-1- \beta_j} \right)
\\ & \leq e^{- \mu \beta_j (1+j)^{-1/4}}
\left(e^{-\mu} + e^{\mu} \right) e^{\mu^2 } e^{\mu^2 \beta_j} \\ & \leq 4 \, e^{\beta_j(\mu^2 - (1+j)^{-1/4} \mu)}
\end{align*}
Taking $\mu = (1/2)(1+j)^{-1/4} \in (0,1/2]$, we deduce 
$$  \P \left( |1 - (E_j + \Gamma_j)/\beta_j |
> (1+j)^{-1/4} \right) \leq 4 \, e^{-(\beta_j/4) (1+j)^{-1/2} }$$
and then 
$$I_2 \leq  2 \,  e^{\frac{2 \lambda^2}{\beta^2} \left(
1 + \beta(j+1)\right) - \frac{\beta(1+j)^{1/2}}{16} }.$$
If $\lambda \leq (1+j)^{-1/3}$, the first term in the exponential is negligible with respect to the second one when $j$ goes to infinity, and then: 
$$ I_2 \ll_{\beta} e^{- \frac{\beta(1+j)^{1/2}}{17} } \ll_{\beta} (1+j)^{-2}.$$
We then get, for $\lambda \leq k^{-1/3}$ and $0 \leq j \leq k-1$, 
$$ \E\left[ e^{(j+1) \lambda \Re \gamma_j \left( 1 - \sqrt{\frac{E_j+\Gamma_j}{\beta_j} } \right) } \right]
\leq e^{\frac{\lambda^2(j+1)^{1/2}}{2 \beta}} + \Oc_{\beta} ((1+j)^{-2})
\leq e^{\frac{\lambda^2 k^{1/2}}{2 \beta}
+ \Oc_{\beta} ((1+j)^{-2})}
$$
Multiplying these inequalities gives: 
$$\E \left[e^{\lambda \Re T_k(0)} \right] \ll_{\beta} e^{\frac{\lambda^2 k^{3/2}} {2 \beta}},$$
$$\P \left[ \Re T_k(0) \geq k^{\alpha}/2 \right] \ll_{\beta} 
\exp  \left( - \frac{\lambda k^{\alpha}}{2}
 + \frac{\lambda^2 k^{3/2}}{2 \beta} \right).$$
 Taking $\lambda = \beta k^{\alpha - (3/2)}/2$, we get, for $k$ large enough (depending on $\beta$) in order to insure that $\lambda$ is sufficiently small and at most $k^{-1/3}$ (recall that $\alpha < 1$), we get 
 $$ \P \left[ \Re T_k(0) \geq k^{\alpha}/2 \right] \ll_{\beta} 
 \exp \left( - \frac{\beta k^{2 \alpha - (3/2)}}{8} \right) \ll_{\alpha,\beta,C} 
 k^{-C}$$
since $\alpha > \frac{3}{4}$. 

\paragraph{Multiple points estimate:} Using the union bound, and Borel-Cantelli lemma, we deduce that almost surely, for all integers $D>0$,

\begin{align*}
	\sup_{\theta \in [0,2\pi)_{k^D}} |T_k(\theta)|<k^\alpha
\end{align*}
occurs for all but finitely many values of $k$, and then
\begin{align*}
	\sup_{\theta \in [0,2\pi)_{k^D}} |T_k(\theta)|= \Oc(k^\alpha) \ .
\end{align*}
Now, it remains to fill the gaps. From the formula \eqref{eq:def_T} giving $T_k(\theta)$, it is enough to have almost surely:
\begin{align*}
\sup_{0\leq j\leq k-1} \max_{ \substack{\theta_1 \in [0,2\pi) , \theta_2 \in [0,2\pi)_{k^D} \\
                                        0 \leq  \theta_1-\theta_2 \leq 2\pi k^{-D}} }
\left| \gamma_j \left( 1 - \sqrt{\frac{E_j+\Gamma_j}{\beta_j} } \right) (e^{i\psi_j(\theta_1)} -e^{i\psi_j(\theta_2)}) 
\right|
= \Oc\left( k^{-2 + \alpha} \right) \ .
\end{align*}
Now, we have that $\sup_{0  \leq j \leq k-1}\left| \gamma_j \left( 1 - \sqrt{\frac{E_j+\Gamma_j}{\beta_j} } \right) \right| = \Oc( k )$ almost surely by Borel-Cantelli lemma:
\begin{align*}
       \P\left(\sup_{0 \leq j \leq k-1}\left| \gamma_j \left( 1 - \sqrt{\frac{E_j+\Gamma_j}{\beta_j} } \right) \right| \geq k \right)
\leq & \sum_{0 \leq j\leq k-1} \P\left(\left| \gamma_j \left( 1 - \sqrt{\frac{E_j+\Gamma_j}{\beta_j} } \right) \right| \geq k \right)\\
\leq & \frac{1}{k^2}\sum_{0 \leq j \leq k-1} \E\left( |\gamma_j|^2 \left( 1 - \sqrt{\frac{E_j+\Gamma_j}{\beta_j} } \right)^2 \right)\\
\ll_{\beta}  & \frac{\log (1+k)}{k^2} \ ,
\end{align*}
which is summable. Therefore, it is enough to prove that for some $D>0$, almost surely, 
$$\sup_{0\leq j\leq k-1} \sup_{ \substack{\theta_1 \in [0,2\pi) , \theta_2 \in [0,2\pi)_{k^D} \\
                                        0 \leq  \theta_1-\theta_2 \leq 2\pi k^{-D}} }
\left| \psi_j(\theta_1) - \psi_j(\theta_2) \right| = \Oc(k^{-3+\alpha}) \ .$$
Again, we can invoke Corollary \ref{corollary:djo}.

\section{The upper bound}
\label{section:upper_bound}

Thanks to Proposition \ref{proposition:comparison_to_Z}, the ``upper bound part'' of Theorem \ref{thm:main} (i.e.  the tightness of the positive part of the variables which are involved)  is a consequence of the following result:  
\begin{proposition}
\label{proposition:upper_bound}
For all $\sigma \in \{1, i, -i\}$, we have 
$$\underset{C \rightarrow \infty}{\lim} \underset{n \rightarrow \infty}{\lim \sup}  \, 
\P   \left[ \sup_{\theta \in [0, 2\pi)} \Re (\sigma Z_n(\theta) )
\geq \log 
n - \frac{3}{4} 
\log_2 n  + C  \right] = 0.
$$
\end{proposition}
The proof is only given after two subsections of preparatory work.

\subsection{Reduction to geometric progressions}
The following Lemma, in combination with Proposition \ref{proposition:comparison_to_Z}, shows that it suffices to handle the case where $n=b^J$ for a certain integer $b\geq 2$ (which can be fixed arbitrarily). 

\begin{lemma}
\label{lemma:geometric_progressions} 
If $b^J \leq n < b^{J+1}$, then:
$$
     \sup_{\theta \in [0, 2\pi)} \Re (\sigma Z_{b^{J  }}(\theta) )+ \Oc(1)
\leq \sup_{\theta \in [0, 2\pi)} \Re (\sigma Z_n(\theta) )
\leq \sup_{\theta \in [0, 2\pi)} \Re (\sigma Z_{b^{J+1}}(\theta) )+ \Oc(1) \ ,
$$
where $\Oc(1)$ corresponds to a tight family of random variables indexed by $n$.
\end{lemma}
\begin{proof}
Let $\theta_0 \in [0, 2\pi )$ be a point where the supremum $\sup_{\theta} \Re\left( \sigma Z_n(\theta) \right)$ is reached. We have from Eq. \eqref{eq:def_Z}:
\begin{align*}
     \sup_{\theta \in [0, 2\pi)} \Re (\sigma Z_n(\theta) )
=  & \Re (\sigma Z_{b^{J+1}}(\theta_0) ) - \Re\left( \sigma \sum_{j=n}^{b^{J+1}-1} \frac{\Nc_j^\C e^{i \psi_j(\theta_0)}}{\sqrt{j+1}}\right)\\
\leq & \sup_{\theta \in [0, 2\pi)} \Re (\sigma Z_{b^{J+1}}(\theta) ) - \Re\left( \sigma \sum_{j=n}^{b^{J+1}-1} \frac{\Nc_j^\C e^{i \psi_j(\theta_0)}}{\sqrt{j+1}}\right) \ .
\end{align*}
Similarly, if $\theta_0'$ is a point where the supremum $\sup_{\theta} \Re\left( \sigma Z_{b^J}(\theta) \right)$ is reached:
\begin{align*}
     \sup_{\theta \in [0, 2\pi)} \Re (\sigma Z_{b^J}(\theta) )
=  & \Re (\sigma Z_{n}(\theta_0') ) - \Re\left( \sigma \sum_{j=b^J}^{n-1} \frac{\Nc_j^\C e^{i \psi_j(\theta_0')}}{\sqrt{j+1}}\right)\\
\leq & \sup_{\theta \in [0, 2\pi)} \Re (\sigma Z_{n}(\theta) ) - \Re\left( \sigma \sum_{j=b^J}^{n-1} \frac{\Nc_j^\C e^{i \psi_j(\theta_0')}}{\sqrt{j+1}}\right) \ .
\end{align*}

Computing the variance, which is made possible by the fact that $\Nc_j^\C$ for $j\geq n$ is independent from $\theta_0$, we find:
$$ \E\left( \left| \sum_{j=n}^{b^{J+1}-1} \frac{\Nc_j^\C e^{i \psi_j(\theta_0)}}{\sqrt{j+1}} \right|^2 \right)
 = \sum_{j=n+1}^{b^{J+1}} \frac{1}{j} \leq \int_n^{b^{J+1}} \frac{dt}{t} \leq \log b \ ,  $$
which implies the second inequality in our Lemma. The first inequality is obtained exactly the same way.
\end{proof}

\subsection{Filling the gaps}
We will use a union bound in order to control the maximum of $Z_n$ on finitely many points of the unit circle, and then interpolate between these points. The latter can be easily made, thanks to the following remarkable general results on polynomials. 

\begin{lemma}
\label{lemma:poly1}
For any polynomial $Q$ of degree at most $k \geq 1$, one has 
$$\sup_{\mathbb{U}} |Q| \leq 14 \sup_{\mathbb{U}_{2k}} |Q|,$$
$\mathbb{U}_m$ denoting the set of $m$-th roots of unity. 
\end{lemma}
\begin{proof}
For $k =1$, we get 
$$Q(z) = \frac{1+z}{2} Q(1) +  \frac{1-z}{2} Q(-1),$$
and for $k = 2$, 
$$Q(z) = \frac{(z+1)(z-i)}{2(1-i)} Q(1) +  \frac{(z-1)(z-i)}{-2(-1-i)}Q(-1) + \frac{(z+1)(z-1)}{(i+1)(i-1)} Q(i),$$
by Lagrange interpolation. We can then assume $k  \geq 3$. 
 Let $m \geq 1$ be the strict integer part of $k/2$. For all integers $r$, $-m \leq r, s \leq k+m$, we have 
 $$\frac{1}{2k} \sum_{\omega \in \mathbb{U}_{2k}} \omega^r \bar{\omega}^s 
 = \mathds{1}_{r = s},$$
 since $|r-s| \leq k + 2m < 2k$. Hence, if $(\lambda_s)_{-m \leq s \leq k+m}$ is a sequence such that 
 $\lambda_s = 1$ for $0 \leq s \leq k$, we get, for all $r \in \{0, 1, \dots, k\},$ 
 $$ z^r = \sum_{\omega \in \mathbb{U}_{2k}} \omega^r \sum_{s = -m}^{k+m} \lambda_s \frac{(z \bar{\omega})^s}{2k}$$
 We deduce, by linearity,
 \begin{equation}  Q(z) =  \sum_{\omega \in \mathbb{U}_{2k}} Q (\omega) R(z \bar{\omega}),
  \label{formulaQ}
 \end{equation}
 where 
 $$R (z) = \frac{1}{2k} \sum_{s = -m}^{k+m} \lambda_s z^s.$$
 Now, let us choose $\lambda_s = (k+m-s)/m$ for $k \leq s \leq k+m$, $\lambda_s = 1$ for $0 \leq s \leq k$, 
 $\lambda_s = (m+s)/m$ for $-m \leq s \leq 0$. For all $z \in \mathbb{U}$, 
 $$|R(z)| \leq \frac{1}{2 k} \sum_{s = -m}^{k+m}  |\lambda_s| = \frac{1}{2k} \left( k+1 + \frac{2}{m} \left( \frac{m(m-1)}{2} \right) \right)
 = \frac{k+m}{2k} \leq \frac{3}{4}.$$
 On the other hand, we can write, for $z \neq 1$,  
\begin{align*} 
 R (z) & = \frac{1}{2km}\sum_{p = 0}^{m-1} \sum_{s = -p}^{k+p} z^s
 = \frac{1}{2km} \sum_{p = 0}^{m-1} z^{-p} \frac{z^{k+2p+1} - 1}{z - 1}
  \\ & =  \frac{1}{2km(z-1)}  \sum_{p = 0}^{m-1} (z^{k+p+1} - z^{-p}) 
  = \frac{1}{2km(z-1)} \left( \frac{(z^{k+1}- z^{-m+1})( z^m - 1)}{z-1} \right),
\end{align*}
and since $m \geq k/4$,
$$|R(z)| \leq \frac{4}{2 k m |z-1|^2} \leq \frac{8}{k^2|z-1|^2}.$$
Hence, from \eqref{formulaQ},
$$ |Q(z)| \leq \left(\sup_{\mathbb{U}_{2k}} |Q| \right)  \sum_{\omega \in \mathbb{U}_{2k}}
 \left( \frac{3}{4} \wedge \frac{8}{k^2 |z - \omega|^2} \right).$$
 Now, the distances, for the circular metric, between $z$ and each of the elements of $\mathbb{U}_{2k}$, are, taken in 
 increasing order, in the intervals $[0, \pi/2k], [\pi/2k, 2\pi/2k], \dots, [(2k-1) \pi/2k, \pi]$. 
 Hence, the successive values of $|z-\omega|$ are at least $0, 1/k, 2/k,  \dots, (2k-1)/k$, and then 
 $$  |Q(z)| \leq \left(\sup_{\mathbb{U}_{2k}} |Q| \right)  
 \left( \frac{3}{4} + \sum_{j=1}^{2k-1} \frac{8}{j^2} \right)
 \leq  \left(\sup_{\mathbb{U}_{2k}} |Q| \right)  \left( \frac{3}{4} + \frac{8 \pi^2}{6} \right),$$
 which gives the desired result. 
\end{proof}

\begin{lemma}
\label{lemma:poly2}
For any polynomial $Q$ of degree at most $k \geq 1$, equal to $1$ at zero and who does not vanish on $\overline{\mathbb{D}}$, 
$$\sup_{\mathbb{U}} \Im \log Q
\leq 2 \pi +  \sup_{\mathbb{U}_k} \Im \log Q, 
 $$
 where we take the continous version of the logarithm which vanishes at zero. 
\end{lemma}
\begin{proof}
We can write, for $\theta \in \mathbb{R}$,  
$$ \Im \log  Q(e^{i \theta}) = 
\sum_{j=1}^k \Im \log (1 - \omega_j e^{i \theta}),$$ 
for $|\omega_j| \leq 1$, the arguments in the sum being in $[-\pi/2, \pi/2]$.  
Since $|\omega_j| < 1$, 
$$\frac{d}{d \theta}
\left(\Im \log(1 - \omega_j e^{i \theta}) \right)
= \Im \left( \frac{d}{d \theta}
\log (1 - \omega_j e^{i \theta}) \right)
= \Im \left( \frac{ - i \omega_j  e^{i \theta}}
{1 - \omega_j e^{i \theta}} \right)
=  \Re \left( \frac{u}{u-1} \right)
$$
for $|u| = |\omega_j e^{i \theta}| < 1$. 
Hence, 
$$\frac{d}{d \theta}
\left(\Im \log(1 - \omega_j e^{i \theta}) \right) = 1 + \Re \left(\frac{1}{u-1} \right)  \leq 1,$$
since $u-1$, and then its inverse, has  negative real part.  
We deduce, for $\theta \in [2 \pi m /k,, 
2 \pi (m+1)/k]$, 
$$\Im \log(1 - \omega_j e^{i \theta})
\leq \frac{2 \pi}{k} +  \Im \log(1 - \omega_j e^{2 i \pi m /k}).$$
Summing with respect to $j$, we get 
$$\Im \log Q(e^{i \theta}) 
\leq 2 \pi + \Im \log Q(e^{2 i \pi m /k})
.$$
\end{proof}

Using the two lemmas above, we deduce the following: 
\begin{proposition}
\label{proposition:holes}
Almost surely, the following random variable is finite: 
$$\mathcal{D} := \sup_{n \geq 1} \left(  \underset{\theta \in [0, 2\pi)}{\sup} \Re (\sigma Z_n (\theta) ) 
- \underset{k \in \{0, 1, \dots, 2n-1\}}{\sup} \Re (\sigma Z_n (\pi k / n) )  \right).$$
\end{proposition}
\begin{proof}
For all $\theta \in [0, 2\pi)$, $n \geq 1$: 
\begin{align*}
     \Re (\sigma Z_n (\theta) )
\leq \quad \quad &  \sqrt{\frac{\beta}{2}} \,   \Re (\sigma \log \Phi_n^* (e^{i \theta}) )  + 
        \left|  \Re (\sigma Z_n (\theta) )  -  \sqrt{\frac{\beta}{2}} \,  \Re (\sigma \log \Phi_n^* (e^{i \theta}) ) \right|\\
\leq \quad \quad & \sqrt{\frac{\beta}{2}}  \;  \underset{z \in \mathbb{U}}{\sup} \, \Re (\sigma \log \Phi_n^* (z) )
       + \underset{n \geq 1}{\sup} \underset{\theta \in [0, 2\pi)}{\sup}\left|  \Re (\sigma Z_n (\theta) )  -  \sqrt{\frac{\beta}{2}} \, \Re (\sigma \log \Phi_n^* (e^{i \theta}) ) \right|\\
\stackrel{Lemmas \ \ref{lemma:poly1}, \ \ref{lemma:poly2}}{\leq} & 
              \sqrt{\frac{\beta}{2}} \sup(\log 14, 2 \pi)
       + \sqrt{\frac{\beta}{2}} \underset{z \in \mathbb{U}_{2n}}{\sup} \Re (\sigma \log \Phi_n^* (z) )
\\ &  + \underset{n \geq 1}{\sup} \underset{\theta \in [0, 2\pi)}{\sup}
  \left|  \Re (\sigma Z_n (\theta) )  -  \sqrt{\frac{\beta}{2}}  \, \Re (\sigma \log \Phi_n^* (e^{i \theta}) ) \right|\\
\leq \quad \quad & \sqrt{2\beta} \pi + \underset{k \in \{0, 1, \dots, 2n-1\}}{\sup} \Re (\sigma Z_n (\pi k / n) )
\\ &   + 2 \,  \underset{n \geq 1}{\sup} \underset{\theta \in [0, 2\pi)}{\sup}
  \left|  \Re (\sigma Z_n (\theta) )  -  \sqrt{\frac{\beta}{2}} \,  \Re (\sigma \log \Phi_n^* (e^{i \theta}) ) \right|.
\end{align*}
The last term is almost surely finite by Proposition \ref{proposition:comparison_to_Z}, and independent of $n$ and $\theta$. We conclude by taking the supremum on $n$ and $\theta$ in the left-hand side. 
\end{proof}

\subsection{Proof of the upper bound via a first moment method}

We are now ready to prove Proposition \ref{proposition:upper_bound}. Thanks to Lemma \ref{lemma:geometric_progressions}, we consider only geometric progressions.

\paragraph{Setting up barriers:} A standard approach in proving the upper bound in log-correlated fields and branching random walks, is to work under the event that ancestors in the hierarchical structure are not too large. In our case, we will follow the field $\Re (\sigma Z_{b^j})$ at successive $j$ up to $J$, and look at the first time when its maximum goes above a certain level, representing a ``barrier'' from above. For our purposes, we first choose the increasing barrier function $g(j) = j^\alpha$ with $\alpha = 1/100$ and then we set: 
\begin{align}
\label{eq:def_barrier_a}
a_j^{J,C} := \left\{
\begin{array}{ll}
C + g(j), \qquad & \text{if  } j \leq \half J ,\\
C + g\left(J-j \right) - \frac{3}{4} \log J ,\qquad & \text{if  } \half J < j \leq J.
\end{array} \right. 
\end{align}
Adding to that the term $\tau(b^j) = \sum_{k=1}^{b^j} \frac{1}{k} = j \log b + \Oc(1)$ (see \eqref{eq:def_tau}), we obtain the tilted barrier:
\begin{align}
\label{eq:def_barrier_A}
A_j^{J,C} := \tau(b^j) + a_j^{J,C} \ .
\end{align}

Now consider the ``barrier crossing'' events
$$ \Bc_{J,C} := \left\{ \exists 1 \leq j \leq  J, \underset{\theta
\in [0, 2\pi)}{\sup} \Re (\sigma Z_{b^j} (\theta))
\geq A_j^{J,C} \right\} \ . $$
In the next paragraph, we will prove:
\begin{align}
\label{eq:setting_barriers}
\lim_{C \rightarrow \infty} \sup_{J \geq 0} \P\left( \Bc_{J, C} \right) = 0 \ . 
\end{align}
It shows that, with high probability as $C \rightarrow \infty$, the maxima of the fields $\Re (\sigma Z_{b^j})$ remain below the barrier. We easily check that \eqref{eq:setting_barriers} implies Proposition \ref{proposition:upper_bound} for the geometric progression $(b^J)_{J \geq 0}$, which completes the proof of the upper bound part of Theorem \ref{thm:main}. 

\paragraph{Proof of Eq. \eqref{eq:setting_barriers}:} Since $\lim_{C \rightarrow \infty} \P\left( \Dc \geq C \right) = 0$ by Proposition \ref{proposition:holes}, it suffices to prove that
$$ \lim_{C \rightarrow \infty} \sup_{J \geq 0} \P\left( \Bc_{J,2C} \cap \{ \Dc < C\} \right) = 0 \ .$$ 
We decompose the event depending on the first instance $j$ where $\underset{\theta\in [0, 2\pi)}{\sup} \Re (\sigma Z_{b^j} (\theta)) \geq A_j^{J,C}$:
\begin{align*}
     & \P\left( \Bc_{J,2C} \cap \{ \Dc < C\} \right) \\
=    & \P\left( \Dc < C; \exists 1 \leq j \leq  J, 
                \underset{\theta\in [0, 2\pi)}{\sup} \Re (\sigma Z_{b^j} (\theta))
                \geq A_j^{J,2C} \right)\\
=    & \sum_{j=1}^{J}
       \P\left( \Dc < C; \forall 1 \leq k < j, 
                \underset{\theta\in [0, 2\pi)}{\sup} \Re (\sigma Z_{b^k} (\theta))
                < A_k^{J,2C} ; 
                \underset{\theta\in [0, 2\pi)}{\sup} \Re (\sigma Z_{b^j} (\theta))
                \geq A_j^{J,2C} \right)\\
\leq & \sum_{j=1}^{J}
       \P\left( \forall 1 \leq k < j, 
                \underset{\theta\in \frac{2 \pi \Z}{2 b^{j}}}{\sup} \Re (\sigma Z_{b^k} (\theta))
                < A_k^{J,2C} ; 
                \underset{\theta\in \frac{2 \pi \Z}{2 b^{j}}}{\sup} \Re (\sigma Z_{b^j} (\theta))
                \geq A_j^{J,C} \right) \ ,
\end{align*}
where for the last inequality we used the definition of $\Dc$, which indicates we only need to control the fields $\Re (\sigma Z_{b^j})$ on $2b^j$ points at the cost of an error of $C$. Then, thanks to a union bound and the rotational invariance of the fields $\Re (\sigma Z_{b^j} (\theta))$:
\begin{align*}
     & \P\left( \Bc_{J,2C} \cap \{ \Dc < C\} \right)\\
\leq & \sum_{j=1}^{J}
       2 b^{j}
       \P\left( \forall 1 \leq k < j, 
                \Re (\sigma Z_{2^k} (\theta=0))
                < A_k^{J,2C} ; 
                \Re (\sigma Z_{2^j} (\theta=0))
                \geq A_j^{J,C} \right)\ .
\end{align*}
Now, we embed the Gaussian random variables $\Re (\sigma Z_{2^k} (\theta=0))$ into a standard real Brownian motion $W$ as in Corollary \ref{corollary:coupling_bm} by writing:
$$ \forall k \in \N, \ W_{\tau(k)} = \sqrt{2} \Re (\sigma Z_{k} (\theta=0)) \ .$$
By the Girsanov transform \cite[Chapter VIII, Theorem 1.12]{bib:RY}, we define a new probability measure $\Q$ given by:
$$ \frac{d\Q}{d\P}_{| \sigma\left( W_s ; s \leq t \right) } = e^{ \sqrt{2} W_{t} - t}$$
under which $\widetilde{W}_t = W_t - \sqrt{2} t$ is a $\Q$-Brownian motion. In that setting, we continue from the above inequalities and obtain:
\begin{align*}
     & \P\left( \Bc_{J,2C} \cap \{ \Dc < C\} \right)\\
\leq & \sum_{j=1}^{J}
       2 b^{j}
       \P\left( \forall 1 \leq k < j, 
                W_{\tau(b^k)}
                \leq \sqrt{2} A_k^{J,2C} ; 
                W_{\tau(b^j)}
                > \sqrt{2} A_j^{J,C} \right)\\
=    & \sum_{j=1}^{J}
       2 b^{j}
       \P\left( \forall 1 \leq k < j, 
                \widetilde{W}_{\tau(b^k)}
                \leq \sqrt{2} a_k^{J,2C} ; 
                \widetilde{W}_{\tau(b^j)}
                > \sqrt{2} a_j^{J,C} \right)\\
=   & \sum_{j=1}^{J}
       2 b^{j}
       \E^\Q\left( e^{- \sqrt{2} \widetilde{W}_{\tau(b^j)} - \tau(b^j)}
                \mathds{1}_{ \left\{ \forall 1 \leq k < j, 
                \widetilde{W}_{\tau(b^k)}
                \leq \sqrt{2} a_k^{J,2C} ; 
                \widetilde{W}_{\tau(b^j)}
                > \sqrt{2} a_j^{J,C} \right\} } \right)\\
\ll_b & \sum_{j=1}^{J}
       e^{-2 a_j^{J,C} }
       \Q\left( \forall 1 \leq k < j, 
                \widetilde{W}_{\tau(b^k)}
                \leq \sqrt{2} a_k^{J,2C} ; 
                \widetilde{W}_{\tau(b^j)}
                > \sqrt{2} a_j^{J,C} \right) \ .
\end{align*}
Out of convenience, we replace in the above sum $\widetilde{W}$ under $\Q$ by $W$ under $\P$ while performing a time change:
$$ \widetilde{W}_{\tau(b^k)} = \sqrt{\log b} W_{k + E_k} \ ,$$
where $E_k$ is the time shift arising from the approximation $\tau(b^k) = k \log b + \Oc(1)$. More precisely:
$$ E_k = \frac{\tau(b^k)}{\log b} - k =  \sum_{l=1}^k e_l \ ,$$
with $e_k = \frac{\mathds{1}_{k = 1}}{\log b} -1 + \frac{ \sum_{p=b^{k-1}+1}^{b^k} \frac{1}{p}}{\log b} \stackrel{k \rightarrow \infty}{\longrightarrow}0$. By comparing the harmonic series to the integral of $t \mapsto \frac{1}{t}$, recall the standard inequality:
\begin{align}
\label{eq:tau_log_comparison}
  \forall p<q, \ -\log\left( 1 + \frac{1}{p} \right)
             \leq \tau(q)-\tau(p) - \log \frac{q}{p} 
             \leq 0 \ .
\end{align}
It yields that 
$$||E||_1 := \sum_{j=1}^{\infty} |e_k| 
\leq \frac{1}{\log b}  + \frac{1}{\log b} \sum_{k=1}^{\infty} \log\left( 1 + b^{1-k} \right) $$
is finite and depends only on $b$. At this point, notice that picking $b$ large allows $|| E ||_1$ to be as small as desired. In any case, we have:
\begin{align*}
     & \P\left( \Bc_{J,2C} \cap \{ \Dc < C\} \right)\\
\ll_b & \sum_{j=1}^{J}
       e^{-2 a_j^{J,C} }
       \P\left( \forall 1 \leq k < j, 
                W_{k+E_k}
                \leq \sqrt{\frac{2}{\log b}} a_k^{J,2C} ; 
                W_{j+E_j}
                > \sqrt{\frac{2}{\log b}} a_j^{J,C} \right) \ .
\end{align*}

Because of the barrier's definition, as $J \rightarrow \infty$, we have 
$ \sum_{1 \leq j \leq \frac{2}{3} J}
       e^{-2 a_j^{J,C} } = \Oc(e^{-2C}) . $
And thus:
\begin{align*}
    & P\left( \Bc_{J,2C} \cap \{ \Dc < C\} \right)\\
\ll_b & e^{-2C} + \sum_{ \frac{2}{3} J < j \leq J}
       e^{-2 a_j^{J,C} }
      \P\left( \forall 1 \leq k < j, 
                W_{k+E_k}
                \leq \sqrt{\frac{2}{\log b}} a_k^{J,2C} ; 
                W_{j+E_j}
                > \sqrt{\frac{2}{\log b}} a_j^{J,C} \right)\\
\leq   & e^{-2C} + e^{-2C} J^{\frac{3}{2}} \sum_{ \frac{2}{3} J < j \leq J}
       e^{-2g(J - j) }
       \P\left( \forall 1 \leq k < j, 
                W_{k+E_k}
                \leq \sqrt{\frac{2}{\log b}} a_k^{J,2C} ; 
                W_{j+E_j}
                > \sqrt{\frac{2}{\log b}} a_j^{J,C} \right) \ .
\end{align*}

In order to conclude, we need the following, which is proven in the next paragraph:
\begin{align}
\label{eq:upper_bound_estimate}
\P\left( \forall 1 \leq k <j, 
                W_{k+E_k} \leq \sqrt{\frac{2}{\log b}} a_k^{J,2C} ; 
                W_{j+E_{j}} > \sqrt{\frac{2}{\log b}} a_{j}^{J,C} \right)
 & \ll_b  (1+C)^3 J^{-\frac{3}{2}} \ ,
\end{align}
for all $ \frac{2}{3}J < j \leq J$, the implicit constant being independent of $C>0$ and $j$. Assuming that, we have:
\begin{align*}
      P\left( \Bc_{J,2C} \cap \{ \Dc < C\} \right)
\ll_b & e^{-2C} + e^{-2C} (1+C)^3 \sum_{ \frac{2}{3} J < j \leq J}
       e^{-2g(J - j) }
\ll e^{-2C} (1+C)^3 \ .
\end{align*}
As the implicit constants do not depend on $C$, taking $C \rightarrow \infty$ concludes the proof of Eq. \eqref{eq:setting_barriers}.

\paragraph{Proof of eq. \eqref{eq:upper_bound_estimate}:} This equation is trivial for $J \leq 13$, hence we can assume 
$J \geq 14$. The proof requires the use of classical estimates on random walks which are given in the appendix. In order to obtain a more amenable expression, let us handle the overshoot of the random walk $W$ over the barrier at time $j > \frac{2J}{3} > 9$:
\begin{align*}
  & \P\left( \forall 1 \leq k < j, 
             W_{k+E_k} \leq \sqrt{\frac{2}{\log b}} a_k^{J,2C} ; 
             W_{j+E_{j}} > \sqrt{\frac{2}{\log b}} a_{j}^{J,C} \right)\\
= & \sum_{l=0}^\infty \P\Big( \forall 1 \leq k < j, 
             W_{k+E_k} \leq \sqrt{\frac{2}{\log b}} a_k^{J,2C} ;
             W_{j-1+E_{j-1}} \in \sqrt{\frac{2}{\log b}} a_{j-1}^{J,2C} - [l;l+1) ; \\
  &          \quad \quad
             W_{j+E_{j}} > \sqrt{\frac{2}{\log b}} a_{j}^{J,C} \Big)\\
\leq & \sum_{l=0}^\infty \P\Big( \forall 1 \leq k < j, 
             W_{k+E_k} \leq \sqrt{\frac{2}{\log b}} a_k^{J,2C} ; 
             W_{j-1+E_{j-1}} \in \sqrt{\frac{2}{\log b}} a_{j}^{J,2C} - [l;l+1) ; \\
  &          \quad \quad
             W_{j+E_{j}} - W_{j-1+E_{j-1}} > \sqrt{\frac{2}{\log b}} (a_{j}^{J,C} -  a_{j-1}^{J,2C}) + l \Big)\\
\ll  & \sum_{l=0}^\infty \P\left( \forall 1 \leq k < j, 
             W_{k+E_k} \leq \sqrt{\frac{2}{\log b}} a_k^{J,2C} ; 
             W_{j-1+E_{j-1}} \in \sqrt{\frac{2}{\log b}} a_{j}^{J,2C} - [l;l+1) \right) \times \\
  &          \quad \quad
             \P\left( W_{j+E_{j}} - W_{j-1+E_{j-1}} > -\sqrt{\frac{2}{\log b}}C + l + \Oc(1) \right)
\end{align*}
We now invoke the first point of Corollary \ref{corollary:shifted_barrier_estimate} with
$f = \sqrt{\frac{2}{\log b}} g $, $N = j-1$, $\lambda N = \half J + \frac{1}{4}$, $x=2C \sqrt{\frac{2}{\log b}}$, $ y = \sqrt{\frac{2}{\log b}} \left(g(J-j+1) - \frac{3}{4} \log J\right)$. Note that we can apply this corollary for $J$ large enough depending only on $||E||_1$, and then only on $b$, since 
$$\frac{1}{10} \leq \frac{1}{2} \leq \frac{\half J}{J-1} 
\leq \lambda \leq  \frac{\half (J+1)}{j-1}
\leq  \frac{\half (J+1)}{ \frac{2J}{3} - 1} \leq 
\frac{9}{10},$$
the last equality coming from the fact that we assume $J \geq 14$. Moreover, 
if $W_{k+ E_k} \leq \sqrt{\frac{2}{\log b}} 
a_k^{J, 2C}$, we have for $k \leq J/2 < \lambda N$, 
$$W_{k + E_k} - f(k) \leq 2C \sqrt{\frac{2}{\log b}} = x,
$$
and for $k > J/2$, and then $k \geq (J/2) + (1/2)  > \lambda N$, by subadditivity of $f$ (which is proportional to $k \mapsto k^{1/100}$), 
$$W_{k + E_k} - f(j-1 - k)
\leq W_{k + E_k} - f(J-k) + f(J-j+1)
\leq \sqrt{\frac{2}{\log b}}( 2C - \frac{3}{4}\log J  + g(J-j+1) ) = x + y.$$
Moreover, if 
$$W_{j-1 + E_{j-1}} \in \sqrt{\frac{2}{\log b}}
a_j^{J, 2C}  - [\ell, \ell+1),$$
then 

$$W_{j-1 + E_{j-1}} 
\in x+y - [\ell, \ell + 1).$$

We then obtain, for $J$ large enough depending on $b$:  
\begin{align*}
     & \P\left( \forall 1 \leq k < j, 
             W_{k+E_k} \leq \sqrt{\frac{2}{\log b}} a_k^{J,2C} ; 
             W_{j+E_{j}} > \sqrt{\frac{2}{\log b}} a_{j}^{J,C} \right)\\
\ll_b  & \, J^{-\frac{3}{2}} (1+C) \sum_{l=0}^\infty (1+l)
       \P\left( W_{j+E_{j}} - W_{j-1+E_{j-1}} > -\sqrt{\frac{2}{\log b}}C + l + \Oc(1) \right)\\
\ll  & J^{-\frac{3}{2}} (1+C) \sum_{l=0}^\infty (1+l)
       \P\left( \Nc > \frac{-\sqrt{\frac{2}{\log b}}C + l + \Oc(1)}{\sqrt{1+e_{j}}} \right),
\end{align*}
where $\Nc$ denotes a standard Gaussian random variable. Note that since $j \geq 10$ (say), 
$$1 + e_j = \frac{\tau(b^j) - \tau(b^{j-1})}{\log b}
          \stackrel{\textrm{Eq. \ref{eq:tau_log_comparison}}}{\geq}
           1 - \frac{\log \left( 1 + b^{1-j} \right)}{\log b}
          \geq 1 - \frac{b^{-9}}{\log b}
\geq \half \ .$$
Hence, for some integer 
$$l_0 = \left\lceil \sqrt{\frac{2}{\log b}} \, C + \mathcal{O}(1) \right\rceil 
      \ll 1 + C,$$
\begin{align*}
     & \P\left( \forall 1 \leq k < j, 
             W_{k+E_k} \leq \sqrt{\frac{2}{\log b}} a_k^{J,2C} ; 
             W_{j+E_{j}} > \sqrt{\frac{2}{\log b}} a_{j}^{J,C} \right)\\
\ll_b  & J^{-\frac{3}{2}} (1+C) 
\left( \sum_{l=0}^{l_0} (1+l)
      + \sum_{l = l_0 + 1}^{\infty}
       (1+ l) \exp \left( - (l-l_0)^2/4 \right)      
       \right)\\
       \ll_b & J^{-\frac{3}{2}} (1+C)^3,
\end{align*}
which proves \eqref{eq:upper_bound_estimate} first for $J$ large enough depending on $b$, and then for all $J$ by changing the implicit constant. 

\section{The lower bound}
\label{section:lower_bound}

Thanks to Proposition \ref{proposition:comparison_to_Z} and Corollary \ref{corollary:phistar}, the ``lower bound part" of
Theorem \ref{thm:main} is a consequence of the following result:  
\begin{proposition}
\label{proposition:lower_bound}
For all $\sigma \in \{1, i, -i\}$, we have 
$$\underset{C \rightarrow \infty}{\lim} \underset{n \rightarrow \infty}{\lim \sup}  \, 
\P   \left[ \sup_{\theta \in [0, 2\pi)} \Re (\sigma Z_n(\theta) )
\leq \log 
n - \frac{3}{4} 
\log_2 n  - C  \right] = 0.
$$
\end{proposition}

Before diving into the proof, we will introduce a new process in order to gain more independence. Then 
in subsection \ref{subsection:second_moment}, we will implement the classical second moment method to our problem. The second moment method, originally introduced for the branching processes, like branching Brownian motion or Branching random walk, is now a very classical method pioneered by Bramson \cite{Bra78}. Since that time, it has been widely used in the general 
context of log-correlated field (\cite{BZe10}, \cite{BDZ13}, \cite{bib:Mad13}, \cite{RDZ15}). 

\paragraph{Further notation:}
For any $n\geq p$, $\theta\in [0,2\pi)$, we define
\begin{align*}
Z_n^{(p)}(\theta):= Z_n(\theta) - Z_p(\theta) =  \sum_{j=p}^{n-1} \frac{ \Nc_j^\C }{\sqrt{j+1}} e^{i((j+1)\theta + A_j(\theta))} .
\end{align*}
We stress that $Z_n^{(p)}$ depends implicitely of $\beta$ because of the Pr\"ufer phase. Similarly, 
we define
\begin{align*}
A_n^{(p)}(\theta):= \sum_{j=p}^{n-1} a_j(\theta) = A_n(\theta) - A_p(\theta),\ .
\end{align*} 
where 
$$a_j(\theta) := A_{j+1}(\theta) - A_j(\theta).$$
 More generally, for any family of quantities depending on an index $k$, we will denote the difference 
 of the quantities indexed by $k$ and $p$ by the same notation, with $k$ as an index and $p$ as 
 an upperscript.
 
In the following, it will be convenient to study the field at times which are powers of $2$. In the sequel, 
we denote $\e_k := 2^k$. 

\subsection{A new coupling}
\label{section:new_coupling}

\paragraph{A more independent field:}
For each fixed $\theta$,  $(Z_{\e_k}(\theta))_{k\geq 0}$ is a complex Gaussian random walk. Moreover we
could compute the correlations of $Z_{\e_k}(\theta)$ and $Z_{\e_k}(\theta')$ and observe that 
they behave logarithmically with respect to the distance between $\theta$ and $\theta'$ modulo $2  \pi$.
However, $Z$ is not globally Gaussian, so we cannot directly apply known results on the maximum of Gaussian fields, 
but we will still provide its approximative branching structure. To achieve this aim we will  gain
some independence by making small changes on $Z$. 

Let us fix some integer $r$, which will be  assumed to be larger than some suitable universal constant. 
For $l \geq r$, we denote $\Delta=\Delta^{(l)}:= \lfloor \frac{r}{100} \rfloor  + 100\lfloor  \log^2 l \rfloor $. Observe that for any $N \geq r$, we can rewrite formula \eqref{eq:def_Z} as:
\begin{align}
Z_{\e_{N}}^{(\e_r)}(\theta):= & \sum_{l=r}^{N -1 } \sum_{p=0}^{2^{\Delta}-1}  \left(\sum_{j=0}^{2^{l-\Delta}-1}
\Nc^\C_{2^{l}+p 2^{l-\Delta} + j}
\frac{e^{i\psi_{ 2^{l}+p 2^{l-\Delta}+j}(\theta )}}
     {\sqrt{ 2^{l}+p 2^{l-\Delta} + j +1 }}
\right).
\end{align}
Note that $\Delta^{(l)}$ and $l - \Delta^{(l)} \geq ( 99l/100) - 100 \log^2 l $ are strictly positive if $l \geq r$ and $r$ is large enough. 
Now, let $Z^{(\e_r,\Delta)}$ be the process defined by
\begin{align}
Z_{\e_{N}}^{(\e_r,\Delta)}(\theta) := & \sum_{l=r}^{N -1 } \sum_{p=0}^{2^{\Delta}-1}  \left(\sum_{j=0}^{2^{l-\Delta}-1}    \Nc^\C_{2^{l}+p 2^{l-\Delta} + j} e^{i j\theta } \right)
\frac{e^{i\psi_{ 2^{l}+p 2^{l-\Delta}}(\theta )}}
     {\sqrt{ 2^{l}+p 2^{l-\Delta}}}
.
\end{align}
Observe that $Z_{\e_{N}}^{(\e_r)}(\theta)$ and $Z_{\e_{N}}^{(\e_r,\Delta)}(\theta)$ only differ
by the change in the square root of the denominator, and by the replacement of
some increments of the  Pr\"ufer phases by their mean. We claim that 
\begin{proposition}
For $r$ large enough, 
	\label{proposition:new_coupling}
	\begin{align*}
\sum_{l\geq r} \P\left(	\sup_{\theta \in [0,2\pi)} | Z_{\e_{l+1}}^{(\e_{l},\Delta)}(\theta) - Z_{\e_{l+1}}^{(\e_{l})}(\theta)|\geq 2 l^{-2} \right) <+\infty
	\end{align*}
In particular, as $\sum_{l\geq r } 2 l^{-2}<+\infty$, we have that almost surely,
\begin{align}
\sum_{l \geq r} \sup_{\theta \in [0,2\pi)}
|  Z_{\e_{l+1}}^{(\e_{l},\Delta)}(\theta) - Z_{\e_{l+1}}^{(\e_{l})}(\theta) | < \infty \ .
\end{align}
\end{proposition}
\begin{proof} For any $l\geq 0$, $\theta\in [0,2\pi)$, we  introduce
\begin{align}
\label{tighti2}
\blacksquare_l(\theta) & :=  Z_{\e_{l+1}}^{(\e_{l},\Delta)}(\theta) - Z_{\e_{l+1}}^{(\e_{l})}(\theta)
\\
\nonumber &  = \sum_{p=0}^{2^{\Delta}-1} 
               \frac{ e^{i\psi_{ 2^{l}+p 2^{l-\Delta} }(\theta )} }
                    { \sqrt{{ 2^{l}+p 2^{l-\Delta}}} }
               \left(\sum_{j=0}^{2^{l-\Delta}-1} 
               \Nc^\C_{2^{l}+p 2^{l-\Delta} + j} e^{i j \theta }   \square_j^{ (2^{l}+p 2^{l-\Delta}) }(\theta)
               \right) \ ,
\end{align}
with 
\begin{align}
\label{eq:ineq_square}
       \left|\square_j^{(2^{l}+p 2^{l-\Delta} )}(\theta)\right|
\leq &  \left| A_{j+ 2^{l}+p 2^{l-\Delta} }^{(2^{l}+p 2^{l-\Delta})}(\theta) \right| + 2^{-\Delta}.
\end{align}
Indeed:
\begin{align*}
     \square_j^{(2^{l}+p 2^{l-\Delta} )}(\theta)
 = & 1- \sqrt{\frac{2^l+p2^{l-\Delta}}{2^l+p2^{l-\Delta} +j+1}}
        e^{i\psi_{ 2^{l}+p 2^{l-\Delta} + j}(\theta )
          -i\psi_{ 2^{l}+p 2^{l-\Delta} }(\theta )-ij\theta
        }\\
 = & 1- \left( 1 + \Theta \right)
        e^{i A_{j+ 2^{l}+p 2^{l-\Delta} }^{(2^{l}+p 2^{l-\Delta})}(\theta)},
\end{align*}
where $|\Theta| \leq 2^{-\Delta}$, which implies \eqref{eq:ineq_square}.
As we shall see, with an overwhelming probability, the random variables $\square_j^{( 2^l+p2^{l-\Delta}  )}$ are small enough to control $Z_{\e_N}^{(\square)}$. Our strategy is once more to use a one point estimate and a union bound on a fine mesh of $[0,2\pi]$. Then we will apply Corollary \ref{corollary:djo}.

\paragraph{One point estimate:} In this paragraph, we prove that for all $\theta \in [0,2\pi]$, $l\geq 1$:
\begin{align}
\label{eq:one_point_square}
\P\left( | \blacksquare_l(\theta) | \geq l^{-2} \right)
\ll_{\beta} e^{- 2^{ \quart \Delta} }.
\end{align}
In order to do so, we define first the following "good events", on which
we shall be able to ensure that $\square_j^{( 2^l+p2^{l-\Delta}  )}$ is
small. For any $\theta\in [0,2\pi]$, $l\geq r$, let

\begin{align}
\label{eq:def_G_l}
G_l(\theta) := & \ \bigcap_{p=0}^{2^{\Delta}-1}
\left\{ \sup_{0\leq j\leq 2^{l-{\Delta}}-1} \left| A_{j+ 2^{l}+p 2^{l-\Delta} }^{(2^{l}+p 2^{l-\Delta})}(\theta) \right|
\leq 2^{-\quart \Delta}
\right\} \ .
\end{align}

For $j \geq 0$, let $\mathcal{G}_j$ be the $\sigma$-algebra generated by 
$(\alpha_0, \Nc^{\mathbb{C}}_0), \dots, (\alpha_{j-1},  \Nc^{\mathbb{C}}_{j-1})$. 
By a union bound and Proposition \ref{proposition:prufer_deviation}, with $\mathcal{F}_j$ replaced by $\mathcal{G}_j$, 
which does not change the proof:
\begin{align*}
 1 - \P\left( G_l(\theta)  \Big| \mathcal{G}_{2^l} \right) 
\leq & \sum_{p=0}^{2^{\Delta}-1} 
\P\left( \sup_{0\leq j\leq 2^{l-{\Delta}}-1} \left| A_{j+ 2^{l}+p 2^{l-\Delta} }^{(2^{l}+p 2^{l-\Delta})}(\theta) \right|
\geq 2^{-\quart \Delta}    \Big| \mathcal{G}_{2^l}\right)\\
\leq & 4 \sum_{p=0}^{2^{\Delta}-1} \exp\left( -\frac{\beta 2^{-2 \quart \Delta}}
                                                    {32 \log\left( 1 + \frac{\beta 2^{l-\Delta}}{1+\beta(2^{l}+p 2^{l-\Delta})}\right)} \right)\\
\leq & 4 \sum_{p=0}^{2^{\Delta}-1} 
       \exp\left( -\frac{2^{-l+\half \Delta}}{32}
                  \left( 1+\beta(2^{l}+p 2^{l-\Delta}) \right)
           \right)\\
\leq & 4 \sum_{p=0}^{2^{\Delta}-1} 
       \exp\left( -\frac{\beta}{32} 2^{\half \Delta} \left( 1+p 2^{-\Delta} \right)
           \right)
 \ll_{\beta}  \exp\left( -\frac{\beta}{32} 2^{\half \Delta} \right).
\end{align*}
 As such, these good events happen with overwhelming probability. Now, we will estimate the Laplace transform of $\blacksquare_l(\theta)$ on the event $G_l(\theta)$. We will prove that for $\lambda \in \R$:
\begin{align}
\label{eq:laplace_blacksquare}
\sup\left( 
\E\left( e^{ \lambda \Re  \blacksquare_l(\theta)} \mathds{1}_{G_l(\theta)}  \Big| \mathcal{G}_{2^l}\right),
\E\left( e^{ \lambda \Im  \blacksquare_l(\theta)} \mathds{1}_{G_l(\theta)}  \Big| \mathcal{G}_{2^l}\right)
\right)
\leq &
e^{\lambda^2 \left( 2^{-\quart \Delta} + 2^{-\Delta} \right)^2 } \ .
\end{align}
We now handle the real part: the proof for the imaginary part is the same by replacing $\Re$ by $\Im$ everywhere. We have:
\begin{align*}
     & \E\left( e^{ \lambda \Re  \blacksquare_l(\theta)} \mathds{1}_{G_l(\theta)}  \Big| \mathcal{G}_{2^l} \right)\\
=    & \E\left(  
       \exp\left\{ \lambda \Re  \sum_{p=0}^{2^{{\Delta}}-1}
       \frac{ e^{i\psi_{ 2^{l}+p 2^{l-{\Delta}} }(\theta ) }}
            {\sqrt{{ 2^{l}+p 2^{l-{\Delta}}  }}  }
       \left( \sum_{j=0}^{2^{l-{\Delta}}-1} \Nc^\C_{2^{l}+p 2^{l-{\Delta}} + j}   e^{i j \theta}  \square_j^{(  2^{l}+p 2^{l-{\Delta}})}   \right)
           \right\}
       \mathds{1}_{G_l(\theta)} \Big| \mathcal{G}_{2^l} \right)
\end{align*}

To compute the expectation, we shall proceed by backward induction.
 We will condition on $\Gc_{2^{l+1}-1}$ then $\Gc_{2^{l+1}-2}$, $\dots$, all the way down to $\Gc_{2^{l}}$. To that
 endeavor, observe from \eqref{eq:def_G_l} that $G_l(\theta)$ is a decreasing intersection
$$G_l(\theta) = \bigcap_{ \substack{ 0 \leq p \leq 2^{\Delta}-1 \\
                                 0 \leq j \leq 2^{l-\Delta} -1} }
               \downarrow G_{l, p 2^{l-\Delta} + j}(\theta)$$
of events where $G_{l,0} = \Omega$ (the event of probability $1$), for $0 \leq p \leq 2^{\Delta}-1$ and $1 \leq j \leq 2^{l-\Delta} - 1$, 
$G_{l,p 2^{l-\Delta} + j }(\theta)$ is the $\Gc_{2^l + p 2^{l-\Delta} + j}$-measurable event given by:
$$ G_{l, p 2^{l-\Delta} + j}(\theta) = G_{l, p 2^{l-\Delta} + j-1}(\theta) \cap \left\{ \left| A_{j+ 2^{l}+p 2^{l-\Delta} }^{(2^{l}+p 2^{l-\Delta})}(\theta) \right|
\leq 2^{-\quart \Delta} \right\} \ ,
$$
and for $0 \leq p \leq 2^{\Delta} - 2$,
$$  G_{l, (p+1) 2^{l-\Delta} }(\theta) =  G_{l, p 2^{l-\Delta} + 2^{\Delta} - 1}(\theta)  \cap 
\left\{ \left| A_{2^{l}+(p+1) 2^{l-\Delta} }^{(2^{l}+(p+1) 2^{l-\Delta})}(\theta) \right|
\leq 2^{-\quart \Delta} \right\} = G_{l, p 2^{l-\Delta} + 2^{\Delta} - 1}(\theta)  .$$

Because of the inequality \eqref{eq:ineq_square}, we have, for 
$0 \leq p \leq 2^{\Delta}-1$ and $0 \leq j \leq 2^{l-\Delta} - 1$, 
\begin{align*}
     & \E\left(  
       \exp\left\{ \lambda \Re
       \frac{ e^{i\psi_{ 2^{l}+p 2^{l-{\Delta}} }(\theta ) }}
            {\sqrt{{ 2^{l}+p 2^{l-{\Delta}}  }}  }
       \Nc^\C_{2^{l}+p 2^{l-{\Delta}} + j}   e^{i j \theta}  \square_j^{(  2^{l}+p 2^{l-{\Delta}})}
           \right\}
       \mathds{1}_{G_{l, p 2^{l-\Delta} + j}(\theta)} | \Gc_{2^l + p 2^{l-\Delta} + j}\right)\\
=    & \E\left(  
       \exp\left\{ \frac{\lambda^2}{4}
              \frac{ \left| \square_j^{(  2^{l}+p 2^{l-{\Delta}})} \right|^2 }
                   { 2^{l}+p 2^{l-{\Delta}} }
           \right\}
       | \Gc_{2^l + p 2^{l-\Delta} + j}\right)
       \mathds{1}_{G_{l, p 2^{l-\Delta} + j}(\theta)}\\
\leq & \exp\left\{ \lambda^2
              \frac{ \left( 2^{-\quart \Delta} + 2^{-\Delta} \right)^2 }
                   { 2^{l}+p 2^{l-{\Delta}} }
           \right\}
       \mathds{1}_{G_{l, p 2^{l-\Delta} + j-1}(\theta)},
\end{align*}
with the convention $G_{l, -1} = G_{l,0} = \Omega$ for $p=j=0$. As a consequence:
$$
\E\left( e^{ \lambda \Re  \blacksquare_l(\theta)} \mathds{1}_{G_l(\theta)}  \Big| \mathcal{G}_{2^l} \right)\\
\leq e^{ \lambda^2 \sum_{p=0}^{2^{{\Delta}}-1} \sum_{j=0}^{2^{l-{\Delta}}-1} 
         \frac{ \left( 2^{-\quart \Delta} + 2^{-\Delta} \right)^2 }
              { 2^{l}+p 2^{l-{\Delta}} }
       }
\leq e^{ \lambda^2 \left( 2^{-\quart \Delta} + 2^{-\Delta} \right)^2 }, 
$$
which concludes the proof of \eqref{eq:laplace_blacksquare}. From that, one deduces thanks to a Chernoff bound
 ($\lambda = \pm 2^{\Delta/2}/(4 l^2)$):
$$ \sup\left( 
   \P\left( \pm \Re \blacksquare_l(\theta) \geq  l^{-2}/2, G_l(\theta) \right),
   \P\left( \pm \Im \blacksquare_l(\theta) \geq l^{-2}/2, G_l(\theta) \right) 
   \right)
   \leq e^{- \frac{2^{\Delta/2}}{16 l^4} + \mathcal{O}(1)} \ll e^{- 2^{\Delta/4}},$$
   since 
   $$ \frac{2^{\Delta/2}}{16 l^4} - 2^{\Delta/4} \geq 2^{\Delta/4} \left(  \frac{2^{(1/4)\lfloor 100 \log^2 l \rfloor}} 
   {16l^4} - 1 \right) \geq 0 $$
   if $r$ is large enough and $l \geq r$. 
   
   This implies the one-point estimate \eqref{eq:one_point_square}, since:
\begin{align*}
     & \P\left( |  \blacksquare_l(\theta) |\geq l^{-2} \right)\\
\leq & 1-\P\left( G_l(\theta) \right)
       + \P\left( |  \blacksquare_l(\theta) |\geq l^{-2} ,  G_l(\theta) \right)\\
\leq & 1-\P\left( G_l(\theta) \right)
       + 4 \sup \left( \P\left( \pm \Re \blacksquare_l(\theta) \geq l^{-2}/2 ,  G_l(\theta) \right),
                \P\left( \pm \Im \blacksquare_l(\theta) \geq l^{-2}/2,  G_l(\theta) \right) \right)\\
\ll_{\beta}  & e^{-2^{\quart \Delta}} + e^{- \frac{\beta}{33} 2^{\half \Delta}}
\ll_{\beta}    e^{-2^{\quart \Delta}} \ .
\end{align*}

\paragraph{Multiple point estimate:} Let $D>2$. By combining the one-point estimate \eqref{eq:one_point_square} with
a union bound over $2^{D(l+1)}$ points, as $2^{D(l+1)} e^{ - 2^{\quart \Delta^{(l)}} }$ is summable (since 
$\Delta^{(l)} \gg \log^2 l$), we have
$$
\sum_{l\geq r} 
\P\left( 
  \sup_{ \theta \in [0,2\pi)_{2^{D(l+1)}} } \left| \blacksquare_l(\theta) \right| \geq l^{-2}
\right) 
< \infty \ .
$$
Also, if $N_l = \left\{ 2^{-l}\sum_{j=2^{l}}^{2^{l+1}-1} |\Nc^\C_j| \leq 2 \right\}$, $\sum_{l\geq r} (1-\P\left( N_l \right)) < \infty$ by a simple large deviation estimate. Therefore, it is sufficient to prove that
\begin{align*}
\sum_{l\geq r} 
\P\left( 
\sup_{ \substack{ \theta_1 \in [0, 2\pi)_{2^{-D(l+1)}} \\ \theta_2-\theta_1 \in [0, 2\pi 2^{-D(l+1)} ] } } 
\left| \blacksquare_l(\theta_2) - \blacksquare_l(\theta_1) \right| \geq l^{-2} , N_l \right)
& < +\infty \ .
\end{align*}
Moreover, on $N_l$ and for a fixed $\theta_1$, we have the following crude bound:
\begin{align}
\label{eq:crude_bound}
     & \sup_{\theta_2-\theta_1\in [0,2\pi 2^{-D(l+1)}] }|\blacksquare_l(\theta_1) -\blacksquare_l(\theta_2)|\\
\nonumber
\ll &  2^{-(D-(3/2)) l} + 
       2^{l/2}\sup_{2^l\leq j\leq 2^{l+1}-1}
      \left(   \psi_j(\theta_1 + 2\pi 2^{-D(l+1)})-\psi_j(\theta_1) \right). 
\end{align}
Indeed, in order to establish the above inequality, we start by:
\begin{align*}
      & \sup_{\theta_2-\theta_1\in [0,2\pi 2^{-D(l+1)}] }|\blacksquare_l(\theta_1) -\blacksquare_l(\theta_2)|\\
 \leq & \sup_{\theta_2-\theta_1\in [0,2\pi 2^{-D(l+1)}] }
        |Z_{\e_{l+1}}^{(\e_{l})}(\theta_1) -Z_{\e_{l+1}}^{(\e_{l})}(\theta_2)|
      + \sup_{\theta_2-\theta_1\in [0,2\pi 2^{-D(l+1)}] }
        |Z_{\e_{l+1}}^{(\e_{l},\Delta)}(\theta_1) -Z_{\e_{l+1}}^{(\e_{l}, \Delta)}(\theta_2)|
\end{align*}
and then bound each term separately. For the first term, we have:
\begin{align*}
      & | Z_{\e_{l+1}}^{(\e_{l})}(\theta_1) -Z_{\e_{l+1}}^{(\e_{l})}(\theta_2) |\\
 \leq & \left[ 
        \sum_{p=0}^{2^\Delta-1} \sum_{j=0}^{2^{l-\Delta}-1} 
        \frac{\left| \Nc_{2^l+p 2^{l-\Delta}+j}^\C \right|}
             {\sqrt{ 2^{l}+p 2^{l-\Delta} + j +1 }}
        \right]
        \sup_{2^l\leq j\leq 2^{l+1}-1}
        |\psi_j(\theta_1)-\psi_j(\theta_2)|\\
 \leq & 2^{-l/2} \left[ \sum_{ 2^l \leq j \leq 2^{l+1}-1 } \left| \Nc_{j}^\C \right| \right]
        \sup_{2^l\leq j\leq 2^{l+1}-1}
        |\psi_j(\theta_1)-\psi_j(\theta_2)|\\
 \leq & 2 \cdot 2^{l/2}
        \sup_{2^l\leq j\leq 2^{l+1}-1}
        |\psi_j(\theta_1)-\psi_j(\theta_2)| \ .
\end{align*}
The second term is treated in a similar manner:
\begin{align*}
      & |Z_{\e_{l+1}}^{(\e_{l}, \Delta)}(\theta_1) -Z_{\e_{l+1}}^{(\e_{l}, \Delta)}(\theta_2) |\\
 \leq & \sum_{p=0}^{2^\Delta-1} \sum_{j=0}^{2^{l-\Delta}-1} 
        \frac{\left| \Nc_{2^l+p 2^{l-\Delta}+j}^\C \right|}
             {\sqrt{ 2^{l}+p 2^{l-\Delta}}}
             |\psi_{2^l + p 2^{l-\Delta}}(\theta_1)-\psi_{2^l + p 2^{l-\Delta}}(\theta_2)+j(\theta_1-\theta_2)| 
        \\
 \leq & 
        2 \cdot 2^{l/2} \left(
       2^{l-\Delta} |\theta_1-\theta_2|  + \sup_{2^l\leq j\leq 2^{l+1}-1}
        |\psi_j(\theta_1)-\psi_j(\theta_2)| \right),
\end{align*}
and maximizing over $\theta_2-\theta_1\in [0,2\pi 2^{-D(l+1)}]$ yields the right bound. As such, from \eqref{eq:crude_bound}, the problem is further reduced to proving, 
for all $c >0$:
$$
\sum_{l \geq r} \P\left( 
       c 2^{l/2} \sup_{2^l\leq j\leq 2^{l+1}-1} \sup_{\theta \in [ 0,2\pi)_{2^{-D(l+1)}} }
       \left( \psi_j(\theta + 2\pi 2^{-D(l+1)})-\psi_j(\theta) \right)
       \geq l^{-2}/2 \right) < \infty.
$$
An application of Proposition \eqref{proposition:prufer_modulus} yields, for $\rho > 1$ depending only on $\beta$:
 \begin{align*}
 &\sum_{l \geq r} \P\left( 
       c 2^{l/2} \sup_{2^l\leq j\leq 2^{l+1}-1} \sup_{\theta \in [ 0,2\pi)_{2^{-D(l+1)}} }
       \left( \psi_j(\theta_1 + 2\pi 2^{-D(l+1)})-\psi_j(\theta) \right)
       \geq l^{-2}/2 \right)\\
 &\leq \sum_{l\geq r}  \sum_{2^l\leq j\leq 2^{l+1}-1} \P\left( \sup_{\theta \in [ 0,2\pi)_{2^{-D(l+1)}} }
    \left(   \psi_j(\theta + 2\pi 2^{-D(l+1)})-\psi_j(\theta) \right) \geq \frac{1}{2c} l^{-2} 2^{-l/2}\right)\\
 &\leq \sum_{l\geq r}  \sum_{2^l\leq j\leq 2^{l+1}-1} 2^{D(l+1)} \P\left(
       \psi_j(2\pi 2^{-D(l+1)}) \geq \frac{1}{2c} l^{-2} 2^{-l/2}\right)\\
 & \ll_{\beta,c} \sum_{l\geq r}  \sum_{2^l\leq j\leq 2^{l+1}-1} 2^{D(l+1)}
         2^{-\rho D(l+1)} l^{2 \rho} 2^{\rho l /2} j^\rho
\ <+\infty,
\end{align*}
for $D$ large enough.
\end{proof}

\subsection{The second moment method} 
\label{subsection:second_moment}

In the sequel, for all fields denoted by $Z$ with some indices and superscripts, we write $R$ with the same indices
and superscripts for the real part of $\sigma$ times the initial field (recall that $\sigma \in \{1,i,-i\}$).
\paragraph{An envelope for the paths of $R_{\e_{N}}^{(\e_r,\Delta)}(\theta) $:} For $j \geq r$, let 
\begin{align}
\label{eq:def_tau_r}   \tau^{(r)}_j&:=  \sum_{l=r}^{j-1} \sum_{p=0}^{2^{\Delta^{(l)} }-1} 
\frac{2^{l-\Delta^{(l)}}}{2^l +p2^{l-\Delta^{(l)}}}=  \sum_{l=r}^{j-1} \sum_{p=0}^{2^{\Delta^{(l)} }-1}
\frac{1}{2^{\Delta^{(l)}} + p} =: (j-r) \log 2 + \lambda_{j}^{(r)},
\end{align}
where $(\lambda_j^{(r)})_{j\geq r}$ is a nonnegative and
increasing sequence, tending, when $j \rightarrow \infty$, to a limit $\lambda_{\infty}^{(r)}$ such that 
$$\lambda_\infty^{(r)} < \sum_{l\geq r} 2^{-\Delta^{(l)}} \leq   2^{-\lfloor \frac{r}{100} \rfloor}
\sum_{l=1}^{\infty} 2^{- 100 \lfloor \log^2 l \rfloor }   \ll  2^{-\lfloor \frac{r}{100} \rfloor} \to 0  $$
when $r$ goes to $\infty$. Fix $\upsilon >0$. Let $\alpha_+:= 1 - \frac{1}{10}$ and $\alpha_-:= \frac{1}{10}$, 
and for $N \geq 2r$, $k \in [|r, N|]$, let us define
\begin{align*}
u_k^{(N)}:=\left\{ \begin{array}{ll} \upsilon- (k-r)^{\alpha_-},\qquad &\text{if  } k\leq \lfloor N/2\rfloor,
\\
\upsilon - (N-k)^{\alpha_-} -   \frac{3}{4}  \log N,\qquad &\text{if  } \lfloor N/2\rfloor <k \leq N,
\end{array} \right. 
\end{align*}
and
\begin{align*}
l_k^{(N)}:=\left\{ \begin{array}{ll} -\upsilon- (k-r)^{\alpha_+},\qquad &\text{if  } k\leq \lfloor N/2\rfloor,
\\
-\upsilon - (N-k)^{\alpha_+} -    \frac{3}{4} \log  N,\qquad &\text{if  } \lfloor N/2\rfloor <k \leq N.
\end{array} \right.
\end{align*}
Note that $l_k^{(N)}$ and $u_k^{(N)}$ implicitly depend on $\upsilon$ and $r$. 
We then define an envelope by its lower bound and its upper bound at each $k\in [|r,N|]$:
\begin{align*}
U_k^{(N)}:=     \tau^{(r)}_k+  u_k^{(N)} \quad \text{and    }\quad L_k^{(N)}:=       \tau^{(r)}_k+ l_k^{(N)}  .
\end{align*}
Now, we will apply the second moment method to the following random variable:
\begin{align}
\I_N:= \sum_{\theta \in [0,2\pi)_{2^N}} \mathds{1}_{\{ \forall k\in [|r,N|],\,  L_k^{(N)}  \leq R_{\e_k}^{(\e_r,\Delta)}(\theta) \leq U_k^{(N)} \}}.
\end{align}
The random walk  $(R_{\e_k}^{(\e_r,\Delta)}(\theta))_{r \leq k \leq N}$ is a Gaussian random walk whose distribution is the same as $\sqrt{\frac{1}{2}} (W_{\tau_j^{(r)}})_{r \leq k \leq N}$. 
In the case where an  event involved in $\I_N$ occurs for some $\theta \in [0, 2\pi)_{2^N}$, it means that  $(R_{\e_k}^{(\e_r,\Delta)}(\theta)$ is around $\tau_k^{(r)}$ for $r \leq k \leq N$, i.e. the Brownian motion $W$ is roughly growing linearly with rate $\sqrt{2}$. For this reason, in the sequel of this part of the paper, we will often estimate the probability of an event $Ev$ concerning the random walk $(R_{\e_k}^{(\e_r,\Delta)}(\theta))_{r \leq k \leq N}$ to a the probability of a similar event $GEv$, where a linear function $t \mapsto t \sqrt{2}$ has been subtracted from the possible trajectories of the underlying Brownian motion $W$ for which the event $Ev$ is satisfied. If $Ev$ depends only on the trajectory of $W$ up to a certain time $T$, we get, by using the Girsanov transfomation, an equality of the form 
$$\P [Ev((W_t)_{0 \leq t \leq T})] = 
  \P [ GEv((W_t - t \sqrt{2})_{0 \leq t \leq T})] 
= \E[ e^{-\sqrt{2}W_T - T} \mathds{1}_{GEv((W_t)_{0 \leq t \leq T})}],$$  
and then the inequality
$$\P [Ev] \leq e^{-T - \sqrt{2} \mu}  
\P [GEv]$$
where $\mu$ denotes the smallest possible value of $W_T$ for which the event $GEv$ can occur. 

The following proposition gives a lower bound for the first moment of $\I_N$. 
\begin{proposition}[First moment of $\I_N$]
\label{Prmomen1Lower} For any $\upsilon\in (0,1)$, $r\in \N$ large enough
and  $N$ large enough depending on $r$: 
\begin{align}
 \label{momen1Lower}
\E(\I_N) &\geq   e^{-2   \upsilon -2\lambda_\infty^{(r)} } 2^r N^{\frac{3}{2} }  \P(Event_{r,N}) \gg_{\upsilon} 2^{r} ,
\end{align}
with $Event_{r,N} := \left\{  \forall j\in [|r,N|],\, l_j^{(N)}  \leq \sqrt{\frac{1}{2}} W_{\tau^{(r)}_j}
\leq u_j^{(N)} \right\}$, $W$ being a standard Brownian motion.
\end{proposition}
\begin{proof}
 Since  $(R_{\e_j}^{(\e_r, \Delta)}(\theta))_{r \leq j \leq N}$ is a Gaussian random walk whose distribution does not depend on $\theta$, we have
\begin{align*}
\E\left( \I_N \right) = 2^N \P\left( \forall j\in [|r,N|],\,  L_j^{(N)}  \leq R_{\e_j}^{(\e_r,\Delta)}(0) \leq U_j^{(N)} \right)
\end{align*}
More precisely, we know that $(R_{\e_j}^{(\e_r,\Delta)}(0))_{j\geq r}$ is distributed like $  \sqrt{\frac{1}{2}}   (W_{    \tau^{(r)}_j})_{j\geq r}  $. By Girsanov's transform, with density $ e^{\sqrt{2}W_{\tau^{(r)}_N }-   \tau^{(r)}_N }  $, we have
\begin{align*}
\E(\I_N)&= 2^N \E\left( e^{ - \sqrt{2} \left( W_{\tau^{(r)}_N} + \sqrt{2} \tau^{(r)}_N \right) + \tau_N^{(r)} } \mathds{1}_{\{ \forall j\in [|r,N|],\, l_j^{(N)} 
\leq \sqrt{\frac{1}{2}}  W_{\tau^{(r)}_j} \leq u_j^{(N)}    \}}      \right)
\\
&= 2^{N} e^{ -\tau^{(r)}_N} \E\left( e^{- \sqrt{2}  W_{\tau^{(r)}_N} }   \mathds{1}_{\{ \forall j\in [|r,N|],\, l_j^{(N)}  \leq  \sqrt{\frac{1}{2}}  W_{\tau^{(r)}_j} \leq u_j^{(N)}    \}}   \right)
\\
&\geq 2^{r} e^{ - 2\upsilon -\lambda_\infty^{(r)} } N^{\frac{3}{2} }  \P\left(  Event_{r,N} \right).
\end{align*}
Now, Propositions \ref{lemma:2sided_barrier_estimate} and \ref{PropProclas1}, applied 
for $d = \sqrt{(\log 2)/2}$, $E_j = (\lambda_{j+r}^{(r)})/\log 2$, and then $||E||$ tending to zero when $r \rightarrow
\infty$, we deduce that $\P \left( Event_{r,N} \right) \gg_{\upsilon} N^{-3/2}$.
Since $\lambda_{\infty}^{(r)}$ is bounded, we get that 
$$2^{r} e^{ - 2\upsilon -\lambda_\infty^{(r)} } N^{\frac{3}{2} }  \P\left(  Event_{r,N} \right)
\gg_{\upsilon} 2^r.$$
\end{proof}

The following proposition gives an upper bound for the second moment of $\I_N$. 
\begin{proposition}[ Second moment of $\I_n$ ]
\label{PrSecondmo}
For $r$ large enough, $\upsilon > 0$ small enough, 
there exists  $\epsilon_{\upsilon,r} \geq 0$
such that 
$$\limsup_{\upsilon \to 0} \limsup_{r\to\infty}
\epsilon_{\upsilon,r} = 0$$
and for $N$ large enough depending on $r$, 
\begin{align}
\label{Secondmo}
\E\left(\I_N^2\right) &\leq (1+\epsilon_{\upsilon,r})\E(\I_N)^2,
\end{align}
when $N$ goes to infinity.
\end{proposition}

Before going into the details, let us show that Proposition \ref{PrSecondmo} implies the lower bound part of our main theorem. 
\begin{proof}[Proof of the lower bound i.e Proposition \ref{proposition:lower_bound} (Assuming Proposition \ref{PrSecondmo})]
By Lemma \ref{lemma:geometric_progressions}, we can assume $n = 2^N$ without loss of generality. Let us fix $r$ and $\upsilon> 0$. If $r$ is large enough and $\upsilon$ small enough, we get, by Paley-Zygmund inequality, for $N$ large enough depending on $r$: 
\begin{align*}
\P\left( \sup_{\theta \in [0,2\pi)_{2^N}} R_{\e_N}^{(\e_r, \Delta )}(\theta) \geq  (N-r)(\log 2) -  \frac{3}{4}  \log N  - \upsilon \right) \geq (1+\epsilon_{\upsilon,r})^{-1} \ .
\end{align*}

From Proposition \ref{proposition:new_coupling}, we have, for $N \geq r$, 
\begin{align*}
  \sup_{\theta \in [0,2\pi)} \Re (\sigma Z_{2^N}(\theta)) & = 
 \sup_{\theta \in [0,2\pi)} R_{\e_N}(\theta)
 \geq \sup_{\theta \in [0,2\pi)} R^{(\e_r)}_{\e_N}(\theta) + \inf_{\theta \in [0,2\pi)} R_{\e_r}(\theta)
\\ &  \geq \sup_{\theta \in [0,2\pi)} R_{\e_N}^{(\e_r,\Delta)}(\theta)+
\inf_{\theta \in [0,2\pi)} R_{\e_r}(\theta)
- X\ ,
\end{align*}
where 
$$X =\sum_{l \geq r} \sup_{\theta \in [0,2\pi)}
|  Z_{\e_{l+1}}^{(\e_{l},\Delta)}(\theta) - Z_{\e_{l+1}}^{(\e_{l})}(\theta) | < \infty.$$
 Moreover:
\begin{align*}
     & \P\left( \sup_{\theta \in [0,2\pi]} R_{\e_N}(\theta) \geq  \log \e_N -  \frac{3}{4}  \log_2 \e_N - C  \right)\\
\geq & \P\left( \sup_{\theta \in [0,2\pi]} R_{\e_N}^{(\e_r,\Delta )}(\theta) \geq  \log \e_N -  \frac{3}{4}  \log_2 \e_N - C + X -  \inf_{\theta \in [0,2\pi]} R_{\e_r }(\theta)  \right)\\
\geq & \P\left( \sup_{\theta \in [0,2\pi)_{2^N}} R_{\e_N}^{(\e_r, \Delta )}(\theta) \geq  \log \e_N -  \frac{3}{4}  \log_2 \e_N - C + X -  \inf_{\theta \in [0,2\pi]} R_{\e_r }(\theta)  \right)\\
\geq & \P\left( \sup_{\theta \in [0,2\pi)_{2^N}} R_{\e_N}^{(\e_r, \Delta )}(\theta) \geq  \log \e_N -  \frac{3}{4}  \log N  - r \log 2 - \upsilon \right)\\ 
 -  & \P\left( r \log 2 + \upsilon - \frac{3}{4} \log_2 2 \geq  C - X + \inf_{\theta \in [0,2\pi]} R_{\e_r }(\theta)  \right)\\
 =   & (1+\epsilon_{\upsilon,r})^{-1} - \P\left( r \log 2 + \upsilon - \frac{3}{4} \log_2 2 \geq  C  - X+ \inf_{\theta \in [0,2\pi]} R_{\e_r }(\theta)  \right).
\end{align*}
By first letting $N \rightarrow \infty$ and then $C \rightarrow \infty$, we have
$$ \liminf_{C \rightarrow \infty} \liminf_{N \rightarrow \infty}
   \P\left( \sup_{\theta \in [0,2\pi]} R_{\e_N}(\theta) \geq  \log \e_N -  \frac{3}{4}  \log_2 \e_N - C  \right)
   \geq (1 +  \epsilon_{\upsilon,r})^{-1} \ .
$$
Letting $r \rightarrow \infty$ and then $\upsilon \rightarrow 0$ concludes the proof of Proposition \ref{proposition:lower_bound}.
\end{proof}

\subsection{Proof of Proposition \ref{PrSecondmo}}

We have:
\begin{align*}
\E\left( \I_N^2\right) &=\E\left(  \sum_{\theta,\, \theta' \in [0,2\pi)_{2^N}}  \mathds{1}_{\{ \forall j\in [|r,N|],\,  L_j^{(N)}  \leq R_{\e_j}^{(\e_r,\Delta )}(\theta), R_{\e_j}^{(\e_r,\Delta)}(\theta') \leq U_j^{(N)} \}}  \right).
\end{align*}
Hence, the proof of Proposition \ref{PrSecondmo} is intimately related to studying for $\theta, \theta' \in [0,2\pi)_{2^N}$ the function:
\begin{align}
\label{eq:def_P_N}
\P_N(\theta, \theta'):= \P\left( \forall j\in [|r,N|],\,  L_j^{(N)}  \leq R_{\e_j}^{(\e_r,\Delta)}(\theta), R_{\e_j}^{(\e_r,\Delta)}(\theta') \leq U_j^{(N)} \right) \ .
\end{align}
This study, which is technical, is based on the fact that $(R_{\e_j}^{(\e_r,\Delta)}(\theta))_{r \leq j \leq N}$ and $(R_{\e_j}^{(\e_r,\Delta)}(\theta'))_{r \leq j \leq N}$ are two Gaussian random walks whose increments are approximately independent after some branching time which is roughly the logarithm in base $2$ of the distance modulo $2 \pi$ between $\theta$ and $\theta'$.

The general idea is as follows. For given $\theta, \theta'$, we will consider the integer $\k$ such that $2^{-\k} \leq \frac{||\theta- \theta'||}{2\pi} < 2^{-(\k-1)}$, $|| \cdot ||$ denoting the distance on the set $\R / 2 \pi \Z$. One can understand $\k$ as the time of (approximate) branching between the field at $\theta $ and at $\theta' $. We will show that after some time $\k_+ = \k_+(\k) $ ``slightly larger'' than $\k$, we are able to bring out independence between the increments of  $(R_{\e_k}^{(\e_r,\Delta)}(\theta))_{k\geq r} $ and $ (R_{\e_k}^{(\e_r,\Delta)}(\theta'))_{k\geq r}$. By analogy with the Gaussian field, we will see this time as a time of decorrelation. It is defined as follows: 
 \begin{itemize} 
 \item For $\k  \leq  N/2$, the time of decorrelation is $\k_+ := \k + 3 \Delta^{(\k)}$. Recall that:
 $$ \Delta^{(\k)}:= \lfloor \frac{r}{100} \rfloor  + 100\lfloor  \log^2 \k \rfloor \ .$$
 In particular, for $\k \leq r/2$ and $r$ large enough, the fields $(R_{\e_k}^{(\e_r,\Delta)}(\theta))_{k\geq r} $ and $ (R_{\e_k}^{(\e_r,\Delta)}(\theta'))_{k\geq r}$ will have ``almost independent increments'' from the starting time $r$, since $r \leq \k_+$. 
 \item For $N/2 < \k \leq N - (r/2)$, we will require a faster decorrelation . We take $\k_+
 := \k + 3 \kappa^{(\k)}$, where 
 $$\kappa^{(\k)} :=  \lfloor r/100 \rfloor + 100 \lfloor \log (N-\k) \rfloor^2.$$
 However, the price to pay is that we will have to modify our field $(R_{\e_j}^{(\e_r,\Delta)}(\theta))_{r \leq j \leq N}$ in the spirit of subsection \ref{section:new_coupling}.
 \item For $\k > N - (r/2)$, the branching time is close to the end and then we do not need to use any time of decorrelation. 
 \end{itemize}
 The main part of the proof of Proposition
 \ref{PrSecondmo} consists in four lemmas, numbered \ref{lemma:kDebut}, \ref{lemma:ktoutDebut}, \ref{lemma:kMil} and \ref{lemma:kFin}. The statement of each of these lemmas gives a suitable majorization of $\P_N(\theta, \theta')$, for a given range of values of $\k$.

\begin{lemma}[Time of branching $\k \leq N/2$]
  \label{lemma:kDebut}
  For any $\upsilon >0$, $r$ large enough, $N$ large enough depending on $r$,   $\k \leq \frac{N}{2}$ and $2^{-\k}\leq \frac{||\theta-\theta'||}{2\pi} < 2^{-(\k-1)}$,  we have
  \begin{align}
  \label{eqkDebut}
  \P_N(\theta, \theta')&\ll_{\upsilon} \left\{ 
  \begin{array}{ll}
    2^{2r-2N}                                    ,\qquad &\text{when } \k_+ \leq r,\\
    2^{r-N} 2^{\k_+-N} e^{- (\k_+-r)^{\alpha_-}} ,\qquad &\text{when } \k_+ \geq r.
  \end{array}
  \right. 
  \end{align} 
\end{lemma}

The main contribution of $\E(\I_n^2)$ comes from the terms whose the time of branching happen before $r/2$. The following Lemma studies this case. It refines the estimate obtained in the previous Lemma, which is not sufficient for our purposes.  
\begin{lemma}[Time of branching $\k \leq r/2$]
  \label{lemma:ktoutDebut}
  For any $\upsilon \in (0,1)$, $r\in \N$ large enough, $ \k \leq \frac{r}{2}$, $2^{-\k} \leq \frac{||\theta-\theta'||}{2\pi} \leq 2^{-(\k-1)}$ and  $N$ large enough depending on $r$, we have
  \begin{align*}
  \P_N(\theta, \theta')&\leq   (1+\eta_{r,\upsilon  })    2^{2r -2N} N^{3}   (\P(Event_{r,N}))^2,
  \end{align*}
  where 
  $$\underset{\upsilon \rightarrow 0}{\limsup} \, \underset{r \rightarrow \infty} {\limsup} \, \eta_{r,\upsilon} = 0.$$
\end{lemma}

\begin{lemma}[Time of branching $N/2 < \k \leq N - (r/2)$]
  \label{lemma:kMil}
  For $\upsilon \in (0,1)$, $r\in \N$ large enough, $N$ large enough depending on $r$,  $N/2 < \k \leq N-\frac{r}{2} $ and $2^{-\k}\leq \frac{||\theta-  \theta'||}{2\pi}  < 2^{-(\k-1)}$,  we have $\k_+ < N$ and 
  \begin{align}
  \P_N(\theta,\theta') \ll_{\upsilon}  2^{r-N} 2^{\k_+-N} e^{- \half (N-\k_+ )^{\alpha_-}} , 
  \end{align}
\end{lemma}
\begin{rmk}
The proof of the Lemma \ref{lemma:kMil} is the unique place where we use $(l_j^{(N)})_{r \leq j\leq N}$, the lower part of the envelope.  
\end{rmk}

\begin{lemma}[Time of branching $\k > N - (r/2)$]
    \label{lemma:kFin}
    For $\upsilon \in (0,1)$, $r\in \N$ large enough, $N$ large enough depending on $r$,  $N-(r/2) <  \k $ and  $2^{-k}\leq  \frac{||\theta-\theta'||}{2\pi}  \leq 2^{-(\k-1)}$, we have
    \begin{align}
	    \P_N(\theta,\theta') \ll_{\upsilon} 2^{r-N}.
    \end{align}
\end{lemma}

\subsubsection{Dyadic case}
We start by assuming $\frac{||\theta-\theta'||}{2\pi} = 2^{-\k}$ and prove Lemmas \ref{lemma:kDebut} and \ref{lemma:kMil} in this dyadic case. This part is mainly for pedagogical purposes while laying the ground for the general case. It illustrates perfectly the machinery of the proof in a simpler setting.

It will be convenient to denote,  for any $l\in [|r,N-1|]$, $p\in [|0,2^{\Delta^{(l)}}-1|]$,
\begin{align}
 \label{defI_n}
 I_{l,p}(\theta) := \frac{ \sum_{j=0}^{2^{l-\Delta^{(l)}}-1} \Nc^\C_{2^{l}+p 2^{l-\Delta^{(l)}} + j} e^{i \theta j} }
                         { \sqrt{{ 2^{l}+p 2^{l-\Delta^{(l)}}  }}  }
                 \eqlaw \Nc^{\C}\left(0, (2^{\Delta^{(l)}}+ p)^{-1} \right).
\end{align} 
Recall that $R_{\e_N}^{(\e_r,\Delta)}(\theta)$ and $R_{\e_N}^{(\e_r, \Delta)}(\theta')$ can be written as
$$
R_{\e_N}^{(\e_r,\Delta)}(\theta)= \Re \left( \sigma \sum_{l=r}^{N -1 } \sum_{p=0}^{2^{\Delta^{(l)}}-1}  I_{l,p}(\theta) e^{ i\psi_{2^{l}+p2^{l-\Delta^{(l)}}} (\theta) } \right)
$$
$$
R_{\e_N}^{(\e_r,\Delta)}(\theta')= \Re \left( \sigma \sum_{l=r}^{N -1 } \sum_{p=0}^{2^{\Delta^{(l)}}-1}  I_{l,p}(\theta') e^{ i\psi_{2^{l}+p2^{l-\Delta^{(l)}}} (\theta') } \right).
$$
{\it The crucial observation} is that for any $  l \in [|\k_+,N-1|] $  (we easily check that $\k_+ < N-1$ if $r$ is large enough, $N \geq r$ and $\k \leq N/2$) and any $p\leq 2^{\Delta^{(l)}}-1$, the random variables $ I_{l,p}(\theta)$ and $ I_{l,p}(\theta')$  are independent and identically distributed. Indeed, they form a complex Gaussian vector, and they are uncorrelated, since for $l \geq \k_+$, one has $l - \Delta^{(l)} \geq \k$ if $l \geq r $ and $r$ is large enough,   
and then 
$$ \sum_{j=0}^{2^{l - \Delta^{(l)}} } e^{i( \theta - \theta') j}
 = \sum_{j=0}^{2^{l - \Delta^{(l)}}} e^{\pm 2  i j \pi / 2^{\k}} = 0.$$

We deduce that the increments of $R_{\e_{j}}^{(\e_r,\Delta)}(\theta) $ and  $R_{\e_{j}}^{(\e_r,\Delta)}(\theta') $ after the time $\k_+$ are independent and identically distributed. Recalling the definition of $\tau^{(r)}$ in (\ref{eq:def_tau_r}) and \eqref{defI_n}, it follows that
\begin{align*}
(R_{\e_{j}}^{(\e_r,\Delta)}(\theta))_{j\geq r} \eqlaw (R_{\e_{j}}^{(\e_r,\Delta)}(\theta'))_{j\geq r} \eqlaw \sqrt{\frac{1}{2}} ( W_{\tau_j^{(r)}})_{j\geq r}.
\end{align*}
For any $k \in [|r, N|]$, $z\in \R$, we introduce the events
\begin{align}
\nonumber
Ev(k,z)&:= \left\{   \forall j\in [|k, N |] ,\, l_j^{(N)} \leq \sqrt{\frac{1}{2}}W_{\tau_j^{(k)}}-  \tau^{(k)}_j  +z\leq u_j^{(N)}   \right\},
\\
\label{eq:def_Ev_GEv}
GEv(k,z)&:= \left\{   \forall j\in [|k, N |] ,\, l_j^{(N)} \leq \sqrt{\frac{1}{2}}W_{\tau_j^{(k)}} +z \leq u_j^{(N)}    \right\}.
\end{align}
Notice that $GEv(r,0)=Event_{r,N}$ from Proposition \ref{Prmomen1Lower}. Furthermore notice that $GEv(k,z)$ is equal to the event obtained from $Ev(k,z)$ after the Girsanov transform with density $ \exp\left( \sqrt{2}W_{\tau^{(k)}_N }- \tau^{(k)}_N \right) $. Performing the transform yields:
\begin{align}
\label{eq:girsanov1}
\P\left(Ev(k,z)\right) & = e^{-\tau_N^{(k)}} \E\left( e^{-\sqrt{2} W_{\tau^{(k)}_N }  } \mathds{1}_{ GEv(k,z) } \right)
                         \leq 2^{k-N} e^{2(z+\upsilon)} N^{\frac{3}{2}} \P\left( GEv(k,z) \right),
\end{align}
where in the last inequality we used the definition \eqref{eq:def_tau_r} of $\tau_N^{(k)}$ and the fact that $e^{-\sqrt{2} W_{\tau^{(k)}_N }  }\leq N^{\frac{3}{2}} e^{2(z+\upsilon)} $ on $GEv(k,z)$.  

In order to allow for more flexibility and for later use, let us record the following analogous events. For $k \in [|r,N|]$, $a\geq 0$, $E=(E_j)_{j\geq k}$ a sequence of reals such that $(\tau_{j}^{(k)} - E_j)_{j\geq k}$ is positive and nondecreasing, $z\in \R$, define
\begin{align}
\nonumber
Ev(k,a,E,z)  &:= \left\{   \forall j\in [|k, N |] ,\, l_j^{(N)}-a \leq \sqrt{\frac{1}{2}}W_{\tau_j^{(k)}-E_j}+z- \tau^{(k)}_{j}  \leq u_j^{(N)}+a    \right\},\\
\label{eq:def_Ev_GEv_ext}
GEv(k,a,E,z) &:= \left\{   \forall j\in [|k, N |] ,\, l_j^{(N)}-a \leq \sqrt{\frac{1}{2}}W_{\tau_j^{(k)}-E_j} +z\leq u_j^{(N)}+a    \right\}.
\end{align}
Again, the event $GEv(k,a,E,z) $ is, up to an error due to the time shift $E$, "quasi equal" to what we obtain when we apply the Girsanov' transform with density $\exp\left( \sqrt{2}W_{\tau_N^{(k)}-E_N } - (\tau_N^{(k)}- E_N) \right) $ to the event $ Ev(k,a,E,z)$. This time, the inequality takes the form:
\begin{align}
\label{eq:girsanov2}
\P\left(Ev(k,a,E,z)\right) & \leq 2^{k-N} e^{-E_N+2(z+a+\upsilon)} N^{\frac{3}{2}} \P\left( GEv(k,a+\sup_{k \leq j \leq N} |E_j|,E,z) \right).
\end{align}
Indeed, by the Girsanov transform and then using the barrier at time $N$:
\begin{align*}
  & \P\left(Ev(k,a,E,z)\right) \\
= & e^{-\tau_N^{(k)}+E_N} \E\left( e^{-\sqrt{2} W_{\tau^{(k)}_N - E_N}  } \mathds{1}_{ \left\{   \forall j\in [|k, N |] ,\, l_j^{(N)}-a \leq \sqrt{\frac{1}{2}}W_{\tau_j^{(k)}-E_j}+z- E_j  \leq u_j^{(N)}+a  \right\} } \right)\\
\leq & 2^{k-N} e^{E_N+2(z-E_N+a+\upsilon)} N^{\frac{3}{2}} \P\left( \forall j\in [|k, N |] ,\, l_j^{(N)}-a \leq \sqrt{\frac{1}{2}}W_{\tau_j^{(k)}-E_j}+z- E_j  \leq u_j^{(N)}+a \right) \\
\leq & 2^{k-N} e^{-E_N+2(z+a+\upsilon)} N^{\frac{3}{2}} \P\left( GEv(k,a+\sup_{k \leq j \leq N} |E_j|,E,z) \right).
\end{align*}

\begin{proof}[Proof of Lemma \ref{lemma:kDebut} in dyadic case]{ \quad }

{ \bf When $\k_+ \leq r$:}
The increments of $R_{\e_{j}}^{(\e_r,\Delta)}(\theta) $ and  $R_{\e_{j}}^{(\e_r,\Delta)}(\theta') $ after the time $\k_+$ are independent and identically distributed, thus we have
\begin{align*}
\P_N(\theta,\theta') = \quad & \P\left(Ev(r,0)\right)^2\\
                     \stackrel{Eq. \eqref{eq:girsanov1}}{\ll_v} & 2^{2r-2N} \left( N^{\frac{3}{2}} \P\left( GEv(r,0) \right) \right)^2 \ .
\end{align*}
Finally by applying \eqref{class7} (with $E_{j-r}=\lambda_j^{(r)}/\log 2$, which implies $||E|| \leq 1$ for $r$ large enough), we get for $N$ large enough depending  on $r$: 
\begin{align*}
\P_N(\theta, \theta')\ll_{\upsilon} 2^{2r-2N} .
\end{align*}
It concludes the study when $\k_+\leq r$. 

{ \bf When $\k_+ \geq r$:}
The increments of $R_{\e_{j}}^{(\e_r,\Delta)}(\theta) $ and  $R_{\e_{j}}^{(\e_r,\Delta)}(\theta') $ after time $\k_+$ are independent and identically distributed. Moreover, 
all these increments are independent of 
those of $R_{\e_{j}}^{(\e_r,\Delta)}(\theta) $
for $j$ between $r$ and $\k_+$ (we see this fact  by first conditioning with respect to the $\sigma$-algebra $\mathcal{G}_{2^{\k_+}}$). We then have:
\begin{align*}
\P_N(\theta, \theta') \leq  \quad & \P\left(  Ev(r,0) \right) \sup_{ l_{\k_+}^{(N)}  \leq z\leq u_{\k_+}^{(N)} }\P( Ev(\k_+,z ) )\\
\stackrel{Eq. \eqref{eq:girsanov1}}{\ll_\upsilon}
                                  & 2^{r-N} \left( N^{\frac{3}{2}} \P\left( GEv(r,0) \right) \right)
                                    2^{\k_+-N} \sup_{ l_{\k_+}^{(N)}  \leq z\leq u_{\k_+}^{(N)} } \left( N^{\frac{3}{2}} e^{2z} \P\left( GEv(\k_+,z) \right) \right) \\
                      \ll   \quad & 2^{r-N}
                                    2^{\k_+-N} \sup_{ l_{\k_+}^{(N)}  \leq z\leq u_{\k_+}^{(N)} } \left( N^{\frac{3}{2}} e^{2z} \P\left( GEv(\k_+,z) \right) \right).
\end{align*}
If $\k_+ \leq N/4$, according to \eqref{class7}, for any $z\in \R$, we have
\begin{align*}
\P\left( GEv(\k_+,z ) \right) \ll_{\upsilon}  ( 1+ ( -z)_+) (N-\k_+)^{-\frac{3}{2}} \ .
\end{align*}
Recalling that $ z\leq u_{\k_+}^{(N)}= \upsilon -(\k_+-r)^{\alpha_-} $, we have
\begin{align}
\label{tocliam}
\sup_{ l_{\k_+}^{(N)}  \leq z\leq u_{\k_+}^{(N)} } \left( N^{\frac{3}{2}} e^{2z} \P\left( GEv(\k_+,z) \right) \right)
\ll_\upsilon e^{- (\k_+-r)^{\alpha_-}}.
\end{align}
One the other hand, for $\k_+ > N/4$, we can crudely bound $\P\left( GEv(\k_+,z ) \right)$ by $1$, and using the fact that $z  \leq \upsilon - 0.9(\k_+-r)^{\alpha_-}$, the factor $0.9$ being used in order to handle the case where $\k \leq N/2 \leq \k_+$, which implies, for $r$ large enough and $N$ large enough depending on $r$,  
$$\k_+ \leq (N/2) + 3(r/100) + 300 \log^2 (N/2) \leq  0.51 N,$$
and then 
$$N - \k_+ \geq 0.49 N \geq 49 \k_+ / 51 \geq (0.9)^{1/\alpha_-} (\k_+ - r).$$
Hence, we get
\begin{align}
\label{tocliam1}
\sup_{ l_{\k_+}^{(N)}  \leq z\leq u_{\k_+}^{(N)} } \left( N^{\frac{3}{2}} e^{2z} \P\left( GEv(\k_+,z) \right) \right)
\leq e^{2 z} N^{\frac{3}{2}} 
\ll_\upsilon e^{- 1.8 (\k_+-r)^{\alpha_-}} N^{3/2}
\ll e^{- (\k_+-r)^{\alpha_-}},
\end{align}
since $\k_+-r \geq N/5$ for $N$ large enough depending on $r$. 

In all cases, by combining Eq. \eqref{tocliam} and \eqref{tocliam1}, we deduce:
\begin{align*}
\P_N(\theta,\theta') \ll_{\upsilon}   2^{r-N} 2^{\k_+-N} e^{- (\k_+-r)^{\alpha_-}} .
\end{align*}
It concludes the proof of Lemma \ref{lemma:kDebut} when $\frac{||\theta- \theta'||}{2\pi}$ is a negative power of $2$.
\end{proof}

\begin{proof}[Proof of Lemma \ref{lemma:kMil} in the dyadic case]
 Now we shall study $\P_N(\theta,\theta')$, when the branching between the field in $\theta$ and the field in $\theta'$  appears after the time $N/2$. This time we shall prove that when one restricts to the paths which are in the envelope, the increments of the path of the field at $\theta$ and at $\theta'$ are approximately independent after the time of decorrelation $k_+$.  We recall that for this range
 $\k_+ = \k+ 3\kappa^{(\k)}$ where $\kappa^{(\k)}:= \lfloor \frac{r }{100}\rfloor + 100 \lfloor \log (N-\k)\rfloor^2$, if 
$\k \leq N-1$. 

We need to exhibit the independence between the increments of $(R_{\e_j}^{(\e_{\k_+},\Delta)}(\theta))_{j\geq \k_+}$ and $(R_{\e_j}^{(\e_{\k_+},\Delta)}(\theta'))_{j\geq \k_+}$. The {\it crucial observation} we used in case of $\k \leq N/2$ does not work anymore for such a short decorrelation time. We first need to modify our field using similar arguments to those used for the proof of Proposition \ref{proposition:new_coupling}. 

In the following we shall use the quantity
\begin{align*}
	J_{l,p}^{(\k)}(\theta) := \frac{ \sum_{j=0}^{2^{l-\kappa^{(\k)}}-1}    \Nc^\C_{2^{l}+p 2^{l-\kappa^{(\k)}} + j} e^{i \theta j}     }{\sqrt{{ 2^{l}+p 2^{l-\kappa^{(\k)}}  }}  } \eqlaw \Nc^{\C}\left((0, 2^{\kappa^{(\k)}} +p)^{-1} \right)
\end{align*} 

Let $\k \in [N/2,N-r/2]$ and $2^{-\k}\leq \frac{||\theta- \theta'||}{2\pi} < 2^{-(\k-1)}$. We have, for $r$ large enough, since $N - \k \geq r/2$ is large,  
$$\k_+ \leq  \k + (r/33) + 300 \log^2 (N-\k) \leq \k + (N- \k)/16 < N. $$
Recall that for $ \tilde{\theta}\in \{\theta,\theta'  \}$, $r\leq k\leq N$, $ R_{\e_k}^{(\e_{r},\Delta )}( \tilde{\theta})   =  \Re(\sigma Z_{\e_k}^{(\e_{r},\Delta )}( \tilde{\theta}))   $.  Since for $l \geq \k_+$, $\kappa^{(\k)} \leq \Delta^{(l)}$, we can write for all $\k_+ \leq k \leq N$, $ \tilde{\theta}\in \{  \theta,\theta'\}$: 
\begin{align*}
	Z_{\e_k}^{(\e_{\k_+},\Delta )}( \tilde{\theta}) &=    \sum_{l=\k_+}^{k-1 } \sum_{p=0}^{2^{\Delta^{(l)}}-1}  I_{l,p}( \tilde{\theta}) e^{ i\psi_{ 2^{l}+p 2^{l-\Delta^{(l)}}} ( \tilde{\theta} )  }\\
	&=    \sum_{l=\k_+}^{k -1 } \sum_{p=0}^{2^{\Delta^{(l)}}-1}  \frac{ \sum_{j=0}^{2^{l-\Delta^{(l)}}-1}    \Nc^\C_{2^{l}+p 2^{l-\Delta^{(l)}} + j} e^{i   \tilde{\theta} j}     }{\sqrt{{ 2^{l}+p 2^{l-\Delta^{(l)}}  }}  }e^{ i\psi_{ 2^{l}+p 2^{l-\Delta^{(l)}}}(   \tilde{\theta}) }
	\\
	=  &  \sum_{l=\k_+}^{k-1 } \sum_{p=0}^{2^{\kappa^{(\k)}}-1}     \left(\sum_{j=0}^{2^{l-\kappa^{(\k)}}-1}        \frac{   \Nc^\C_{2^{l}+p 2^{l-\kappa^{(\k)}} + j} e^{i  \tilde{\theta} \left\{  j-      \mathfrak{j}       \right\}} }{\sqrt{  2^{l}+p 2^{l-\kappa^{(\k)}} + \mathfrak{j}  }  } e^{ i  \psi_{ 2^{l}+p 2^{l-\kappa^{(\k)} } +   \mathfrak{j}   }(  \tilde{\theta})}  \right),\, \text{(with $\mathfrak{j}=   \lfloor \frac{j}{2^{l-\Delta^{(l)}}}\rfloor    2^{l-\Delta^{(l)}}  $)}
	\\
	&=  Z_{\e_k}^{(\e_{\k_+},\kappa^{(\k)})}(  \tilde{\theta})  + \mathfrak{E}_{\e_k}( \tilde{\theta})
\end{align*}
where
\begin{align*}
	Z_{\e_k}^{(\e_{\k_+},\kappa^{(\k)})}(   \tilde{\theta}) &:= \sum_{l=\k_+}^{k-1} \sum_{p=0}^{2^{\kappa^{(\k)}}-1}  J_{l,p}^{(\k)}(  \tilde{\theta}) e^{ i  \psi_{ 2^{l}+p 2^{l-\kappa^{(\k)}}  }(  \tilde{\theta})}  ,
	\\
	\mathfrak{E}_{\e_k}(  \tilde{\theta})&:=   \sum_{l=\k_+}^{k-1} \sum_{p=0}^{2^{\kappa^{(\k)}}-1} \frac{  e^{ i  \psi_{ 2^{l}+p 2^{l-\kappa^{(\k)}}  }( \tilde{\theta} )}  }{\sqrt{{ 2^{l}+p 2^{l-\kappa^{(\k)}}  }}  } \left(\sum_{j=0}^{2^{l-\kappa^{(\k)}}-1}    \Nc^\C_{2^{l}+p 2^{l-\kappa^{(\k)}} + j} e^{i  \tilde{\theta} j}   \lozenge_j^{ (2^{l}+p 2^{l-\kappa^{(\k)}})  }(  \tilde{\theta}) \right),
\end{align*}
with 
\begin{align}
	\label{eq:ineq_squarebis}
	\left|\lozenge_j^{(2^{l}+p 2^{l-\kappa^{(\k)}} )}(  \tilde{\theta})\right|
	\leq &  \left|  A_{ 2^{l}+p 2^{l-\kappa^{(\k)}} +  \mathfrak{j} }^{(2^{l}+p 2^{l-\kappa^{(\k)}})    }(  \tilde{\theta})   ) \right| + 2^{-\kappa^{(\k)}}.
\end{align}
Indeed:
\begin{align*}
	\lozenge_j^{  (2^{l}+p 2^{l-\kappa^{(\k)}} )  }(  \tilde{\theta})
	= & -1+ \sqrt{\frac{2^l+p2^{l-\kappa^{(\k)}}}{2^l+p2^{l-\kappa^{(\k)}} +  \mathfrak{j}   }  }e^{i A_{ 2^{l}+p 2^{l-\kappa^{(\k)}} +\mathfrak{j}  }^{(2^{l}+p 2^{l-\kappa^{(\k)}})}(  \tilde{\theta})  }
	\\
	= & - 1+ \left( 1 + \Theta \right)
	e^{i A_{2^{l}+p 2^{l-\kappa^{(\k)}} +  \mathfrak{j}}^{(2^{l}+p 2^{l-\kappa^{(\k)}})    }(  \tilde{\theta})  },
\end{align*}
with $|\Theta| \leq 2^{-\kappa^{(\k)}}$
which implies inequality \eqref{eq:ineq_squarebis}. In the following, we shall denote:
\begin{align*}
	\blacklozenge^{(\k)}_l(  \tilde{\theta}):=\sum_{p=0}^{2^{\kappa^{(\k)}}-1} \frac{  e^{ i  \psi_{ 2^{l}+p 2^{l-\kappa^{(\k)}}  }(  \tilde{\theta} )}  }{\sqrt{{ 2^{l}+p 2^{l-\kappa^{(\k)}}  }}  } \left(\sum_{j=0}^{2^{l-\kappa^{(\k)}}-1}    \Nc^\C_{2^{l}+p 2^{l-\kappa^{(\k)}} + j} e^{i   \tilde{\theta} j}   \lozenge_j^{ (2^{l}+p 2^{l-\kappa^{(\k)}})  }(   \tilde{\theta}) \right).
\end{align*}

Notice that, on the contrary of the proof of Proposition \ref{proposition:new_coupling}, where $\Delta^{(l)}$ varies with $l$, here we fix $\kappa^{(\k)}$ as soon as we know that  $2^{-\k} \leq \frac{||\theta- \theta'||}{2\pi}  < 2^{-(\k-1)} $.  By using the same arguments used to prove  \eqref{eq:one_point_square} and (\ref{eq:laplace_blacksquare}), one can show similarly that for any $l\geq \k_+$, $ \tilde{\theta}\in \{\theta, \theta'\}$,
\begin{align}
\label{nonind}
	& \P\left( |  \blacklozenge^{(\k)}_l( \tilde{\theta}) |\geq 2^{-\frac{\kappa^{(\k)}}{8}},  \cap_{l \geq \k_+} \widetilde{G}_{l}( \tilde{\theta})  \Big|  \Gc_{2^{\k_+}}  \right)  \ll e^{-2^{\frac{\kappa^{(\k)}}{8}   }}, 
	\\ & \qquad  \P\left(( \cap_{l\geq \k_+} \widetilde{G}_{l}(  \tilde{\theta} ))^c  \Big|    \Gc_{2^{\k_+}} \right)  \ll_{\beta}      \exp\left( - \frac{\beta}{33} 2^{\half \kappa^{(\k)}}    \right).
\end{align}
with
\begin{align*}
	\widetilde{G}_{l}(  \tilde{\theta} ) := & \ \bigcap_{p=0}^{2^{\kappa^{(\k)}}-1}
	\left\{ \sup_{0\leq j\leq 2^{l-{\kappa^{(\k)}}}-1} \left| A_{j+ 2^{l}+p 2^{l- \kappa^{(\k)}} }^{(2^{l}+p 2^{l-  \kappa^{(\k)}})}(  \tilde{\theta} ) \right|
	\leq 2^{-\quart \kappa^{(\k)} }
	\right\} \ .
\end{align*}
Moreover it is plain to observe that for any $r$ large enough, $N$ large enough depending on $r$, and $\k\in [|N/2, N-r/2|]$, and under the complement of the two events just above, 
\begin{align}
	\sum_{l=\k_+}^{N-1}  |  \blacklozenge^{(\k)}_l(  \tilde{\theta}) |  \leq   \sum_{l=\k_+}^{N-1} 2^{- \frac{\kappa^{(\k)}}{8}} \leq (N-\k_+) 2^{- \frac{1}{8} \lfloor \frac{r}{100} \rfloor -  \frac{25}{2}\lfloor\log (N-\k)\rfloor^2  } \leq 1.
\end{align}
So for any $\tilde{\theta}\in \{\theta,\theta'\}$,  we can replace $(R_{\e_k}^{(\e_{\k_+},\Delta )}(  \tilde{\theta}))_{\k_+ \leq k \leq N} $ by $(R_{\e_k}^{(\e_{\k_+},\kappa^{(\k)})}(  \tilde{\theta}))_{ \k_+ \leq k \leq N}  $ with an error at most $1$.
 Thus we have 
\begin{align*}
 &  \P_N(\theta,\theta') \leq   \P\left( \forall j\in [|r,N|],\,  L_j^{(N)} -1 \leq   R_{\e_{ j\wedge \k_+}}^{(\e_{r},\Delta )}(\theta') +  R_{\e_j}^{(\e_{\k_+},\kappa^{(\k)})}(\theta')\mathds{1}_{{\{ j >k_0  \}}} \leq U_j^{(N)} +1,\, \right.
 \\
 & \left. \qquad\qquad\qquad\qquad\qquad\qquad \forall j\in [|\k_+,N|],\,  L_j^{(N)} -1 \leq    R_{\e_{ \k_+}}^{(\e_{r},\Delta )}(\theta) +  R_{\e_j}^{(\e_{\k_+},\kappa^{(\k)})}(\theta) \leq U_j^{(N)} +1 \right)\\
 + & \sum_{\tilde{\theta}\in \{\theta,\theta'\}} \E\Bigg(  \mathds{1}_{\{  \forall j\in [|r,\k_+|],\,  L_j^{(N)}  \leq R_{\e_j}^{(\e_r,\Delta)}( \tilde{\theta}) \leq U_j^{(N)}\}}  \Big( \sum_{l=\k_+}^{N-1} \P\left( |  \blacklozenge^{(\k)}_l(  \tilde{\theta}  ) |\geq 2^{-\frac{\kappa^{(\k)}}{8}},  \cap_{l\geq \k_+} \tilde{G}_{l}( \tilde{\theta} ) \Big| \Gc_{2^{\k_+}}  \right)\\
 & \qquad \qquad\qquad\qquad\qquad \qquad\qquad\qquad \qquad\qquad\qquad \qquad\qquad\qquad+   \P\left( ( \cap_{l\geq \k_+} \tilde{G}_{l}( \tilde{\theta} ))^c \Big|  \Gc_{2^{\k_+}}  \right) \Big) \Bigg)&
\end{align*}
We first deal with the sum in $\tilde{\theta}\in \{  \theta,\theta'\}$. By using \eqref{nonind}, then
the Girsanov transfom with density $e^{\sqrt{2 }W_{\tau^{(r)}_{\k_+} }-    \tau_{\k_+}^{(r)} }$, and Corollary  
\ref{corollary:shifted_barrier_estimate} (when $\k_+\geq \frac{2N}{3}$), the sum is
\begin{align*}
 \ll_{\upsilon} 2^{r-\k_+} N^{\frac{3}{2}} e^{ 2 (N-\k_+)^{\alpha_+}}  \left( (N-\k_+)e^{-2^{\frac{\kappa^{(\k)}}{8}}}   +    \exp\left( - \frac{\beta}{33} 2^{\half \kappa^{(\k)}} \right)  \right),\;\text{when }\, \k_+ \leq \frac{2N}{3},
\\
 \ll_{\upsilon} 2^{r-\k_+} \frac{N^{\frac{3}{2}}( N-\k_+)^{2\alpha_+}   }{(\k_+-r)^{\frac{3}{2}}} e^{ 2 (N-\k_+)^{\alpha_+}}  \left( (N-\k_+)e^{-2^{\frac{\kappa^{(\k)}}{8}}}   +    \exp\left( - \frac{\beta}{33} 2^{\half \kappa^{(\k)}} \right)  \right),\; \text{when }\, \k_+ \geq \frac{2N}{3}
\end{align*}
which are both dominated by $2^{-2N+ \k_+ +r}      e^{- \half (N-\k_+ )^{\alpha_-}} $. It remains to bound 
\begin{align*}
\P_N^{(\Delta,\kappa)}(\theta,\theta'):=\P\left( \forall j\in [|r,N|],\,  L_j^{(N)} -1 \leq   R_{\e_{ j\wedge \k_+}}^{(\e_{r},\Delta )}(\theta') +  R_{\e_j}^{(\e_{\k_+},\kappa^{(\k)})}(\theta')\mathds{1}_{{\{ j >k_0  \}}} \leq U_j^{(N)} +1,\, \right.
\\
 \left.  \forall j\in [|\k_+,N|],\,  L_j^{(N)} -1 \leq    R_{\e_{ \k_+}}^{(\e_{r},\Delta )}(\theta) +  R_{\e_j}^{(\e_{\k_+},\kappa^{(\k)})}(\theta) \leq U_j^{(N)} +1 \right)
\end{align*}
Let $\tau_j^{(r,\k_+)}:= \tau_j^{(r)}$ if $j\leq \k_+$ and $\tau_j^{(r,\k_+)}:= \tau_{\k_+}^{(r)}+ \sum_{l=\k_+}^{j-1} \sum_{p=0}^{2^{\kappa^{(\k)}}-1} (2^{\kappa^{(\k)}} +p)^{-1}$ if $j\geq \k_+$ and $E^{(\k_+)}_j:= \tau_j^{(r)}-\tau_j^{(r,\k_+)}$, for any $r\leq j\leq N$ . It is plain to check  that $\sum_{j=r}^{N-1}|E^{(\k_+)}_{j+1} -E^{(\k_+)}_{j} | \ll 2^{- \frac{r}{100}}$ and
\begin{align*}
(R_{\e_{ j\wedge \k_+}}^{(\e_{r},\Delta )}(\theta') +  R_{\e_j}^{(\e_{\k_+},\kappa^{(\k)})}(\theta')\mathds{1}_{{\{ j >k_0  \}}})_{r\leq j\leq N} \eqlaw \sqrt{\frac{1}{2}} ( W_{\tau_j^{(r,\k_+)}})_{j\geq r} \ .
\end{align*}
Now it suffices to reproduce the proof of Lemma \ref{lemma:kDebut}. In this first part, we assume $\frac{||\theta - \theta'||}{2 \pi} = 2^{-\k}$. In this case, we check the independence of $J_{l,p}^{(\k)}( \theta)$ and $J_{l,p}^{(\k)}( \theta')$ for $l \geq \k_+$, since $\k_+ - \kappa^{(\k)} \geq \k$. We then show, by doing the suitable conditionings, that the increments of $(R_{\e_k}^{(\e_{\k_+},\kappa^{(\k)})}(\theta))_{\k_+ \leq k \leq N}$, $(R_{\e_k}^{(\e_{\k_+},\kappa^{(\k)})}(\theta'))_{\k_+ \leq k \leq N}$, 
$(R_{\e_k}^{(\e_{r},\Delta)}(\theta'))_{r \leq k \leq \k_+}$ are independent. Thus we have
\begin{align}
\nonumber 
  &  \P_N^{(\Delta,\kappa)}(\theta,\theta')\leq  \P\left(  Ev(r,1,E^{(\k_+)},0) \right) \max_{    l_{\k_+}^{(N)} -1 \leq z\leq u_{\k_+}^{(N)}+1   }\P( Ev(\k_+,1,E^{(\k_+)},z)   ).
\end{align}

By the same arguments as in the proof of Lemma \ref{lemma:kDebut} in the dyadic case we have 
$$	\max_{    l_{\k_+}^{(N)}-1  \leq z\leq u_{\k_+}^{(N)} +1  }\P( Ev(\k_+,1,E^{(\k_+)},z)   )\ll_{\upsilon} 
	2^{-(N-\k_+)}   e^{- (N-\k_+)^{\alpha_-}}$$ 
	and 
	$$
	\P\left(  Ev(r,1,E^{(\k_+)},0) \right) \ll_{\upsilon} 2^{-N+r}.$$
 Finally one gets
\begin{align*}
	\P_N(\theta,\theta') \ll_{\upsilon}  2^{r-N}  2^{\k_+-N} e^{-  (N-\k_+)^{\alpha_-}} \ .
\end{align*}
It concludes the proof of Lemma \ref{lemma:kMil}, when $\frac{||\theta-\theta'||}{2\pi}$ is a negative power of $2$. 
\end{proof}

\subsubsection{General case}
 
\begin{proof}[Proof of Lemma \ref{lemma:kDebut} in general case]
Fix $\upsilon >0$, $r$ large enough, $N$ large enough depending on $r$, and $\k$ such that  $ \k \leq \frac{N}{2} $. Unlike the previous dyadic case, for $ \k_+ \leq l \leq N-1$ and $p\leq 2^{\Delta^{(l)}}-1$, the random variables $ I_{l,p}(\theta)$ and $ I_{l,p}(\theta')$ are not rigorously independent. However observe that for any $\k_+ \leq l \leq N-1$, the absolute value of their correlations decreases exponentially with $l$. Indeed, 
if $C_{l,p}(\theta,\theta') := \E \left( I_{l,p}(\theta)\overline{I_{l,p}(\theta')} \right)$, then
\begin{align*}
\left| C_{l,p}(\theta,\theta')\right| &=\frac{1}{{2^l+p2^{l-\Delta^{(l)}}}}\left|  \sum_{j=0}^{2^{l-\Delta^{(l)}}-1}  e^{i(\theta-\theta') j} \right|  \leq \frac{4}{2^l ||\theta-\theta'|| } \ll 2^{\k-l}.
\end{align*}
Since $ \E \left( I_{l,p}(\theta)I_{l,p}(\theta') \right) = 0$, and $(I_{l,p}(\theta), I_{l,p} (\theta'))$ is a centered complex Gaussian vector, one checks, by computing  covariances, that it is possible to write 
 $$I_{l,p} (\theta) = \frac{C_{l,p}(\theta,\theta')}{C_{l,p} (\theta', \theta')}
 I_{l,p} (\theta') 
 + I^{ind}_{l,p} (\theta, \theta'),$$
where the two terms of the sums are independent, with an expectation of the square equal to zero. Note that 
 $$C_{l,p} := C_{l,p}(\theta', \theta')
 = \frac{2^{l-\Delta^{(l)}}}{2^l + p2^{l-\Delta^{(l)}}} = (2^{\Delta^{(l)}} + p)^{-1}$$
does not depend on $\theta'$. Moreover, we have by Pythagoras' theorem: 
 $$\E [|I_{l,p}^{ind}(\theta, \theta')|^2] = \E [|I_{l,p}(\theta)|^2]
 - \left|\frac{C_{l,p}(\theta, \theta')}{C_{l,p}} \right|^2 \E [|I_{l,p}(\theta')|^2]
 = C_{l,p} - \frac{|C_{l,p}(\theta, \theta')|^2}{C_{l,p}}$$
 Using this decomposition of $I_{l,p}(\theta)$ and the measurability of the different quantities with respect to the $\sigma$-algebras of the form $\mathcal{G}_j$, we deduce that one can write: 
\begin{align}
\label{decomposi}
 (R_{\e_l}^{(\e_{\k_+},\Delta)}(\theta))_{l\geq \k_+}= (R_{\e_l}^{(\e_{\k_+},ind)}+ E_{\e_l}^{(\e_{\k_+})} )_{l\geq\k_+} \ .
\end{align}
Here $(R_{\e_l}^{(\e_{\k_+},ind)})_{l\geq \k_+}$ is a Gaussian process, independent of $(R_{\e_l}^{(\e_{\k_+},\Delta)}(\theta'))_{l\geq \k_+}$, and distributed as $  ( \sqrt{\frac{1}{2}}W_{ \tau^{(\k_+)}_l- \mathtt{C}_l}  )_{l\geq \k_+}$ with 
$$\mathtt{C}_l:= \sum_{t=\k_+}^{l-1}  \sum_{p=0}^{2^{\Delta^{(t)}} -1 } \frac{ |C_{t,p}(\theta, \theta')|^2}{C_{t,p}} \ .$$
Notice that $\mathtt{C}=\mathtt{C}^{(\k_+)}$ implicitly depends of $\k_+$. $(E_{\e_l}^{(\e_r)} )_{l\geq\k_+}$ on the other hand is defined by
\begin{align}
E_{\e_{l+1}}^{(\e_{l})}  = \Re \left(\sigma \sum_{p=0}^{2^{\Delta^{(l)}} -1 }e^{ i\psi_{ 2^{l}+p 2^{l-\Delta^{(l)}}} (\theta )   }
\frac{C_{l,p} (\theta, \theta')}{C_{l,p}}
 I_{l,p}(\theta') \right).  \label{DefE}
\end{align}

Furthermore notice that
\paragraph{Fact 1:} For any $l\geq \k_+$, 
\begin{align*}
|E_{\e_{l+1}}^{(\e_{l})}| \leq |E|_{\e_{l+1}}^{(\e_{l})}:= \sum_{p=0}^{2^{\Delta^{(l)}} -1 } |C_{l,p}(\theta, \theta')| (2^{\Delta^{(l)}}+p ) |I_{l,p}(\theta') |
\end{align*}
is measurable with respect to the sigma field $\sigma\left(  \Nc_t^\C,\, t\in  [|2^{l},2^{l+1}-1|]\right)$.

\paragraph{Fact 2:} The process $(R_{\e_l}^{(\e_{\k_+},ind)})_{l\geq\k_+} $ is independent of the couple $( R_{\e_l}^{(\e_{\k_+},\Delta)}(\theta') , |E|_{\e_l}^{(\e_r)})_{l\geq \k_+}$ 

\paragraph{Fact 3:}  $\sup_{l\geq \k_+}|\mathtt{C}_l| \ll 2^{-\Delta^{(\k)}}$ if $r$ is large enough. Indeed, in this case, 
since $\k_+ \geq 3 \Delta^{(\k)} \geq r/40$ is also large, we have
\begin{align*}
\sup_{l\geq \k_+}|\mathtt{C}_l| 
& \ll \sum_{l = \k_+}^{\infty} 
\sum_{p=0}^{2^{\Delta^{(l)}} - 1}  (2^{\k - l})^2
(2^{\Delta^{(l)}}  + p)
\ll  \sum_{l = \k_+}^{\infty} 2^{2\k - 2l + 2 \Delta^{(l)}}
\ll \sum_{l = \k_+}^{\infty} 2^{2\k - 2l + 
(r/50) + 200 \log^2 l}
\\ &  \ll \sum_{\k_+ \leq l \leq 100 \k} 
2^{2\k - 2 l + (r/50) + 300 \log^2 \k}
+ \sum_{l \geq \max (100\k, \k_+)} 2^{2\k - 2 l + 0.8 \k_+ + 200 \log^2 l}
\\ & \ll \sum_{\k_+ \leq l \leq 100 \k}
2^{2 \k - 2 l + 3 \Delta^{(\k)}}
+ \sum_{l\geq \max (100\k, \k_+)} 2^{2 \k + 0.8 \k_+ - 1.99 l} 
\\ & \ll 2^{2 \k - 2 \k_+ +  3 \Delta^{(\k)}}
+ 2^{2 \k + 0.8 \k_+  - 1.99[(0.05)(100 \k) + 
0.95 (\k_+)]}
\ll  2^{- 3 \Delta^{(\k)}} + 2^{- \k_+}
\ll 2^{- 3 \Delta^{(\k)}}.
\end{align*}
It means (see Lemma \ref{lemma:shifted_barrier_estimate}) that the process $(R_{\e_l}^{(\e_{\k_+},ind)})_{l\geq \k_+}$ is very "similar" to $ \sqrt{\frac{1}{2}} (W_{ \tau_{l}^{(\k_+)}})_{l\geq \k_+ }$. 
 Moreover, if for $\k_+ \leq r \leq l$, 
  $\mathtt{C}_l^{(r)} := \mathtt{C}_l - 
  \mathtt{C}_r$, then we have for $r$ large enough: 
  $$\sup_{l \geq r} |\mathtt{C}_l^{(r)}|
  \leq \sup_{l \geq \k_+}|\mathtt{C}_l|
  \leq 2^{-  \Delta^{(\k)}}
  \leq 2^{-(r/100) + 1} \leq 1/2.$$

\paragraph{Fact 4:} $|E|$ is small. For any $m \geq 0$, $l\geq \k_+$, we introduce the event
$$\mathtt{E}_m^{(l)}:=\{ |E|_{\e_{l+1}}^{(\e_l)}  \geq 2^{-\frac{\Delta^{(l)}}{4}}m\} \ .$$
For some universal constants $c, c' > 0$, and $m \geq 1/2$, the probability of $\mathtt{E}_m^{(l)} $ is smaller than
\begin{align}
 \nonumber  \P\left( |E|_{\e_{l+1}}^{(\e_l)}
 \geq 2^{-\frac{\Delta^{(l)}}{4}}m \right) \leq 2^{\Delta^{(l)}} \P\left( c 2^{\k-l} |I_{l,0}(\theta')|\geq m  2^{-\frac{9}{4}\Delta^{(l)}} \right) &\ll 2^{\Delta^{(l)}} e^{- c'm^2 2^{2(l-\k - \frac{7}{4}\Delta^{(l)}) }    }
\\
\label{fact4}&\ll  e^{- c'm^2 2^{2(l-\k - 2\Delta^{(l)}) }    }.
\end{align} 
For $r$ (and then $\k_+$ and $l$) large enough and $l \leq 100 \k$, $2\Delta^{(l)} \leq 3 \Delta^{(\k)}
+(\log (c')/\log 4)$, and 
then 
$$ \P\left( |E|_{\e_{l+1}}^{(\e_l)} \geq 2^{-\frac{\Delta^{(l)}}{4}}m \right)  \ll e^{-m^2 2^{2(l-\k_+)} }.$$
If $r$ is large enough and $l \geq \sup(100 \k, \k_+)$, we use that
$\k_+ \geq 3 \Delta^{(\k)} \geq r/40$,
$$l - 2 \Delta^{(l)} \geq 
l - \frac{r}{50} - 200 \log^2 l 
\geq 0.99l - 0.8 \k_+,$$ 
and then 
$$\P\left( |E|_{\e_{l+1}}^{(\e_l)}  \geq  
2^{-\frac{\Delta^{(l)}}{4}}m \right)  
\ll e^{-c' m^2 2^{2(0.99 l - \k - 0.8 \k_+)}}
\leq e^{-c' m^2 2^{l + 0.98 [0.05(100 \k) + 0.95 \k_+]
- 2\k - 1.6 \k_+}}. 
$$
Hence in any case, for $r$ large and $l \geq \k_+$, 
$$\P\left( |E|_{\e_{l+1}}^{(\e_l)}  \geq  
2^{-\frac{\Delta^{(l)}}{4}}m \right) 
\ll e^{-m^2 2^{l -\k_+}}.$$

{\bf When $\k_+\leq r$:} Using the decomposition (\ref{decomposi}) and the Fact 2, and noticing that $\sum_{l=r}^{+\infty} m 2^{-\frac{\Delta^{(l)}}{4}} \leq  m2^{-\frac{r}{400}}$, for $r$ large enough, we can affirm that
\begin{align}
\label{REM}
\nonumber &\P_N(\theta,\theta') \leq \P\left(  Ev(r,0,0,0)  \right) \P\left( Ev(r, 2^{-\frac{r}{400}} ,\mathtt{C}^{(r)},0)\right)
\\
&\qquad \qquad +\sum_{m\geq 1} \P\left(  Ev(r,0,0,0) ,\, \cup_{j\in [|r,N-1|]}\mathtt{E}_m^{(j)}  \right) \P\left( Ev(r,(m+1)2^{-\frac{r}{400}},\mathtt{C}^{(r)},0)\right),
\end{align}
where the Brownian motion involved in the event $Ev(r,0,0,0)$ is suitably coupled with the complex Gaussian random walk whose increments are of the form
$ I_{l,p}(\theta') e^{i \psi_{2^l + p2^{l-\Delta^{(l)}}} (\theta')}$ 
for $r \leq l \leq N-1$ and $0 \leq p \leq 2^{\Delta^{(l)}} - 1$. 

By using Eq. \eqref{eq:girsanov2} and then the fact that $\sup_{r \leq j \leq N}\mathtt{C}^{(r)}_j \leq 2^{1-(r/100)} \leq 2^{-r/400}$
if $r$ is large enough, we have for any $m \geq 0$: 
\begin{align}
\nonumber &\P\left( Ev(r,(m+1)2^{-\frac{r}{400}} ,\mathtt{C}^{(r)},0)\right) 
\\
\nonumber &\leq 2^{r -N} e^{2\upsilon + 2(m+1) 2^{-\frac{r}{400}} - \mathtt{C}^{(r)}_N}N^{\frac{3}{2}}
                \P\left( GEv(r,(m+1)2^{-\frac{r}{400}} + \sup_{r \leq j \leq N}\mathtt{C}^{(r)}_j,\mathtt{C}^{(r)},0) \right)
\\
\label{subiel}&\leq  2^{r -N} e^{2\upsilon + 2(m+1) 2^{-\frac{r}{400}} - \mathtt{C}^{(r)}_N}N^{\frac{3}{2}} \P\left( GEv(r,2(m+1)2^{-\frac{r}{400}}  ,\mathtt{C}^{(r)},0) \right) \ .
\end{align}
Now, we invoke Corollary \ref{corollary:shifted_barrier_estimate} as, by the Fact 3, with the notation of the corollary, $||E||_1 \leq 1$ if $r$ is large enough. Thus, we deduce, for $N$ large enough depending on $r$, that 
\begin{align}
\label{kpluspetit0}
 \P\left( Ev(r,(m+1)2^{-\frac{r}{400}} ,\mathtt{C}^{(r)},0)\right)  \ll_{\upsilon}  2^{r -N} e^{ 2m 2^{-\frac{r}{400}} } (1 + m 2^{-\frac{r}{400}} )^3 \ .
\end{align}
Similarly, to compute $\P\left(Ev(r,0,0,0) ,\,  \cup_{j\in [|r,N-1|]}\mathtt{E}_m^{(j)}\right)$ we will apply the Girsanov transform with the density $e^{\sqrt{2 }W_{\tau_N^{(r)}}- \tau_N^{(r)} }$. It requires to study what is the effect of this density on the event $ \cup_{j\in [|r,N-1|]}\mathtt{E}_m^{(j)}$.
The increments of the complex random walk which
were $I_{j,p}(\theta') e^{i \psi_{2^j + p 2^{j - \Delta^{(j)}}}(\theta')}$ before the Girsanov transform, increase
by $\sigma^{-1} C_{j,p}$ afterwards. Hence, between the two situations, before and after the Girsanov 
transform, $|E|_{\e_{j+1}}^{(\e_{j})}$, defined as the sum, for $0 \leq p \leq 2^{\Delta^{(j)}} - 1$, of 
the absolute value of the increments of the random walk multiplied by $|C_{j,p} (\theta, \theta')|/C_{j,p}$, vary, for $r$ large enough, at most by $  2^{2\Delta^{(j)} +\k -j}$, since 
$$\sum_{p = 0}^{2^{\Delta^{(j)}}-1} 
|C_{j,p} (\theta, \theta')|
\ll 2^{\Delta^{(j)}+ \k - j}.$$
Now, for $j\geq r\geq \k_+$ and $r$ large enough, we  have:
$$  2^{2\Delta^{(j)} +\k -j}
< \frac{1}{2} 2^{-\frac{1}{4}\Delta^{(j)}}.$$
Indeed, for $j \leq 100 \k$, 
$$ 2^{2\Delta^{(j)} +\k -j}
\leq 2^{2\Delta^{(j)} +\k -\k_+}
\leq 2^{2\Delta^{(j)} - 3\Delta^{(\k)}}
\leq 2^{-0.9\Delta^{(j)} },
$$
and for $j \geq \max(\k_+, 100 \k)$, 
\begin{align*} 2^{2\Delta^{(j)} +\k -j}
& \leq 2^{2\Delta^{(j)} -0.99j}
\leq 2^{(r/50) + 200 \log^2 j  -0.99j}
\leq 2^{(j/50) + 200 \log^2 j  -0.99j}
\\ & \leq 2^{-0.96j} \leq 2^{-(r/50)- 0.94j}
\leq 2^{-\Delta^{(j)}}. 
\end{align*}

Hence, if for $m \geq 1$, before (respectively after) the Girsanov transform, $ \cup_{j\in [|r,N-1|]}\mathtt{E}_m^{(j)} $ occurs, then  $  \cup_{j\in [|r,N-1|]}\mathtt{E}_{\half m}^{(j)}  $ still occurs after (respectively before) the transform. Finally we get, for any $m \geq 1$, 
\begin{align*}
\P\left(  Ev(r,0,0,0) ,\, \cup_{j\in [|r,N-1|]}\mathtt{E}_m^{(j)}  \right)  &\leq  2^{r -N}  \E\left( e^{-\sqrt{2}W_{\tau_N^{(r)}}  }  \mathds{1}_{\{   \forall j\in [|r, N |] ,\, l_j^{(N)}\leq \sqrt{\frac{1}{2}}  W_{\tau_j^{(r)}}  \leq u_j^{(N)} \}} \mathds{1}_{ \cup_{j\in [|r,N-1|]}\mathtt{E}_{\half m}^{(j)}   }    \right)
\\
&\ll_{\upsilon} 2^{r-N} N^{\frac{3}{2}} \P\left( GEv(r,0,0,0)\cap\left\{ \cup_{j\in [|r,N-1|]}\mathtt{E}_{\half m}^{(j)} \right\} \right)
\end{align*}
As $\mathtt{E}_{\frac{m}{2}}^{(l)}$ is measurable with respect to $\sigma( \Nc_t^\C,\, t\in  [|2^{l},2^{l+1}-1|])$, by applying  Corollary \ref{corollary:shifted_barrier_estimate} and using the {\bf Fact 4} we get, for $r$ large enough and $N$ large enough depending on $r$: 
\begin{align*}
\P\left( GEv(r,0,0,0)\cap \left\{ \cup_{j\in [|r,N-1|]}\mathtt{E}_{\half m}^{(j)} \right\}  \right)   & \ll_{\upsilon} N^{-\frac{3}{2}} \sum_{j\geq r} \sqrt{\P(\mathtt{E}_{\frac{1}{2}m}^{(j)})  } \ll N^{-\frac{3}{2}} 
\sum_{j \geq r} e^{-(m^2/8) 2^{j - \k_+}}
\\ &  \ll N^{-\frac{3}{2}} e^{-2^{r-\k_+}m^2/8},
\end{align*}
since $r \geq \k_+$. 

By combining this inequality with  \eqref{kpluspetit0} and \eqref{REM}, we get:
\begin{align}
\nonumber \P_N(\theta, \theta') &\ll_{\upsilon} 2^{2 (r-N)}   +   \sum_{m\geq 1} 
2^{r-N} 2^{2m 2^{-r/400}} (1+ m 2^{-r/400})^3
2^{r-N} e^{-(m^2/8) 2^{r-\k_+}}
\\ & \leq 2^{2(r-N)} \left[ 1+ \left(\sum_{m \geq 1}
2^{2m} (1+m)^3 e^{-m^2/8} \right) e^{-(1/8)(2^{r-\k_+} - 1)} \right]
\ll 2^{2(r-N)} \label{Gillet},
\end{align}
which concludes the case $\k_+\leq r$.

{\bf When $\k_+\geq r$:} Using to the decomposition (\ref{decomposi}) and the Fact 2 and noticing that $\sum_{l=\k_+}^{+\infty} m 2^{-\frac{\Delta^{(l)}}{4}}\leq m2^{-\frac{r}{400}}$ for $r$ large enough,  we can affirm that
\begin{align}
\label{REM2}  \P_N(\theta,\theta') &\leq \sum_{m\geq 0} \P\left(  Ev(r,0,0,0) ,\, \cup_{j\in [|\k_+,N-1|]}\mathtt{E}_m^{(j)}  \right) \sup_{    l_{\k_+}^{(N)}  \leq z\leq u_{\k_+}^{(N)}  } \P\left( Ev(\k_+,(m+1) 2^{-\frac{r}{400}},\mathtt{C},z)\right).
\end{align}
 By a similar computation as what we have done in the case $\k_+ \leq r$, we get: 
\begin{align}
\label{oto}
\P\left(  Ev(r,0,0,0) ,\, \cup_{j\in [|\k_+,N-1|]}\mathtt{E}_m^{(j)}  \right) 
\ll_{\upsilon}  2^{r-N-(m^2/8)}.
\end{align}
On the other hand, by using Eq. \eqref{eq:girsanov2}, for any $z\in [ l_{\k_+}^{(N)} , u_{\k_+}^{(N)} ]$,  we obtain: 
\begin{align}
\nonumber &\P\left( Ev(\k_+, (m+1) 2^{-\frac{r}{400}},\mathtt{C},z)\right) 
\\
\nonumber &\ll_{\upsilon}
2^{\k_+  -N}  e^{   2(m+1) 2^{-\frac{r}{400}}} N^{\frac{3}{2}} e^{2z  } \P\left( GEv(\k_+, (m+1) 2^{-\frac{r}{400}} + \sup_{\k_+ \leq j \leq N} |\mathtt{C}_j|,\mathtt{C},z) \right)
\\
\label{subielbi}&  \leq
2^{\k_+  -N}  e^{   2(m+1) 2^{-\frac{r}{400}}} N^{\frac{3}{2}} e^{2z  } \P\left( GEv(\k_+,   2(m+1) 2^{-\frac{r}{400}},\mathtt{C},z) \right),
\end{align}
where we used that $\sup_{j\geq \k_+} |\mathtt{C}_j|\leq 2^{-\Delta^{(\k)}} \leq \frac{1}{2}(m+1) 2^{-\frac{r}{400}} $. 

For $\k_+ \leq N/4$, we can use Corollary \ref{corollary:shifted_barrier_estimate} to deduce, for $z\leq u_{\k^+}^{(N)}=\upsilon- (\k_+-r)^{\alpha_-}$,  
\begin{align}
 \P\left( Ev(\k_+,   (m+1) 2^{-\frac{r}{400}} ,\mathtt{C},z)\right) &
\ll_{\upsilon} 2^{\k_+ -N} e^{   2(m+1) 2^{-\frac{r}{400}}}  (1+ 2(m+1) 2^{-\frac{r}{400}} )^3 (1+ (-z)_+)  e^{2z  } 
\nonumber \\ & \ll_{\upsilon}  2^{\k_+ -N} e^{3(m +1)- (\k_+-r)^{\alpha_-}}\label{kpluspetit0bi} \ .
\end{align}

For $\k_+ \geq N/4$, we bound the probability of the $GEv$ event by $1$ and use the fact that $z \leq \upsilon - 0.9(\k_+ - r)^{\alpha_-}$ 
(the factor $0.9$ coming from the case $\k \leq N/2 \leq \k_+$). We then get  
\begin{align*}
 P\left( Ev(\k_+,   (m+1) 2^{-\frac{r}{400}} ,\mathtt{C},z)\right) & \ll_{\upsilon}
2^{\k_+  -N}  e^{   2(m+1) 2^{-\frac{r}{400}}} N^{\frac{3}{2}} e^{2z  }
\\ & \ll_{\upsilon} 2^{\k_+  -N}  e^{2m+2} e^{-(\k_+ - r)^{\alpha_-}} (N^{3/2} e^{-0.8(\k_+ - r)^{\alpha_-}})
\\ & \ll 2^{\k_+  -N}  e^{2m+2} e^{-(\k_+ - r)^{\alpha_-}},
\end{align*}
the last line coming from the fact that  $\k_+ - r \geq (N/4) - r \gg N$ if $N \geq 5r$. This again implies \eqref{kpluspetit0bi}. 
Finally, by combining  this equation with (\ref{oto}) and (\ref{REM2}), we get
\begin{align*}
\P_{N}(\theta, \theta')& \ll_{\upsilon}  2^{r-N} 2^{\k_+ -N}  e^{- (\k_+-r)^{\alpha_-}}   \sum_{m\geq 0}  e^{3m+3-(m^2/8)}
\ll 2^{r-N} 2^{\k_+ -N}  e^{- (\k_+-r)^{\alpha_-}}  
\end{align*}
which concludes the proof of Lemma \ref{lemma:kDebut}.
\end{proof}

\begin{proof}[Proof of Lemma \ref{lemma:ktoutDebut}]
We can use \eqref{subiel} in order to get (for $r$ large enough and $N$ large enough depending on $r$): 
\begin{align*}
&\P(Ev(r,0,0,0)) 
\P(Ev(r, 2^{-\frac{r}{400}}, \mathtt{C}^{(r)}, 0))
\\ & \leq e^{4 \upsilon + 2^{1-(r/400)}} 2^{2(r-N)} N^{3} 
\P(GEv(r,0,0,0))
\P(GEv(r,2^{1-\frac{r}{400}},\mathtt{C}^{(r)},0)),
\end{align*}
and then, by the majorization of the second term of \eqref{REM} which is involved in \eqref{Gillet}: 
$$\P_N(\theta, \theta') 
\leq e^{4 \upsilon + 2^{1-(r/400)}} 2^{2(r-N)} N^{3} 
\P(GEv(r,0,0,0))
\P(GEv(r,2^{1-\frac{r}{400}},\mathtt{C}^{(r)},0)) + \mathcal{O}_{\upsilon} 
\left( 2^{2(r-N)} e^{-\frac{2^{r-\k_+} - 1}{8}} \right).$$
Hence, we have: 
\begin{align*} \P_N(\theta, \theta') 
& \leq e^{4 \upsilon + 2^{1 - (r/400)}} 
2^{2(r-N)} N^{3} 
\P(GEv(r,0,0,0))
\P(GEv(r,2^{1-\frac{r}{400}},\mathtt{C}^{(r)},0)) 
\\ & + \mathcal{O}_{\upsilon} 
\left( 2^{2(r-N)}
 e^{-\frac{2^{r - (r/2)- (r/100) - 100 \log^2 (r/2)} - 1}{8}} \right).
 \end{align*}
By applying \eqref{ProClas1} and using the fact that, with the notation of this equation, 
$||E||_1$ and $\delta$ go to zero when $r$ goes to infinity, we get 
$$\P_N(\theta, \theta') 
\leq e^{4 \upsilon} 2^{2(r -N)} N^{3}   (\P(Event_{r,N}))^2 (1+ \eta_r)
+  \mathcal{O}_{\upsilon} 
\left( 2^{2(r-N)} e^{-\frac{2^{r/3} - 1}{8}} \right),
 $$
 where $\eta_r$ goes to zero when $r$ goes to infinity. 
 Now, by computing the lower bound of the first moment of $\I_N$, we have proven 
 that $\P(Event_{r,N}) \gg_{\upsilon}
 N^{-3/2}$. 
 Hence, 
 $$\P_N(\theta, \theta') 
\leq 2^{2(r -N)} N^{3}   (\P(Event_{r,N}))^2 (1+ \eta_{r, \upsilon}),$$
where 
$$\eta_{r, \upsilon} 
=  e^{4 \upsilon} - 1 + \eta_r  e^{4 \upsilon} + \mathcal{O}_{\upsilon} \left(e^{-\frac{2^{r/3} - 1}{8}} \right).$$
 When we let $r \rightarrow \infty$, we 
 get $e^{4 \upsilon} - 1$, which tends to zero 
 with $\upsilon$. 
 \end{proof}

\begin{proof}[Proof of Lemma \ref{lemma:kMil} in the general case]
The general case needs to uses exactly the same arguments used in the general case of the proof of  Lemma \ref{lemma:kDebut}.  This time, for $ \k_+ \leq l \leq N-1$ and $p\leq 2^{\kappa^{(\k)}}-1$, the random variables $ J_{l,p}^{(\k)}(\theta)$ and $ J_{l,p}^{(\k)}(\theta')$ are not rigorously independent. However, we observe that for $\k_+ \leq l \leq N-1$, the absolute value of their correlations,  decreases exponentially with $l$. Indeed, if $C_{l,p}^{(\k)}(\theta,\theta') := \E \left( J_{l,p}^{(\k)}(\theta)\overline{J_{l,p}^{(\k)}(\theta')} \right)$, then

\begin{align*}
\left| C_{l,p}^{(\k)}(\theta,\theta')\right| &=\frac{1}{{2^l+p2^{l-\kappa^{(\k)}}}}\left|  \sum_{j=0}^{2^{l-\kappa^{(\k)}}-1}  e^{i(\theta-\theta') j} \right|  \leq \frac{4}{2^l ||\theta-\theta'|| } \ll 2^{\k-l} \ .
\end{align*}

Since $ \E \left( J_{l,p}^{(\k)}(\theta)J_{l,p}^{(\k)}(\theta') \right) = 0$, and the vector 
$(J_{l,p}^{(\k)}(\theta), J_{l,p}^{(\k)} (\theta'))$ is centered complex Gaussian, one checks, by computing  covariances, that it is possible to write 
$$J_{l,p}^{(\k)} (\theta) = \frac{C_{l,p}^{(\k)}(\theta,\theta')}{C_{l,p}^{(\k)} (\theta', \theta')}
J_{l,p}^{(\k)} (\theta') 
+ J^{(\k),ind}_{l,p} (\theta, \theta'),$$
where the two terms of the sums are independent, with an expectation of the square 
equal to zero. Note that 
$$C^{(\k)}_{l,p} := C^{(\k)}_{l,p}(\theta', \theta')
= \frac{2^{l-\kappa^{(\k)}}}{2^l + p2^{l-\kappa^{(\k)}}} = (2^{\kappa^{(\k)}} + p)^{-1},$$
does not depend on 
$\theta'$. Moreover, we have by Pythagoras' theorem: 
$$\E [|J_{l,p}^{(\k),ind}(\theta, \theta')|^2] = \E [|J_{l,p}^{(\k)}(\theta)|^2]
- \left|\frac{C_{l,p}^{(\k)}(\theta, \theta')}{C_{l,p}^{(\k)}} \right|^2 \E [|J_{l,p}^{(\k)}(\theta')|^2]
= C_{l,p}^{(\k)} - \frac{|C_{l,p}^{(\k)}(\theta, \theta')|^2}{C_{l,p}^{(\k)}}.$$
Using this decomposition of $J_{l,p}^{(\k)}(\theta)$ and the measurability 
of the different quantities with respect to 
the $\sigma$-algebras of the form $\mathcal{G}_j$, we deduce that one can write: 
\begin{align}
\label{decomposiBIS}
(R_{\e_l}^{(\e_{\k_+},\kappa^{(\k)})}(\theta))_{l\geq \k_+}= (R_{\e_l}^{(\e_{\k_+},ind)}+ E_{\e_l}^{(\e_{\k_+})} )_{l\geq\k_+},
\end{align}
with $(R_{\e_l}^{(\e_{\k_+},ind)})_{l\geq \k_+}$ is a Gaussian process, independent of
$(R_{\e_l}^{(\e_{\k_+},\kappa^{(\k)})}(\theta'))_{l\geq \k_+}$, and 
distributed as $  ( \sqrt{\frac{1}{2}}W_{ \tau^{(\k_+,\k_+)}_l- \mathtt{C}^{(\k)}_l}  )_{l\geq \k_+}$ with $\tau^{(\k_+,\k_+)}_j=  \sum_{l=\k_+}^{j-1} \sum_{p=0}^{2^{\kappa^{(\k)}}-1} (2^{\kappa^{(\k)}} +p)^{-1}$ and
$$\mathtt{C}^{(\k)}_l:= \sum_{t=\k_+}^{l-1}  \sum_{p=0}^{2^{\kappa^{(\k)}} -1 } \frac{ |C_{t,p}^{(\k)}(\theta, \theta')|^2}{C_{t,p}^{(\k)}} $$ and $(E_{\e_l}^{(\e_{\k_+})} )_{l\geq\k_+}$ defined by
\begin{align}
E_{\e_{l+1}}^{(\e_{l})}  = \Re \Bigg(\sigma \sum_{p=0}^{2^{\kappa^{(\k)}} -1 }e^{ i\psi_{ 2^{l}+p 2^{l-\kappa^{(\k)}}} (\theta )   }
\frac{C_{l,p}^{(\k)}(\theta, \theta')}{C_{l,p}^{(\k)}}
J_{l,p}^{(\k)}(\theta') \Bigg).  \label{DefEbis}
\end{align}
Note that $\mathtt{C}^{(\k)}_l$ and  $E_{\e_{l+1}}^{(\e_{l})}$ here represent quantities which are different from those denoted in the same way in the proof of Lemma \ref{lemma:kDebut}. 
Furthermore notice that
\paragraph{Fact 1:} For any $l\geq \k_+$, 
\begin{align*}
|E_{\e_{l+1}}^{(\e_{l})}| \leq |E|_{\e_{l+1}}^{(\e_{l})}:= \sum_{p=0}^{2^{\kappa^{(\k)}} -1 } |C_{l,p}^{(\k)}(\theta, \theta')| (2^{\kappa^{(\k)}}+p ) |J_{l,p}^{(\k)}(\theta') |
\end{align*}
is measurable with respect to the sigma field $\sigma\left(  \Nc_t^\C,\, t\in  [|2^{l},2^{l+1}-1|]\right)$.

\paragraph{Fact 2:} The process $(R_{\e_l}^{(\e_{\k_+},ind)})_{l\geq\k_+} $ is independent of the couple $( R_{\e_l}^{(\e_{\k_+},\kappa^{(\k)})}(\theta') , |E|_{\e_l}^{(\e_{\k_+})})_{l\geq \k_+}$.

\paragraph{Fact 3:}  $\sup_{l\geq \k_+}|\mathtt{C}_l^{(\k)}| \ll 2^{-\kappa^{(\k)}}$ if $r$ is large enough. Indeed we have
\begin{align*}
\sup_{l\geq \k_+}|\mathtt{C}_l^{(\k)}| 
& \ll \sum_{l = \k_+}^{\infty} 
\sum_{p=0}^{2^{\kappa^{(\k)}} - 1}  (2^{\k - l})^2
(2^{\kappa^{(\k)}}  + p)
\ll  \sum_{l = \k_+}^{\infty} 2^{2\k - 2l + 2 \kappa^{(\k)}}
\ll \sum_{l = \k + 3\kappa^{(\k)}}^{\infty} 2^{2 (\k +3\kappa^{(\k)} - l) -4\kappa^{(\k)}}   	\ll 2^{- 2\kappa^{(\k)}}.
\end{align*}
It means that the process $(R_{\e_l}^{(\e_{\k_+},ind)})_{l\geq \k_+}$ is very "similar" to $ \sqrt{\frac{1}{2}} (W_{ \tau_{l}^{(\k_+,\k_+)}})_{l\geq \k_+ }$.

\paragraph{Fact 4:} $|E|$ is small. For any $m \geq 0$, $l\geq \k_+$, we introduce the event 
$$\mathtt{E}_m^{(l)}:=\{ |E|_{\e_{l+1}}^{(\e_l)}  \geq 2^{-\frac{\kappa^{(\k)}}{4}}m\} \ . $$
For some universal constants $c, c' > 0$, and $m  \geq 1/2$, the probability of $\mathtt{E}_m^{(l)} $ is smaller than
\begin{align*}
    \P\left( |E|_{\e_{l+1}}^{\e_l} \geq 2^{-\frac{\kappa^{(\k)}}{4}}m \right) 
\leq 2^{\kappa^{(\k)}} \P\left( c 2^{\k-l} |J_{l,0}^{(\k)}(\theta')|\geq m  2^{-\frac{9}{4}\kappa^{(\k)}} \right) 
& \ll 2^{\kappa^{(\k)}} e^{- c'm^2 2^{2(l-\k - \frac{7}{4}\kappa^{(\k)}) } }\\
& \ll e^{- c'm^2 2^{2(l-\k - 2\kappa^{(\k)}) } }.
\end{align*} 
Then for $r$ large and $l \geq \k_+$, 
$$\P\left( |E|_{\e_{l+1}}^{\e_l}  \geq 2^{-\frac{\kappa^{(\k)}}{4}}m \right) \ll e^{-m^2 2^{l -\k_+}} \ .$$

Using to the decomposition (\ref{decomposiBIS}) and the Fact 2 and noticing that $\sum_{l=\k_+}^{N} m 2^{-\frac{\kappa^{(\k)}}{4}}\leq m2^{-\frac{r}{400}}$ for $r$ large enough,  we can affirm that
\begin{align}
     & \P_N^{(\Delta,\kappa)}(\theta,\theta') \label{REM2BIS} \\
\nonumber
\leq & \sum_{m\geq 0} \P\left(  Ev(r,1,E^{(\k_+)},0) ,\, \cup_{j\in [|\k_+,N-1|]}\mathtt{E}_m^{(j)}  \right) \\
\nonumber
     & \quad \quad \sup_{    l_{\k_+}^{(N)}-1  \leq z\leq u_{\k_+}^{(N)}+1  } \P\left( Ev(\k_+,1 +(m+1) 2^{-\frac{r}{400}},\mathtt{C}^{(\k)}+E^{(\k_+)},z)\right).
\end{align}
Here, by abuse of notation, we refer to the same event $Ev(k, a, E, z)$ as in Eq. \eqref{eq:def_Ev_GEv_ext} but for the new time clock $\tau_.^{(\k_+, \k_+)}$. By the same arguments used to prove \eqref{Gillet} we have: 
\begin{align}
\label{otoBIS}
\P\left(  Ev(r,1,E^{(\k_+)},0) ,\, \cup_{j\in [|\k_+,N-1|]}\mathtt{E}_m^{(j)}  \right) 
\ll_{\upsilon}  2^{r-N-(m^2/8)}.
\end{align}
On the other hand, by using the inequality \eqref{eq:girsanov2} deduced from the Girsanov transform (which still holds for the time clock $\tau^{(\k_+,\k_+)}_j =  \sum_{l=\k_+}^{j-1} \sum_{p=0}^{2^{\kappa^{(\k)}}-1} (2^{\kappa^{(\k)}} +p)^{-1}$), we obtain for any $z\in [ l_{\k_+}^{(N)}-1, u_{\k_+}^{(N)}+1 ]$:
\begin{align}
\nonumber &\P\left( Ev(\k_+,1+ (m+1) 2^{-\frac{r}{400}},\mathtt{C}^{(\k)}+E^{(\k_+)},z)\right) 
\\
\label{subielbiBIS}&  \ll_{\upsilon}
2^{\k_+  -N}  e^{  1+ 2(m+1) 2^{-\frac{r}{400}}} N^{\frac{3}{2}} e^{2z  } \P\left( GEv(\k_+,  1+ 2(m+1) 2^{-\frac{r}{400}},\mathtt{C}^{(\k)}+E^{(\k_+)},z) \right),
\end{align}
where we used that $\sup_{j\geq \k_+} |\mathtt{C}^{(\k)}_j|\leq 2^{-\kappa^{(\k)}} \leq (m+1) 2^{-\frac{r}{400}} $.  We bound the probability 
of the $GEv$ event by $1$ and use the fact that 
$z \leq 1+ u_{k_0}^{(N)}\leq 1+\upsilon -(N-\k_+)^{\alpha_-}-\frac{3}{4}\log N$.
We then get  
\begin{align*}
P\left( Ev(\k_+,  1+ (m+1) 2^{-\frac{r}{400}} ,\mathtt{C}^{(\k)}+E^{(\k_+)},z)\right) & \ll_{\upsilon}
2^{\k_+ -N}  e^{   2(m+1) 2^{-\frac{r}{400}}} N^{\frac{3}{2}} e^{2z  }
\\ & \ll_{\upsilon} 2^{\k_+  -N}  e^{2m+2}   N^{3/2} e^{- 2(N - \k_+)^{\alpha_-} -\frac{3}{2}\log N}
\\ & \ll 2^{\k_+  -N}  e^{2m+2} e^{-(N - \k_+)^{\alpha_-}},
\end{align*}
Finally, by combining  this equation with (\ref{otoBIS}) and (\ref{REM2BIS}), we get
\begin{align*}
\P_{N}^{(\Delta,\kappa)}(\theta, \theta')& \ll_{\upsilon}  2^{r-N} 2^{\k_+ -N}  e^{- (N-\k_+)^{\alpha_-}}   \sum_{m\geq 0}  e^{2m+2-(m^2/8)}
\ll 2^{r-N} 2^{\k_+ -N}  e^{- (N-\k_+)^{\alpha_-}} \ ,
\end{align*}
which concludes the proof of Lemma \ref{lemma:kMil}. 
\end{proof}

\begin{proof}[Proof of Lemma \ref{lemma:kFin}]
In this case, we simply ignore $R_{\e_j}^{(\e_{r},\Delta)}(\theta')$ in the event associated to $\P_N(\theta, \theta')$ (See Eq. \eqref{eq:def_P_N}).  Recalling that $(R_{\e_j}^{(\e_r,\Delta)}(\theta))_{j\geq r}$ is distributed like $  \sqrt{\frac{1}{2}}   (W_{\tau_j^{(r)}})_{j\geq r}  $, with $W$ a standard Brownian motion, we have:
\begin{align*}
              \P_N(\theta, \theta') 
\leq \quad & \P\left( \forall j\in [|r,N|],\,  L_j^{(N)}  \leq R_{\e_j}^{(\e_{r},\Delta)}(\theta) \leq U_j^{(N)} \right)\\
   = \quad & \P\left( Ev(r,0) \right)\\
\stackrel{Eq. \eqref{eq:girsanov1}}{\leq} &
             2^{r-N} e^{2 \upsilon} N^{\frac{3}{2}} \P\left( GEv(r, 0) \right)\\
\ll_\upsilon \quad & 2^{r-N} \ .
\end{align*}
where the last inequality comes from Corollary \ref{corollary:shifted_barrier_estimate}.  
\end{proof}

Now we are in position to prove the Proposition \ref{PrSecondmo}. 
\begin{proof}[Proof of Proposition \ref{PrSecondmo}]
Notice that for any $k\in \{1,...,N\}$,
$$\#\{ (\theta, \theta') \in ([0,2\pi)_{2^N})^2,\, 2^{-k} \leq \frac{||\theta-\theta'||}{2\pi}\leq 2^{1-k}  \}\ll  2^{2N-k}.$$
By applying the Lemmas \ref{lemma:kDebut}, \ref{lemma:ktoutDebut}, \ref{lemma:kMil} and \ref{lemma:kFin}, one obtains that for $r$ large enough and $N$ large enough depending on $r$, 
$$\E\left(\I_N^2\right)\leq \sum_{\theta; \theta' \in [0,2\pi)_{2^N},\, \frac{||\theta-\theta'||}{2\pi} \geq 2^{- \frac{r}{2}}  }   (1+\eta_{r,\upsilon})    2^{2r -2N} N^{3}   (\P(Event_{r,N}))^2  + Y$$
	where 
	\begin{align*}
	& Y \ll_{\upsilon}  \sum_{k\geq    r/2 ; k+3\Delta^{(k)}\leq r }    2^{2N-k}  
	 2^{2r - 2N}     + \sum_{  k+3\Delta^{(k)}\geq r; k\leq N/2 }    2^{2N-k}    2^{-2N+r+ k+3\Delta^{(k)}} e^{-( k+3\Delta^{(k)}-r)^{\alpha_-}} 
	\\
	&+ \sum_{k\geq N/2; k\leq N-r/2 }   2^{2N-k}  2^{-2N+ k + 3 \kappa^{(k)} +r}      e^{-  \frac{1}{2}(N-k - 3 \kappa^{(k)} )^{\alpha_-}}            +\sum_{k\geq N-r/2; k\leq N}  2^{2N-k}  2^{r-N}
\end{align*}
The first term of $\E [\I_{N}^2]$ is smaller than $(1+\eta_{r,\upsilon })    \left( 2^r N^{\frac{3}{2}}   \P(Event_{r,N})  \right)^2 $. The first sum in the estimate of $Y$  is smaller than 
\begin{align*}
	\sum_{k\geq    r/2 ; k+3\Delta^{(k)}\leq r }       2^{2r-k}    \ll 2^{3r/2}.
\end{align*}
If $k'$ denotes the smallest value of 
$k$ such that $ k+3\Delta^{(k)} \geq r$, we have $k' \geq r/2$ if $r$ is large enough, and then the second sum is smaller than
\begin{align*}
	&  \sum_{  k+3\Delta^{(k)}\geq r; k\leq N/2 }       2^{ r+3\Delta^{(k)}} e^{- (k + 3 \Delta^{(k)} -r)^{\alpha_-}}
	\leq  \sum_{k + 3\Delta^{(k)}\geq r}
	 2^{k + 6\Delta^{(k)}} e^{- (k + 3 \Delta^{(k)} -r)^{\alpha_-}}
	 \\ & \qquad \qquad \leq  2^{-k'} \sum_{s = r}^{\infty} 
	 2^{2s} e^{-(s-r)^{1/10}}
	 \leq 2^{3r/2}  \sum_{t = 0}^{\infty}
	 2^{2t} e^{-t^{1/10}} \ll 2^{3r/2}.
\end{align*}
In the third sum, we have
$$3 \kappa^{(k)} + r \leq 
1.03 r + 300 \log^2 (N - k), \; 
N - k - 3 \kappa^{(k)} 
\geq N - k - 0.03 r - 300 \log^2 (N-k),$$
and then for $r$ large enough and $N-k \geq r/2$, 
$$
\; N - k - 3 \kappa^{(k)} \geq 0.9 (N-k),$$
which gives
\begin{align*}
	&  \sum_{k\geq N/2; k\leq N-(r/2) }     2^{ 3\kappa^{(k)} +r}      e^{-  \frac{1}{2}(N- k - 3 \kappa^{(k)} )^{\alpha_-}}   
	\leq e^{1.03r} \sum_{1 \leq k \leq N-1} e^{300 
	\log^2 (N-k) - (1/2)(0.9 (N-k))^{1/10}}
	\ll e^{1.03r}.
\end{align*}
Finally, the last sum is smaller than
\begin{align*}
	\sum_{k\geq N-r/2; k\leq N}   2^{r+N-k}   \ll 2^{3r/2} .
\end{align*}
Hence, we have, by using Proposition \ref{Prmomen1Lower}, 
$$\E\left(\I_N^2\right)\leq  (1+\eta_{r,\upsilon})    2^{2r} N^{3}   (\P(Event_{r,N}))^2  + \mathcal{O}_{\upsilon}
 (2^{3r/2})
 \leq (\E
 (\I_N))^2 (1 + \eta'_{r, \upsilon})$$
 where 
 $$\eta'_{r, \upsilon} = (1+\eta_{r,\upsilon}) e^{4 \upsilon 
 + 4 \lambda_{\infty}^{(r)}}  - 1 +  
\mathcal{O}_{\upsilon}(2^{-r/2}).$$
Letting $r \rightarrow \infty$, $\eta'_{r, \upsilon}$ has an upper limit equal to $(1+ \limsup_{r \rightarrow \infty} \eta_{r, \upsilon} ) e^{4 \upsilon} - 1$, which tends to zero with $\upsilon$. This completes the proof of Proposition \ref{PrSecondmo}.
\end{proof}

\section{Appendix: Classical estimates on Gaussian walks}
\label{section:appendix}

In the following, the process $\left( W_s ; s \in \R_+ \right)$ is a standard Brownian motion. It will be mainly observed on a discrete set of times. Closely related to the Ballot theorem, a well-known result of Kozlov \cite{Koz76} states that there exists $c_0>0$ such that
\begin{align}
\label{Build0}
\lim_{N\to\infty }  \sqrt{N}\P\left( \forall j\in \{1,2, \dots, N \},\, W_j\geq 0 \right) = c_0 .
\end{align}
On the other hand, recall that for some $c>0$ large enough, for any $h\geq 0$, $N \geq 1$ one has,
\begin{align}
\label{build1}
\frac{1}{c} \left(\frac{h}{\sqrt{N}} \wedge 1 \right) \leq \sup_{x\in \R} \P\left( W_N \in [x,x+h]\right) \leq \frac{ch}{\sqrt{N}}.
\end{align}
Both identities are the building blocks of the following classical results. 

\begin{proposition}[Lemma A.1 in \cite{AShi10}]
For any $x\geq 0$, $b\geq a \geq 0$, 
\begin{align}
\label{eqA.1}
\P\left( \min_{j\in \{1,\dots, N\}} W_j + x \geq 0,\, W_N + x \in [a,b]  \right) \ll  (1+x) (1+(b-a)) (1+b) N^{-\frac{3}{2}} \ .
\end{align}
\end{proposition}


Now let state two extensions of (\ref{Build0}), (\ref{build1}) and (\ref{eqA.1}) which are essentially proved in \cite{bib:Mad13}.

\begin{lemma}
\label{deuxIneq}
For any $N\geq 2$, $z,h\geq 0$, $u\in \{0,\dots,N-1\}$ and for any event $A(u)$
measurable with respect to $\sigma (W_{u+s} - W_u)_{s \in [0,1]}$, 
\begin{align}
\label{eq:class2}
\P\left( \forall j \in \{1, \dots, N\},\, W_j\leq z,\, A(u)\right) &\ll  \frac{1+z}{\sqrt{N}}\sqrt{\P(A(u))},
\\
\label{eq:class3}
\P\left( W_N\in [z,z+h],\, A(u)\right) &\ll \frac{h}{\sqrt{N}}\P(A(u)). 
\end{align}
\end{lemma}
\begin{proof}
The proof is exactly similar to Lemma B.4 in \cite{bib:Mad13}, for $B = -W$. The only differences are that we take $j \in \{1, \dots, N\}$ instead of $j \in [0,N]$, and that $h$ is not necessarily equal to $1$. From the upper bound in \eqref{build1}, on can immediately deal with general $h \geq 0$ instead of $h = 1$. For discrete time $j \in \{1, \dots, N\}$, the ingredients we need in order to mimic the proof in \cite{bib:Mad13} are the following: 
\begin{equation}\P [W_1, \dots, W_k \leq z] \ll 
\frac{1+z}{\sqrt{k}}, \label{discretemax1}
\end{equation}
\begin{equation}
\E [(z - W_k) \mathds{1}_{W_1, \dots, W_k \leq z} ] \ll 1+z. \label{discretemax2}
\end{equation} 
The equation \eqref{discretemax2} is an immediate consequence of 
Lemma 2.3 in \cite{AShi11}. 
In order to prove \eqref{discretemax1}, we start by the case $z \leq 1$, for which we get, by using \eqref{Build0} and the Markov property at time $1$, 
$$k^{-1/2} \gg \P [W_1, \dots, W_{k+1} \leq 0]
\geq \P [W_1 \leq -1]
\P [W_1, \dots, W_k \leq 1]
\gg \P [W_1, \dots, W_k \leq z].$$
For $z > 1$, $m = \lceil z^2 \rceil$, we use the scaling property in order to deduce
\begin{align*}
\P [W_1, \dots, W_k \leq z]
& \leq \P \left[\forall j 
\in \{1, \dots, \lfloor k/m \rfloor \}, 
W_{mj} \leq \sqrt{m} \right]
\\ & =  \P \left[\forall j 
\in \{1, \dots, \lfloor k/m \rfloor \}, 
W_{j} \leq 1 \right]
 \ll ( \lfloor k/m \rfloor)^{-1/2} \wedge 1
\\ & \ll (  1+ \lfloor k/m \rfloor)^{-1/2}
\leq (k/m)^{-1/2} \leq \frac{(1+z^2)^{1/2}}{\sqrt{k}} \leq \frac{1+z}{\sqrt{k}}.
\end{align*}

\end{proof}
We deduce the following result (strongly related 
to Corollary B.5 in \cite{bib:Mad13}, see also the proof of Lemma A.1 in \cite{AShi10}).
\begin{corollary}
\label{corollary:barrier_estimate}
For any $N \geq 2$, $x\geq 0$, $b\geq a\geq 0,u\in \{0, 1, \dots, N-1\}$,   $y \in \mathbb{R}$  and for any event $A(u)$ measurable with respect to  $\sigma  \left(( W_{u+s} - W_u)_{s\in [0,1]} \right)$,
\begin{align}
\nonumber 
\P\left( \max_{j\leq N/2, j \in \mathbb{N}} W_{j} \leq x ,\,  \max_{N/2 \leq j\leq N, j \in \mathbb{N}} W_{j} \leq x+y ,\, W_N-y-x\in [-b,-a]  , A(u)    \right)\\
\ll N^{-\frac{3}{2}} (1+x)( b-a)(1+b)\sqrt{\P(A(u))} \ .
\end{align}
\end{corollary}
\begin{proof}
We can assume $N \geq 10$, otherwise the result is an immediate consequence of \eqref{eq:class3}. We define $N_1 = \lfloor N/3 \rfloor$, $N_2 = N - \lfloor N/3 \rfloor = N-N_1$, which implies that 
$N_1, N_2 - N_1, N - N_2$ are larger than $3$ and $N/4$. For $0 \leq u  <N_1$, we denote $A_1 := A(u)$, $A_2$ and $A_3$ being the full event (i.e. of probability $1$), for $N_1 \leq u < N_2$, 
we denote $A_2 := A(u)$, $A_1$ and $A_3$ being the full event, for $N_2 \leq u < N$,
 we denote $A_3 := A(u)$, $A_1$ and $A_2$ being the full event.

Hence, for $r \in \{1,2,3\}$, $A_r$ is measurable with respect to $\sigma\left( ( W_{v+s} - W_v )_{s \in [0,1]} \right)$ for some $v$ integer, 
$0 \leq v < N_1$ for $r = 1$, $N_1 \leq v < N_2$ for $r = 2$, $N_2 \leq v < N$ for $r = 3$.  
 
 The probability we want to estimate is at most: 
$$\P \left(\max_{j \leq N_1, j \in \mathbb{N}} W_j \leq x, A_1, \max_{j \leq  N_1, j \in \mathbb{N}} (W_{N-j} - W_N) \leq b, A_3, W_N \in [x+y - b, x+y - a], A_2 \right).$$
Let us first condition on $(W_s)_{s \leq N_1}$, 
$(W_{N_2 + s}- W_{N_2} )_{s \geq 0}$. This fixes if the first four events occur or not. 
The law of the increments of $W$ between $N_1$ and $N_2$ is not changed by the conditioning, and then the 
conditional probability that the two last events occur is at most: 
$$\sup_{w \in \mathbb{R}} \P [W_{N_2} - W_{N_1} \in [w - b, w-a], A_2] 
\ll \frac{b-a}{\sqrt{N}} \P[A_2]
,$$
the last estimate coming from \eqref{eq:class3},  applied to the Brownian motion 
$(W_{N_1+s} - W_{N_1} )_{s \in [0, N_2-N_1]}$. 

Hence, the probability to be estimated is dominated by 
$$\frac{b-a}{\sqrt{N}} \P[A_2] \, \P \left(\max_{j \leq N_1, j \in \mathbb{N}} W_j \leq x, A_1, \max_{j \leq  N_1, j \in \mathbb{N}} (W_{N-j} - W_N) \leq b, A_3 \right).$$
In the last probability, the two first events depend only on the increments of $W$ in $[0, N_1]$
whereas the two last depend only on the increments on $[N_2,N]$. Hence, be get 
$$\frac{b-a}{\sqrt{N}} \P[A_2] \, \P \left(\max_{j \leq N_1, j \in \mathbb{N}} W_j \leq x,   A_1\right) 
\P \left( \max_{j \leq  N_1, j \in \mathbb{N}} (W_{N-j} - W_N) \leq b, A_3 \right).$$
Using \eqref{eq:class2}, applied both for 
$(W_s)_{s \in [0, N_1]}$ and the Brownian motion 
$(W'_s := W_{N-s} -  W_N)_{s \in [0, N_1]}$ 
(note that $A_3$ is measurable 
with respect to $(W'_{v' + s}- W'_{v'})_{s \in [0,1]}$ for some $v' \in \{0, \dots, N_1 -1\}$), we get the bound
$$\frac{b-a}{\sqrt{N}} \P[A_2]  \,
 \frac{1+x}{\sqrt{N}} \sqrt{\P[A_1]}  \,
  \frac{1+b}{\sqrt{N}} \sqrt{\P[A_3]},$$
  which proves the corollary. 
\end{proof}

In the following we shall state and prove several results which are slight extensions of the previous ones.
\subsection{Estimates with a time error and a positive and curved barrier}
\begin{proposition}
\label{corPEMANt}
Let $f: \R_+ \rightarrow \R_+$ be any increasing function such that $f(z) \ll z^{\frac{1}{10}}$,  in particular $f(0) = 0$. 	Then for  $N\geq 1$, $z>0$
\begin{align}
\label{corPemantl}
\P\left( \forall j\leq N,\,  W_j\leq f(j)+z\right) \ll_f \frac{1+z}{\sqrt{N}}
\end{align}
\end{proposition}
\begin{proof}
Let us decompose the probability in \eqref{corPemantl}, depending on the time $k$ when the random walk $(W_j)_{j\leq N}$ reaches its maximum, and then use the Markov property at this time:
\begin{align*}
  & \P\left( \forall j\leq N,\,  W_j\leq f(j)+z\right)\\
= & \sum_{k=0}^{N} \P\left( \forall j\leq N,\,  W_j\leq f(j)+z,\, W_k= \max_{j\in [|0,N|]} W_j \right)\\
\leq  & \sum_{k=0}^{N} \P\left( \forall j\leq k,\,  W_j\leq f(j)+z,\, W_k = \max_{j\in [|0,k|]}W_j \right) \P\left(\forall j\leq N-k,\,  W_j \leq 0   \right) \\
\leq & \sum_{k=0}^{N} \P\left( W_k\leq f(k)+z,\, W_k = \max_{j\in [|0,k|]}W_j \right) \P\left(\forall j\leq N-k,\,  W_j \leq 0   \right) \ .
\end{align*}
By applying \eqref{Build0}, and the time reversal property of the random walk $(W_j)_{j\leq k}$, one gets
\begin{align*}
\P\left( \forall j\leq N,\,  W_j\leq f(j)+z\right) &\ll \sum_{k=0}^N (N-k+1)^{-1/2} \P\left(  \min_{j\in [|0,k|]}W_j \geq 0,\, W_k \leq f(k) +z \right)  .
\end{align*}
Given the bound we want to prove, there is no loss of generality in assuming $(1+z)^2 \leq \half N$. Moreover by \eqref{eqA.1}, we deduce that
\begin{align*}
    & \P\left( \forall j\leq N,\,  W_j\leq f(j)+z\right)\\
\ll &  N^{-\half}\sum_{1 \leq k \leq (1+z)^2}  \P\left( \max_{j\in [|0,k|]}W_j= W_k  \leq f(k) + z \right) 
          + \sum_{(1+z)^2< k \leq N/2}\frac{(1+ f(k)+z)^2}{k^{\frac{3}{2}}(N-k +1)^{\frac{1}{2}}  }\\
    & \quad \quad 
          + \sum_{N/2 < k \leq N}       \frac{(1+ f(k)+z)^2}{k^{\frac{3}{2}}(N-k +1)^{\frac{1}{2}}  }\\
\ll & N^{-\half} R\left( f((1+z)^2)+z \right) 
      + N^{-\half} \sum_{(1+z)^2< k \leq N/2} \frac{(1+z)^2 + f(k)^2}{k^{\frac{3}{2}}} \\
    & \quad \quad 
      + N^{-\frac{3}{2}} \sum_{N/2 < k \leq N} \frac{(1+z)^2 + f(k)^2}{(N-k +1)^{\frac{1}{2}} } \ .
\end{align*}
Here on the last line, we used $(a+b)^2 \leq 2(a^2 + b^2)$ and $R(u):= \sum_{k=1}^{+\infty}  \P\left(     \max_{j\in [|0,k|]}W_j= W_k  \leq u\right)$ is the renewal function associated to the random walk $(W_j)_{j \geq 0}$. It is well-known that $R(u) \ll (1+u) $.

 As $f(z) \ll z^{\frac{1}{10}}$, we have that
\begin{align*}
    & \P\left( \forall j\leq N,\,  W_j\leq f(j)+z\right)\\
\ll_f & N^{-\half} (1+z)
      + N^{-\half} \left( (1+z)^2 \sum_{k > (1+z)^2} k^{-\frac{3}{2}}
                        + \sum_{k > (1+z)^2} k^{\frac{2}{10}-\frac{3}{2}} \right)      
      + N^{-1}\left( (1+z)^2 + N^\frac{2}{10}\right)\\
\ll & \frac{(1+z)}{N^{\frac{1}{2}}} \ ,
\end{align*}
which concludes the proof of  \eqref{corPemantl}.
\end{proof}

Now, consider a real sequence $(e_j)_{j\geq 1}$ and $E_j:= \sum_{k=1}^j e_k$ the associated time error. In all the following, we assume that  $j+E_j$ is positive for all $j$ and increasing with respect to $j$ (this condition is always satisfied in the paper). Basically we will extend the previous lemma to the process $(W_{j+E_j})_{j \geq 0}$ when $||E||_1:= \sum_{j\geq 1}|e_j|$ is finite. 

\begin{lemma}
\label{lemma:shifted_barrier_estimate}
Assume that $||E||_1:=  \sum_{k=1}^{+\infty} |e_j| < \infty$. Let $f(j)$ be any increasing positive sequence such that $f(j) \ll j^{\frac{1}{10}}$. For any  $N\geq 1, a\geq 0$, $u \in \{0, 1, \dots, N-1\}$  and for any event $A(u)$ measurable with respect to 
$\sigma \left( (W_{u+E_u+ s} - W_{u+E_u})_{s \in [0,1+ e_{u+1}]} \right)$, 
\begin{align}
\label{class6bis}
\P\left( \min_{1 \leq j\leq N  }( W_{j+E_j} +f(j))\geq -a  \right)&  \ll_{f, ||E||_1}   \frac{1+a}{\sqrt{N}}  ,
 \\
\label{class6}
\P\left( \min_{1 \leq j\leq N  } W_{j+E_j} \geq -a  , \, A(u)\right)&  \ll_{ ||E||_1}   \frac{1+a}{\sqrt{N}}  \sqrt{\P(A(u))}  .
\end{align}
Moreover, it is possible to take implicit constants which are nondecreasing with respect to $||E||_1$ (for 
fixed $f$ in the first estimate): in other 
words, for $c > 0$, if we assume $||E||_1 \leq c$, the we can replace the dependence in $||E||_1$ by a dependence in $c$. 
\end{lemma}
\begin{proof}
In all this proof, each time we write the symbol $\ll_{||E||_1}$ or $\ll_{f,||E||_1}$, we assume an implicit
constant nondecreasing in $||E||_1$. 
Without loss of generality, we can suppose that  $N \geq 10$. We start by the proof of \eqref{class6bis}. Suppose that $(e_j)_{j\geq 0}\in \R_+^{\N}$. 
We then check (for example by computing the covariances) that if $(\widetilde{W}_s)_{s \geq 0}$ is a Brownian motion, independent of $W$, then $(W_{j+E_j})_{1 \leq j \leq N}$ has the same law as $(W_j + \widetilde{W}_{E_j})_{1 \leq j \leq N}$: note that this identity in law depends on the fact that $E_j$ increases with $j$. 

Let $S$ be the supremum of $\widetilde{W}$
on the interval $[0, ||E||_1]$. 
For any $\delta >0$, upon partitioning our event with respect to the disjoint union $\bigsqcup_{k \geq 0} \left\{ \delta k \leq S < \delta (k+1) \right\}$, 
\begin{align*}
           & \P\left( \min_{1 \leq j\leq N} (W_{j+E_j} +f(j)) \geq -a    \right) \\
\leq \quad & \P\left( \min_{1 \leq j\leq N}( W_{j} + \widetilde{W}_{E_j} +f(j))\geq  - a \right) \\
\leq \quad &  \P\left( \min_{1 \leq j\leq N}( W_{j} + S +f(j))\geq  - a \right) \\ 
\leq \quad & \sum_{k\geq 0} \P\left( \min_{1 \leq j\leq N}( W_{j} + f(j))\geq -\delta (k+1) - a  \right) \P\left( S \geq \delta k\right)\\
\stackrel{Eq. \ \ref{corPemantl} }{\ll_f} & \frac{1}{\sqrt{N}} \sum_{k\geq 0} (1+\delta (k+1)+a) e^{-\frac{(\delta k)^2}{2 ||E||_1}} \ .
\end{align*}	
By taking $\delta = \sqrt{||E||_1} $, we get the desired result. Now consider the general case $(e_j)_{j\geq 0} \in \R^{\N}$. Let $\widetilde{W}$ be a standard Brownian motion independent of $W$ and let $|E|_j:= \sum_{k=1}^{j} |e_k|$. The process $(W_{j+ E_j+ |E|_j})_{j\geq 0}$ has the same law as $\left( W_{j+E_j} + \widetilde{W}_{|E|_j} \right)_{j \in \N}$ and $e_j+|e_j|\geq 0$, for all $j$. Moreover, by the study of the case $e_j\in \R_+$, one has
\begin{align*}
  \frac{1}{\sqrt{N}}(1+a) \gg_{f,||E||_1}  & \P\left( \min_{1\leq j\leq N} \left( W_{j+E_j} +f(j)+ \widetilde{W}_{|E|_j}  \right) \geq  -a -(||E||_1)^{1/2} \right)\\
\geq & \P\left( \min_{1\leq j\leq N} \left( W_{j+E_j} +f(j)+ \widetilde{W}_{|E|_j} \right) \geq  -a -(||E||_1)^{1/2}, \ 
                \sup_{j\geq 0} |\widetilde{W}_{|E|_j}| \leq (||E||_1)^{1/2}
        \right)\\
\geq & \P\left( \min_{1\leq j\leq N}( W_{j+E_j} +f(j)) \geq -a \right) \P\left( \sup_{j\geq 1} |\widetilde{W}_{|E|_j}| \leq (||E||_1)^{1/2}  \right),
\end{align*}
which leads to
\begin{align}
\label{eq:hortie}
\P\left( \forall 1\leq j\leq N, \ (W_{j+E_j} + f(j) ) \geq -a  \right) & \ll_{f,||E||_1} \frac{1}{\sqrt{N}}(1+a).
\end{align}

Let us tackle the proof of (\ref{class6}). If $u\geq \frac{N}{2}$, by the Markov property at time $u$, we have:
\begin{align*}
     \P\left( \min_{1 \leq j\leq N}  W_{j+E_j}\geq -a , A(u)\right) 
\leq \P\left( \forall 1 \leq j\leq u,\,  W_{j+E_j}\geq -a \right) \P(A(u))
\ll_{||E||_1} \frac{1+a}{\sqrt{N}} \sqrt{\P(A(u))}.
\end{align*} 
If $u\leq \frac{N}{2}$, by the Markov property at time $u+1$ and by applying  Eq. \eqref{eq:hortie} to the Brownian motion 
$(W'_s = W_{u+1+E_{u+1} + s} - W_{u+1+ E_{u+1}})_{s \in [0, N-u-1+ E_N - E_{u+1}]}$ and the sequence $(e_{u+1+j})_{j \geq 1}$,
 we get
\begin{align*}
& \P\left( \min_{j\leq N}W_{j+E_j}\geq -a , A(u)\right)\\
& \leq  \E\left( \mathds{1}_{\{   \underset{{j\leq u+1}}{\min} W_{j+E_j}\geq -a  \ ; \ A(u)\}} 
                 \P\left(\min_{1 \leq j\leq N -u-1}  W'_{j+E_{u+1+j}-E_{u+1}}\geq -a-z \right)_{\big| z= W_{u+1+E_{u+1}}}  \right)\\
& \ll_{||E||_1} \frac{1}{\sqrt{N}}  \E\left( (1 + a+ W_{u+1+E_{u+1}})  \mathds{1}_{\{\forall  j\leq u+1,\,  W_{j+E_j}\geq -a   \}}  \mathds{1}_{A(u)}  \right) .
\end{align*}
By noticing that $W_{u+1+ E_{u+1}}\leq W_{u+E_u} + \max_{s\leq 1+ e_{u+1}}|W_{u+s+E_{u}} -W_{u+E_u}|$, we deduce that 
\begin{align*}
      \P\left( \min_{j\leq N   } W_{j+E_j}\geq -a , A(u)\right)
& \ll_{||E||_1} \frac{1}{\sqrt{N}}  \E\left( (1+a+ W_{u+E_{u}}) \mathds{1}_{\{  \min_{j\leq u  } W_{j+E_j}\geq -a \}}  \right) \P(A(u))\\
& \qquad \qquad +  \frac{1}{\sqrt{N}}  \E\left( \max_{s\leq 1+ e_{u+1}}|W_{u+s+E_{u}} -W_{u+E_u}| \mathds{1}_{A(u)}  \right)\\
& \ll_{||E||_1} \frac{1+a}{\sqrt{N}}\sqrt{\P(A(u))},
\end{align*}
where we have used the Cauchy-Schwarz inequality and the fact that $1+e_{u+1} \leq 1+ ||E||_1$ to bound the second term. For the first term, it is sufficient to show: 
\begin{equation}
\label{discretemaximum}
\E\left( (1+a+ W_{u+E_{u}}) \mathds{1}_{\{  \min_{j\leq u  } W_{j+E_j}\geq -a \}}  \right)
\ll_{||E||_1} 1+a, 
\end{equation}
which is proven as follows. 
If $T_a$ denotes the smallest integer $j$ such that 
$W_{j+E_j} < -a$ (necessarily $T_a \geq 1$), we get, by the martingale property:  
\begin{align*}
& \E [(1+ a+ W_{u+E_u}) \mathds{1}_{T_a \leq u} ] 
  = \E [(1+ a+W_{T_a  + E_{T_a}})  \mathds{1}_{T_a \leq u} ]
 \\ & = \E [(1+ a + W_{T_a-1  + E_{T_a-1}})  \mathds{1}_{T_a \leq u} ] +  \E [(W_{T_a + E_{T_a}} - W_{T_a-1 + E_{T_a - 1}} ) \mathds{1}_{T_a\leq u} ].
 \end{align*}
 In the last expression, the first term is nonnegative, and the second is nonpositive. Hence 
 $$ \E [(1+a+ W_{u+E_u}) \mathds{1}_{T_a \leq u} ] \geq  \E  [W_{T_a + E_{T_a}} - W_{T_a -1+ E_{T_a-1}}].$$ 
whereas 
$$\E [1+a+ W_{u+E_u}] = 1+a,$$
which gives \eqref{discretemaximum} by taking the difference, provided that  
$\E [W_{T_a-1 + E_{T_a-1}} - W_{T_a + E_{T_a}}] \ll_{||E||_1} 1+a$.
Now, for $A > 0$, 
\begin{align*}
\P [W_{T_a-1 + E_{T_a-1}}&  - W_{T_a + E_{T_a}} \geq A] 
 \leq \sum_{p = 0}^{\infty} 
\P \left[\sup_{j \in \{0, 1, \dots, 2^{p+1}-1\}}
W_{j+E_j} - W_{j+1+E_{j+1}} \geq A, T_a \geq 2^{p} \right] 
\\ & \leq \sum_{p = 0}^{\infty}
\left(\P \left[\sup_{j \in \{0, 1, \dots, 2^{p+1}-1\}}
W_{j+E_j} - W_{j+1+E_{j+1}} \geq A \right] \right)^{1/2} 
\left(\P [T_a \geq 2^p] \right)^{1/2}
\\ &\stackrel{Eq. \ \ref{class6bis} }{\ll_{||E_||_1}} \sum_{p=0}^{\infty} 
\left(1  \wedge 2^{p+1} e^{-\frac{A^2}{2(1+||E||_1)}} \right)^{1/2} ( (1+a)  2^{-p/2} )^{1/2}
\\ & \ll_{||E||_1} \sum_{0 \leq p \leq A^2/(2 (1+||E||_1)\log 2)}
 e^{-\frac{A^2}{4(1+||E||_1)}} (1+a)^{1/2} 2^{p/4} 
\\ & + \sum_{p  \geq A^2/(2(1+||E||_1) \log 2)} (1+a)^{1/2} 
2^{-p/4} 
\\ & \ll_{||E||_1} e^{-\frac{A^2}{8(1+||E||_1)} } (1+a)^{1/2},
\end{align*}
which gives the desired bound.

\end{proof}

By the same arguments used to deduce Corollary \ref{corollary:barrier_estimate} from Equations (\ref{eq:class2}) and (\ref{eq:class3}), one can deduce from Lemma \ref{lemma:shifted_barrier_estimate}:
\begin{corollary}
\label{corollary:shifted_barrier_estimate}
Let $f(j)$ be any increasing positive sequence such that $f(j) \ll j^{\frac{1}{10}}$ and then $f(0)=0$. Assume that $||E||_1 < \infty$.

\begin{enumerate}

\item For any $N \geq 100(1+||E||_1)$, $x\geq 0$, $b\geq a\geq 0,  \lambda \in (0.1,0.9), y\in \R$,
\begin{align}
\nonumber 
\P\left( \max_{1 \leq j\leq \lambda N} ( W_{j+E_j} -f(j))\leq x,\, \right.  & \left. \max_{\lambda N\leq j\leq N} (W_{j+E_j} -f(N-j))\leq x+y ,\, W_{N+E_N} \in x + y - [a,b]     \right)\\
\label{class7bis}
& \ll_{f,||E||_1} N^{-\frac{3}{2}} (1+x)( b-a)(1+b).
\end{align}

\item For any $N \geq 100 (1+||E||_1)$, $x\geq 0$, $b\geq a\geq 0,
u \in \{0, \dots, N-1\}$, $\lambda \in (0.1,0.9),  y\in \R$  and for any event $A(u)$ measurable with respect to $\sigma \left( 
(W_{u+E_u + s} - W_{u+E_u})_{s \in [0,1+e_{u+1}]} \right)$,
\begin{align}
\nonumber 
\P\left( \max_{1 \leq j\leq \lambda N}  W_{j+E_j} \leq x,\,  \right. & \left. \max_{\lambda N\leq j\leq N} W_{j+E_j}  \leq x+y ,\, W_{N+E_N}-y-x\in [-b,-a]  , A(u)   \right)\\
\label{class7}
& \ll_{||E||_1} N^{-\frac{3}{2}} (1+x)( b-a)(1+b)\sqrt{\P(A(u))} \ .
\end{align}
\end{enumerate}
Again, we can take implicit constants which are nondecreasing in $||E||_1$. 
\end{corollary}

\begin{proof}
Inequalities (\ref{class7bis}) and (\ref{class7}) can be obtained similarly by combining (\ref{class6bis}) and (\ref{build1}) for the first one and (\ref{class6}) and (\ref{eq:class3}) for the second one. 
 
Let us show the inequality (\ref{class7bis}). 
If $N \geq 20$ and $N_1 := \lfloor \lambda N \rfloor$, 
we have $N_1 \in (0.05\,N, 0.95 \, N)$.  
 By the Markov property at time $N_1 + E_{N_1}$, one has:
\begin{align*}
     & \P\left( \max_{1 \leq  j\leq \lambda N} W_{j+E_j} -f(j) \leq x, \ 
                \max_{\lambda N\leq j\leq N} (W_{j+E_j} -f(N-j))\leq x+y , \ 
                W_{N+E_N}-y-x\in [-b,-a] \right) \\
\leq & \E\left( \mathds{1}_{\{ \max_{1 \leq j\leq N_1} ( W_{j+E_j} -f(j))\leq x  \}} \P\left( \max_{ 1 \leq j\leq N - N_1} \left(W'_{j+E_{j}^{(N_1)}}-f(N - N_1 -j)  \right)+z \leq x+y ,\, \right.\right. \\
     &  \qquad \qquad \qquad  \qquad \qquad \qquad \qquad \left. \left. W'_{N-N_1+E_{N-N_1}^{(N_1)}} +z-y-x\in [-b,-a]  \right)_{\big| z=W_{N_1+ E_{N_1}}} \right) \ 
\end{align*}
with for any $p,k\geq 0$, $E_p^{(k)}= E_{p+k}-E_k$, and for $s \geq 0$, $W'_s = W_{N_1+ E_{N_1} + s} - W_{N_1 + E_{N_1}}$. 

 Let us introduce  
$$(\widetilde{W}_{j+\widetilde{E}_j })_{ 0\leq j \leq N-N_1}:= \left( W'_{N-N_1+E_{N-N_1}^{(N_1)}}  -  W'_{ N-N_1-j+E_{ N-N_1 -j}^{(N_1)}}  \right)_{ 0\leq j\leq N-N_1},$$
the time reversal walk of $(W'_{j+E_{j}^{(N_1)}})_{0 \leq j\leq N-N_1}$. One can easily check that $(\widetilde{W}_{j+\widetilde{E}_j})_{0 \leq j \leq N-N_1}$ is again a standard Gaussian random walk with a time error $\widetilde{E}_j$ satisfying $||\widetilde{E}||_1 \leq ||E||_1$ and $\widetilde{E}_0=0$.  For any $z\in \R$ we have
\begin{align*}
& \P\left( \max_{ 1 \leq j\leq N-N_1} \left(W'_{j+E_{j}^{(N_1)}}-f(N-N_1 -j)  \right)+z \leq x+y ,\,   W'_{N-N_1+E_{N-N_1}^{(N_1)}} +z-y-x\in [-b,-a]  \right)\\
&\leq   \P\left(    \min_{1 \leq j\leq N-N_1} ( \widetilde{W}_{j+\widetilde{E}_j} +f(j))\geq -b,\,   \widetilde{W}_{N-N_1+\widetilde{E}_{N-N_1}} +z-y-x\in [-b,-a]  \right)\\
&\leq \P\left(   \min_{1 \leq j\leq  N_2} ( \widetilde{W}_{j+\widetilde{E}_j} +f(j))\geq -b\right) \sup_{t\in \R}\P\left( W_{ 
N-N_1 - N_2 + \widetilde{E}_{N-N_1} - \widetilde{E}_{N_2} } +t\in [-b,-a]  \right) \ ,
\end{align*}
where $N_2 := \lfloor (N-N_1)/2 \rfloor$, and where 
 in the last line we use the Markov property at time $  N_2 + \widetilde{E}_{N_2}$. Observe that this last expression does not depend in $z$ any more.  Finally we have obtained that 
\begin{align}
& \P\left( \max_{1 \leq j\leq \lambda N} ( W_{j+E_j} -f(j))\leq x,\,  \max_{\lambda N\leq j\leq N} (W_{j+E_j} -f(N-j))\leq x+y ,\, W_{N+E_N}-y-x\in [-b,-a] \right) \nonumber \\
& \leq \P\left(  \max_{1 \leq j\leq N_1} ( W_{j+E_j} -f(j))\leq x  \right)   \P\left(  \min_{1 \leq j \leq N_2} ( \widetilde{W}_{j+\widetilde{E}_j} +f(j))\geq -b\right) \times \nonumber \\
& \qquad\qquad\qquad\qquad\qquad \sup_{t\in \R}\P\left( W_{N-N_1-N_2 + E}  + t \in [-b,-a]  \right) \,  \label{troisfacteurs}
\end{align}
where $|E| \leq||E||_1$. 
For $N \geq 100(1 + ||E||_1)$, we have  
$N_1, N_2, N-N_1 - N_2 + E \gg N$. 
By applying (\ref{class6bis}) and (\ref{build1}), one gets
\begin{align*}
& \P\left( \max_{j\leq \lambda N} ( W_{j+E_j} -f(j))\leq x,\,  \max_{\lambda N\leq j\leq N} (W_{j+E_j} -f(N-j))\leq x+y ,\, W_{N+E_{N}}-y-x\in [-b,-a]     \right)\\
& \qquad\qquad\qquad\qquad\qquad \ll_{f,||E||_1} N^{-\frac{3}{2}} (1+x)( b-a)(1+b),
\end{align*}  
with an implicit constant nondecreasing in $||E||_1$. 
This concludes the proof of (\ref{class7bis}). 

The proof of (\ref{class7}) is similar, except 
that we remove the function $f$ and we add the event $A(u)$. After doing the computations, 
we get a similar product of three factors as in \eqref{troisfacteurs}, without $f$, and with the 
event $A(u)$ added in the first factor if $0 \leq u < N_1$, in the third factor if $N_1 \leq u < N- N_2$ and in the second factor if $N-N_2 \leq u < N$: note that in the two last cases, $A(u)$ is measurable with respect to the $\sigma$-algebra generated by $(\widetilde{W}_{N-u-1 + 
\widetilde{E}_{N-u-1} + s} -
\widetilde{W}_{N-u-1 + 
\widetilde{E}_{N-u-1}} )_{s \in [0, 1 + \widetilde{E}_{N-u} - \widetilde{E}_{N-u-1}]}$.
Using (\ref{class6}) if $0 \leq u < N_1$ or $N-N_2 \leq u < N$, or (\ref{eq:class3}) if $N_1 \leq u < N-N_2$, we can estimate the factor involving $A(u)$ in a suitable way in order to deduce the result of the corollary. 
\end{proof}

\subsection{Lower bound for the probability to stay in an entropic envelope}
Recall that, with the notation of the upper bound section, 
 \begin{align*}
 u_k^{(N)}:=\left\{ \begin{array}{ll} \upsilon- (k-r)^{\alpha_-},\qquad &\text{if  }r\leq  k\leq \lfloor N/2\rfloor,
 \\
 \upsilon - (N-k)^{\alpha_-} -   \frac{3}{4}  \log N,\qquad &\text{if  } \lfloor N/2\rfloor <k \leq N,
 \end{array} \right. 
 \end{align*}
 and
 \begin{align*}
 l_k^{(N)}:=\left\{ \begin{array}{ll} -\upsilon- (k-r)^{\alpha_+},\qquad &\text{if  }r\leq  k\leq \lfloor N/2\rfloor,
 \\
 -\upsilon - (N-k)^{\alpha_+} -    \frac{3}{4} \log  N,\qquad &\text{if  } \lfloor N/2\rfloor <k \leq N.
 \end{array} \right.
 \end{align*}
We now have the following lower bound: 
\begin{proposition} \label{LemmaA3}
	For any $L > 0$, for any sequence $(a_N)_{N \geq 1}$ of non-negative numbers such that  $\sup_{N\to\infty} \frac{a_N}{N^{\frac{1}{2}}} \leq  L$,  for any $C > 0$, and for any $N'$ integer with 
	$2N/3> N' > N/3$, 
	\begin{align}
\P\left( \min_{j \in \{1, \dots, N'\}} W_j \geq 0,\, \min_{j \in \{N'+1, \dots, N\}} W_j\geq a_N,\, a_N \leq W_{N} \leq a_N +C \right) \gg_{L, C}  N^{-3/2}.
	\end{align}
	for $N$ large enough depending only on $L$ and $C$. 
\end{proposition}
\begin{proof}
This result is very similar to Lemma A.3 in \cite{AShi10}, and the same proof works. In particular one checks that the implicit constant we can obtain depends only on $C$ and $L$ and  that the fact that $N'$ is not exactly $N/2$ but only in $(N/3, 2N/3)$ does not change anything to the arguments we need. Note that $C$ can be taken arbitrarily since the distribution of $W_1$ is non-lattice (see the proof of Lemma A.2 in \cite{AShi10}). 
\end{proof}
Now, we will estimate the probability that a Brownian motion stays inside an entropic envelope. 
\begin{proposition}
	\label{lemma:2sided_barrier_estimate}
	Fix $d,\upsilon,r >0$. Then for all $N\geq 2 r$,
	\begin{align}
	\label{eq:barrier_estimate}
 \P\left( \forall j\in \{r,r+1, \dots,N\}, \,   l_j^{(N)}\leq  d W_{j-r} \leq u_j^{(N)}\right)
 \gg_{d,\upsilon,r} N^{-3/2}.
	\end{align}
	Moreover, if $N$ is assumed to be large enough depending on $d$ and $r$, then 
	we can remove the dependence in $r$ in the implicit constant of the estimate. 
\end{proposition}
\begin{proof}

For each value of $N \geq 2r$, the probability to be estimated is strictly positive, which implies that it is sufficient to prove the second part of the proposition, relative to the case where $N$ is large enough depending on $d$ and $r$.   For any $ 0 \leq A < \frac{N}{2} - r$, by the Markov property, the probability is bigger than the product of the following three terms:
\begin{align*}
p_1&:= \P\left( \forall j\leq \{r, \dots, A+r\},\, l_j^{(N)}\leq d W_{j-r}\leq u_j^{(N)},\, u_{A+r}^{(N)}- A^{\frac{1}{2}} - \frac{1}{2}  \leq d  W_{A}\leq u_{A+r}^{(N)}-A^{\frac{1}{2}}\right)
\\
p_2&:=   \inf_{u_{A+r}^{(N)}-A^{\frac{1}{2}} -\frac{1}{2} \leq x \leq u_{A+r}^{(N)}- A^{\frac{1}{2}}  } \P\left( \forall j\in \{A+r,\dots,N-A\},\,  l_j^{(N)} \leq d W_{j-A-r}+ x  \leq u_j^{(N)},\, \right.
\\
&\qquad \qquad\qquad\qquad\qquad\qquad\qquad\qquad \left.   u_{N-A}^{(N)} -A^{\frac{1}{2}}-1 \leq d W_{N-2A-r} +x \leq u_{N-A}^{(N)}-A^{\frac{1}{2}} \right) ,
\\
p_3&:=  \inf_{ u_{N-A}^{(N)}-A^{\frac{1}{2}}-1 \leq  z\leq  u_{N-A}^{(N)}-A^{\frac{1}{2}} }\P\left(   \forall j\in \{N-A, \dots, N\},\,   l_j^{(N)}\leq d W_{j-(N-A)}+ z \leq u_j^{(N)}  \right).
\end{align*}
If $A \geq 10$, we have 
$$A^{9/10} - A^{1/10} 
\geq  A^{1/2} (10^{0.4} -  10^{-0.4}) \geq  2 A^{1/2} > A^{1/2} + 1,$$
 which implies that $p_1$ and $p_3$ are strictly positive. Moreover, they do not depend on $r$ and $N$ if $A$ and $\upsilon$ are fixed: for $p_1$, we check this fact by shifting $j$ by $-r$, for $p_3$, by shifting $z$, $l_j^{(N)}$, $u_j^{(N)}$ by $(3/4) \log N$, and $j$ by $-(N-A)$. 
 
Now, by first fixing $y = (3/4) \log N$, one can check that for any $A\geq 100$,  $p_2$ is bigger than 
\begin{align*}
&\tilde{p}_2=  \inf_{0 \leq y \leq 1+ \log N} \P\left(   \forall  j\leq  N'_A, j \in \mathbb{N}, \,  - (j^{\frac{9}{10}}+A^{\frac{9}{10}})/10    \leq d W_{j}   \leq - j^{\frac{1}{10}} +A^{\frac{1}{2}}  ,\,\right.
\\
&\left. \forall j\in \left(N'_A,N_A \right] \cap \mathbb{N} ,\,  \right. \\ & \left. - ((N_{A}-j)^{\frac{9}{10}}+A^{\frac{9}{10}})/10  \leq d W_{j} +y\leq - (N_{A}-j)^{\frac{1}{10}} + A^{\frac{1}{2}},\,    -\frac{1}{2} \leq d W_{N_{A}} +y    \leq 0 \right)
\end{align*}
with $N_{A}:= N-r -2A$ and $N'_A := \lfloor N/2 \rfloor -r-A$. 
 For $b>0$, $1 \leq N' < N$ and $y\in \R$, let us introduce:
\begin{align*}
\mathfrak{E}_{b,N,N',y}&:= \left\{ \max_{j \in \{1, \dots, N'\}} d W_j \leq b ,\, \max_{j \in \{N'+1, \dots N\}} d  W_j \leq -y +b,\,d  W_N \geq -y - \frac{1}{2}   \right\},
\\
\mathfrak{E}^{(c)}_{b,N,N',y}&:=  \left\{   \forall j \in \{1, \dots, N'\},d  W_{j}   \leq - j^{\frac{1}{10}} +b ,\, \right. \\ & \left.  \forall j\in \{N'+1, \dots, N\},\,  d W_{j} +y\leq - (N-j)^{\frac{1}{10}} + b,\,     -\frac{1}{2} \leq d W_{N} +y\leq 0 \right\}.
\end{align*} 
It is plain to check that $\tilde{p}_2$ is bigger than
\begin{align*}
\inf_{0 \leq y \leq 1+ \log N} \left\{\P\left(  \mathfrak{E}^{(c)}_{A^{\frac{1}{2}},N_{A},N'_A,y}    \right) - \sum_{k=1}^{ N'_A}  \P\left( \mathfrak{E}_{A^{\frac{1}{2}},N_A, N'_{A},y} ,\,d W_k\leq -(k^{\frac{9}{10}}+A^{\frac{9}{10}} )/10 \right) \qquad\qquad\qquad \right.
\\
\left.-  \sum_{ N'_A< k\leq N_{A} } \P\left( \mathfrak{E}_{A^{\frac{1}{2}},N_{A},N'_A,y} ,\, d W_k +y \leq -((N_{A}-k)^{\frac{9}{10}}+A^{\frac{9}{10}})/10  \right)  \right\}.
\end{align*}
Thus  Proposition \ref{lemma:2sided_barrier_estimate} will be proved once the following two assertions are shown:
\\

i) for any $\epsilon>0$, for $B>0$ large enough depending only on $d$ and  $\epsilon$,  for $2N/3> N' > N/3 > 10 B$, $0 \leq y  \leq 2 + \log N$,
\begin{align*}
&\sum_{k=1}^{N'}  \P\left( \mathfrak{E}_{B,N,N',y} ,\, d W_k\leq -(k^{\frac{9}{10}}+B^{9/5})/10 \right) \\ &   +  \sum_{k= N' + 1}^{N} \P\left( \mathfrak{E}_{B,N,N',y} ,\, d W_k +y\leq -((N-k)^{\frac{9}{10}}-B^{9/5})/10 \right) \leq \frac{\epsilon}{N^{3/2}}.
\end{align*}

ii) There exists $c>0$ depending only on $d$ such that for any $B > 0$ large enough depending  on $d$, $2N/3 > N' > N/3 >10 B$ and for $0 \leq y \leq 2 + \log N$,
\begin{align*}
\P\left(    \mathfrak{E}^{(c)}_{B,N,N',y}  \right) \geq \frac{c}{N^{3/2}}.
\end{align*}

Indeed, if i) and ii) hold, we get the desired bound by taking $\epsilon = c/2$ (depending only on $d$), $B > 10$ large enough (depending only on $d$) in order to have the conclusions of i) and ii), $A = B^2$, and a value of $N$ in the proposition which is large enough (possibly depending only on $d$ and $r$, since one can take $A$ and $B$ depending only on $d$) in order to ensure that $2N_A/3 > N'_A > N_A/3 > 10B$ and  $1 + \log N \leq 2 + \log N_A$.

Let start the proof of assertion i). Fix $\epsilon >0$. By applying the Markov property at time $k\leq N'$, one has
\begin{align*}
&\P\left( \mathfrak{E}_{B,N,N',y} ,\, dW_k\leq -(k^{\frac{9}{10}}+B^{9/5})/10 \right)
\\
&\leq  \sum_{p\geq 0} \P\left(   dW_k\leq -(k^{\frac{9}{10}}+B^{9/5})/10 -p \right) 
\\ & \times \P_{-k_{B,p}} \left( \max_{j\leq N'-k} d W_j\leq  B,\, \max_{N'-k < j \leq N-k} d W_j \leq -y +B,\, dW_{N-k} \geq -y - \frac{1}{2}  \right),
\end{align*}
with $k_{B,p}=(k^{\frac{9}{10}}+B^{9/5})/10  +p+1$. When $k\leq N'/2$ it stems, by using Corollary
\ref{corollary:shifted_barrier_estimate} (with  $E_j = 0$), that for $N$ large enough (which is ensured by $B$ large enough),
\begin{align*}
\P\left( \mathfrak{E}_{B,N,N',y} ,\, d W_k\leq \right. & \left. -(k^{\frac{9}{10}}+B^{9/5})/10 \right) \ll  \sum_{p\geq 0} e^{ - \frac{k^{1.8} + B^{3.6} + p k^{0.9}}{200 k d^2}} 
 \left(1 + \frac{1+ B+k_{B,p}}{d} \right)^3 \frac{1}{N^{\frac{3}{2}} }
 \\ &  \ll \frac{1}{N^\frac{3}{2}} 
 e^{- \frac{k^{0.8} + B^{3.6} k^{-1} }{200 d^2} } \sum_{p \geq 0}  e^{- \frac{p k^{-0.1}}{200 d^2} }  \left( 1 + \frac{1 + B + k^{0.9} + B^{1.8}  + p + 1 }{d} \right)^3.
\end{align*}
We have 
$$ k^{0.8} + B^{3.6} k^{-1} \geq (k^{0.8})^{2/3} ( B^{3.6} k^{-1})^{1/3} 
= k^{0.2} B^{1.2},$$ 
and then 
\begin{align*}
\P\left( \mathfrak{E}_{B,N,N',y} ,\, d W_k\leq -(k^{\frac{9}{10}}+B^{9/5})/10 \right) 
& \ll_d N^{-3/2} 
e^{-\frac{k^{0.2} B^{1.2}}{200 d^2} }
(1 + k^{2.7}  + B^{5.4} ) \sum_{p \geq 0} 
(1+ p^3) e^{-\frac{p}{200 d^2 k}}
\\ & \ll_d N^{-3/2} e^{-\frac{k^{0.2} B^{1.2}}{200 d^2}} (1+ k^{2.7})(1 + B^{5.4}) (k  + k^4)
\\ & \ll_d  N^{-3/2} e^{-\frac{k^{0.2} B^{1.2}}{200 d^2}} 
k^{6.7} B^{5.4} \ll_d N^{-3/2} e^{-\frac{k^{0.2} B^{1.2}}{400d^2}}.
\end{align*}

When $N'/2 < k \leq N'$, one can simply write: 
\begin{align*} 
\P\left( \mathfrak{E}_{B,N,N',y} 
,\, dW_k
\leq - (k^{9/10}+B^{9/5})/10 \right)
 &\ll  e^{- \frac{k^{1.8} + B^{3.6}}{200 k d^2}}
 \ll e^{-\frac{k^{0.2} B^{1.2}}{200 d^2}}
 \\ & \ll_d  k^{-3/2}
 e^{-\frac{k^{0.2} B^{1.2}}{400 d^2}}
 \ll N^{-3/2}
 e^{-\frac{k^{0.2} B^{1.2}}{400 d^2}}.
 \end{align*}
A combination of these two inequalities gives:
\begin{align*}
\sum_{k=1}^{N/2}  \P\left( \mathfrak{E}_{B,N,N',y} ,\, d W_k\leq -(k^{\frac{9}{10}}+B^{9/5} )/10  \right) \ll_d
N^{-3/2} \sum_{k=1}^{\infty}  e^{-\frac{k^{0.2} B^{1.2}}{400 d^2}}  \leq \frac{\epsilon}{N^{\frac{3}{2}}}
\end{align*}
if $B$ is large enough depending on $d$ and $\epsilon$. 
 
The second sum (for $k > N'$) can be treated similarly. Indeed, by  operating a time reversal one has, for $k' = N-k$, 
\begin{align*}
&\P\left( \mathfrak{E}_{B,N,N',y} ,\, dW_k +y\leq -((N-k)^{\frac{9}{10}}+B^{9/5})/10 \right)
\\
&\leq  \P\left(   \min_{1 \leq j < N-N'} d W_j\geq -B- \frac{1}{2},\, \min_{N-N' \leq j < N} dW_j \geq -B-y - \frac{1}{2}, \right. \\ & \left.      dW_N \leq B - y,\,  d W_{k'} \geq  \frac{(k')^{\frac{9}{10}} +B^{9/5}}{10} - \frac{1}{2} \right).
\end{align*}
which can be bounded with the same arguments. 

In order to prove the assertion ii), let 
us set:  \begin{align*}
k_i= \left\{  \begin{array}{ll} i^{\frac{1}{10}},\qquad &\text{if }  i\leq N',
\\
(N-i)^{\frac{1}{10}}+y,\qquad &\text{if }  N' <i\leq N,
\end{array}   \right. 
\end{align*}
  We notice that: 
\begin{align*}
&\P\left(    \mathfrak{E}^{(c)}_{B,N,N',y}  \right) \geq \P\left( \max_{1 \leq j\leq N'} d W_j\leq 0,\, \max_{N' \leq j \leq N} dW_j\leq -y,\, d W_{N}+y\geq - \frac{1}{2} \right)
\\
& - \P\left( \max_{1 \leq j\leq N'} d W_j\leq 0,\, \max_{N' \leq j \leq N} d W_j+y \leq 0,\, d W_{N}+y\geq -\frac{1}{2},\, \exists i\in [|1,N|],\, d W_i\geq - k_i + B\right).
\end{align*} 
From Proposition \ref{LemmaA3}, applied with $C = 1/2$, and $L = d^{-1} \sup_{N \geq 1} N^{-1/2}(2 + \log N)$, and all sequences $(a_N)_{N \geq 1}$ with  $0 \leq a_N \leq 2 + \log N$, we deduce, uniformly in $y \in [0, 2 + \log N]$, 
$$ \P\left( \max_{1 \leq j\leq N'} d W_j\leq 0,\, \max_{N' \leq j \leq N} dW_j\leq -y,\, d W_{N}+y\geq - \frac{1}{2} \right)
\gg_d N^{-3/2},$$
for $N$ large enough depending only on $d$, which is guaranteed if $B$ is large enough depending on $d$.
On the other hand, by a version of \cite{Aid11}, Lemma B.3 which allows $N'$ to be in $(N/3, 2N/3)$ instead of being exactly $N/2$, and which can be proven exactly in the same way, we get for all $\epsilon > 0$, for $B > 0$ large enough depending only on $d$ and $\epsilon$, 
\begin{align*}
& \P\left( \max_{1 \leq j\leq N'} d W_j\leq 0,\, \max_{N' \leq j \leq N} d W_j+y \leq 0,\, d W_{N}+y\geq -\frac{1}{2},\, \exists i\in [|1,N|],\, d W_i\geq - k_i + B\right)
 \\ & \ll_d \epsilon N^{-3/2} + \mathcal{O}_{d, \epsilon} (N^{-1.7}).
 \end{align*}
 By taking $\epsilon$ small enough depending only on $d$, and then $B >0 $ large enough, we deduce the assertion ii), which ends the proof of Proposition \ref{lemma:2sided_barrier_estimate}. 
\end{proof}

\subsection{Asymptotic equivalence of barriers}
For any $N \geq 2r \geq 0$, $d>0$, $\delta \geq 0$, 
$(e_k)_{k \geq 1}$, $E_k := \sum_{j=1}^{k} e_j$, such that $k + E_k$ is increasing in $k$,  
 we define the event of our Gaussian random walk staying within two barriers:
\begin{align*}
	G_0v(d,\delta,E) & := \left\{   \forall j\in [|r, N |] ,\, l_j^{(N)}-\delta \leq d W_{j-r + E_{j-r}} \leq u_j^{(N)}+\delta    \right\}.
\end{align*}
Note that $G_0v(d,\delta,E)$ depends implicitly on $\upsilon, r$ and $N$. 
The following proposition enhances the estimate in Proposition \ref{lemma:2sided_barrier_estimate} by showing that small shifts in space or time have little influence:
\begin{proposition}
	\label{PropProclas1}
	Let $\upsilon, d > 0$. The following equality holds for all $r \geq 1$: 
	\begin{align}
	\label{ProClas0}
	\lim_{\delta\to 0} \sup_{N \geq 2r} \left| \frac{\P\left(  G_0v(d,\delta,0) \right) }{\P\left( G_0v(d,0,0)  \right) } - 1 \right|
	=
	\lim_{\delta, ||E||_1\to 0} \sup_{N \geq 2r} \left| \frac{\P\left( G_0v(d,\delta,E)    \right) }{\P\left( G_0v(d,0,0)   \right) } - 1 \right|
	= 0,
	\end{align}
	where we recall that $||E||_1 = \sum_{k=1}^{\infty} e_k$. Moreover, there exists a function $g$
	from $ \mathbb{R}_+^*  \times \mathbb{N}$ to $\mathbb{N}$, 
	such that $g(d,r) \geq 2r$ for all $d > 0$, $r \geq 1$, and: 
	\begin{align}
	\label{ProClas1}
	\lim_{\delta\to 0} \sup_{r \geq 1} 
	\sup_{N \geq g(d,r)} \left| \frac{\P\left(  G_0v(d,\delta,0) \right) }{\P\left( G_0v(d,0,0)  \right) } - 1 \right|
	=
	\lim_{\delta, ||E||_1\to 0} \sup_{r \geq 1} 
	\sup_{N \geq g(d,r)} \left| \frac{\P\left( G_0v(d,\delta,E)    \right) }{\P\left( G_0v(d,0,0)   \right) } - 1 \right|
	= 0. 
	\end{align}
\end{proposition}
\begin{proof}
We first prove the first limit in
 \eqref{ProClas0} and \eqref{ProClas1}. For any $\delta >0$, we have
\begin{align*}
	& \P\left(   G_0v(d,\delta,0)  \right) = \P\left(    G_0v(d,0,0) \right)
	+ \P\left( G_0v(d,\delta,0);\,\exists j_0\in [|r,N|],\, d W_{j_0-r}\leq l_{j_0}^{(N)} \text{or}\, d W_{j_0-r}\geq u_{j_0}^{(N)} \right).
\end{align*}
By Proposition \ref{lemma:2sided_barrier_estimate}, it suffices to prove
 the following two estimates for $\delta \in [0,1/2]$: 
\begin{align}
\label{etoil1}
\sup_{N \geq 2r} N^{\frac{3}{2}} 
\sum_{j_0=r}^{N}  \P\left( G_0v(d,\delta,0),\, d W_{j_0-r}\geq u_{j_0}^{(N)} \right)
\ll_{r, d} \delta (1+ \upsilon)^4 \ ,\\
\label{etoil2} 
\sup_{N \geq 2r} N^{\frac{3}{2}} 
\sum_{j_0=r}^{N}  \P\left( G_0v(d,\delta,0),\, d W_{j_0-r}\leq l_{j_0}^{(N)} \right) 
\ll_{r, d} \delta (1 + \upsilon)^4 \ ,
\end{align}
where the dependence in $r$ in the implicit constants can be removed, provided that $N$ is large enough depending on $d$ and $r$. 
The control of the dependence in $\upsilon$ of the estimates will be useful at the end of the proof of the proposition.  
We can now directly assume that $N$ is large enough depending on $d$ and $r$.  
 Indeed, all probabilities involved in the sums are bounded by the probability that $d W_s$ is in an interval of size $\delta$, for some $s \geq 1$, and then they are $\ll_d \delta$. This implies that for any function $g$ as in the proposition, the supremums 
in \eqref{etoil1} and \eqref{etoil2} restricted to $2r \leq N \leq g(d,r)$ are $\ll_d \delta g(r,d)^{5/2}  \ll_{r,d} \delta$ if the function $g$ is fixed.

Thanks to Corollaries \ref{corollary:barrier_estimate} and \ref{corollary:shifted_barrier_estimate} (with $f = E = 0$) and the Markov property, for any $\delta \in [0,1/2]$, and $N$ large enough depending only on $r$, the probability in \eqref{etoil1} is smaller than the following quantities: 

i)  when  $j_0\in [r, \frac{N}{3}]$,
\begin{align*}
&\P\left(  \forall j\in [|r,j_0|],\, d W_{j-r}\leq \delta+\upsilon ,\,  u_{j_0}^{(N)}+\delta \geq d W_{j_0-r}\geq u_{j_0}^{(N)}\right) \times 
\\
& \P\left(  \max_{j\leq (N/2)-j_0} d W_j\leq \delta,\, \max_{(N/2)-j_0 < j\leq N-j_0} d W_j \leq (j_0-r)^{1/10} -\frac{3}{4}\log  N + \delta,  \right. \\ 
& \left. \, \qquad \qquad  d W_{N-j_0} \geq  (j_0 - r)^{1/10} -\frac{3}{4}\log  N - 2 \delta   - 2 \upsilon \right)
\\ & \ll_d \left( \frac{\delta(1+ \delta + \upsilon)( 1+ \delta + (j_0-r)^{1/10})}{1 + (j_0-r)^{3/2}} \right) \left( \frac{(\delta + \upsilon)(1+ \delta) (1+ \delta + \upsilon)}{(N-j_0)^{3/2}} \right) 
\\ & \ll_{d} \delta (1+ \upsilon)^3 N^{-3/2} (1+ j_0-r)^{-1.4}
\end{align*}
ii) when $j_0\in [ \frac{N}{3},\frac{2N}{3}]$,
\begin{align*}
&\P\left(  \forall j\in [|r,j_0|],\, d W_{j-r}\leq \delta+\upsilon ,\,  u_{j_0}^{(N)}+\delta \geq  d W_{j_0-r}\geq u_{j_0}^{(N)}\right) \times \\
&   \qquad \qquad \P\left(  \max_{j\leq N-j_0} d W_j\leq  - u_{j_0}^{(N)} - \frac{3}{4} \log N 
 + \upsilon + \delta,\,   d W_{N-j_0} \geq   - u_{j_0}^{(N)} -\frac{3}{4}\log  N - 2 \delta   -  \upsilon  \right)
\\
& \ll_d  \left( \frac{\delta(1 + \delta + \upsilon)(1 + \delta + N^{1/10})}{N^{3/2}} \right)
\left( \frac{(1 + N^{1/10} + \upsilon + \delta)^3}{N^{3/2}} \right) 
 \ll_{d} \delta (1 + \upsilon)^4 N^{-2.6}.
\end{align*}

iii) When $j_0\in [ \frac{2N}{3},N]$, by using the time reversal property of the random walk $(W_j)_{0 \leq j\leq N-r}$, i.e $(W_j)_{0 \leq j\leq N-r} \eqlaw ( W_{N-r}- W_{N-r-j})_{0 \leq j\leq N-r}$, the probability in \eqref{etoil1} is equal to the probability of the intersection of the following three events:
\begin{align*}
&\left\{\forall j < \frac{N}{2},\,    -\upsilon +j^{1/10}-\delta \leq d W_j -\left(\frac{3}{4}\log N + dW_{N-r} \right) \leq \upsilon +j^{9/10} + \delta \right\},
\\
&\{d W_{N-j_0} - \left(  \frac{3}{4}\log N + dW_{N-r}  \right)\leq -\upsilon  +(N-j_0)^{1/10}  \},
\\
&\left\{ \forall j\in \left[\frac{N}{2},\, N-r \right],\,  -\upsilon +(N-r-j)^{1/10}-\delta\leq dW_j-d W_{N-r} \leq \upsilon +(N-r-j)^{9/10}+\delta   \right\}.
\end{align*}
We know that $-\upsilon -\delta   \leq \frac{3}{4}\log N +d W_{N-r} \leq \upsilon +\delta $ on this intersection (take $j =0$). For $a \in [-\upsilon-\delta, \upsilon+ \delta]$, if we restrict the intersection with the event  $a \leq \frac{3}{4}  \log N  +d  W_{N-r} \leq a + \delta$, what we get is included in the following intersection: 
\begin{align*}
&\left\{\forall j < \frac{N}{2},\,   a -\upsilon +j^{1/10}-\delta \leq d W_j  \leq a + \upsilon +j^{9/10} +  2 \delta \right\},
\\
&\{ d W_{N-j_0} \leq a + \delta -\upsilon  +(N-j_0)^{1/10}  \},
\\
&\left\{ \forall j\in \left[\frac{N}{2},\, N-r \right],\,- \frac{3}{4} \log N + a -\upsilon +(N-r-j)^{1/10}-\delta\leq d W_j,  \right. \\ & \left. d W_j \leq  - \frac{3}{4} \log N + a + \upsilon +(N-r-j)^{9/10}+ 2 \delta, \; 
a- \frac{3}{4} \log N \leq dW_{N-r} \leq 
a+ \delta -   \frac{3}{4} \log N \right\}.
\end{align*}
By doing similarly as in case i), and by using the fact that $|a| \leq \upsilon + \delta$, we get that the probability of this intersection is at most: 
\begin{align*}
&\P\left(  \forall j\in [|0,N-j_0|],\, d W_{j}\geq a - \upsilon - \delta ,\right. \\ & \left.  \qquad a - \upsilon + (N-j_0)^{1/10} - \delta  \leq d W_{N-j_0} \leq  a - \upsilon + (N-j_0)^{1/10} + \delta \right) \times 
\\
& \P\left(  \min_{j < j_0 - (N/2)} d W_j\geq - 2 \delta,\, \min_{ j_0 - (N/2) \leq j \leq  j_0 - r} d W_j \geq -  (N-j_0)^{1/10} -\frac{3}{4}\log  N - 2 \delta,  \right. \\ 
& \left. \, \qquad \qquad  -(N-j_0)^{1/10} -\frac{3}{4}\log  N - \delta   + \upsilon  \leq 
 d W_{j_0-r} \leq  - (N-j_0)^{1/10} -\frac{3}{4}\log  N + 2 \delta   + \upsilon \right)
\\ & \ll_d \left( \frac{\delta(1+ \delta + \upsilon)( 1+ \delta + \upsilon + (N-j_0)^{1/10})}{1 + (N-j_0)^{3/2}} \right) \left( \frac{\delta (1+ \delta) (1+ \delta + \upsilon)}{(j_0-r)^{3/2}} \right) 
\\ &  \qquad \qquad  \ll_{d} \delta^2 (1+ \upsilon)^3 N^{-3/2} (1+ N-j_0)^{-1.4}
\end{align*}
By adding these estimates for $a = - \upsilon - \delta + k \delta$, $0 \leq k \leq \lfloor 2 (\upsilon + \delta)/\delta \rfloor$, we deduce that the probability in \eqref{etoil1} is 
$\ll_{d} \delta(1+ \upsilon)^4 N^{-3/2} (1+ N-j_0)^{-1.4}$. 

Finally, by adding the estimates of all the terms in \eqref{etoil1}, we get 
\begin{align*}
	& \sum_{j_0=r}^{N}  \P\left( G_0v(d,\delta,0),\,  d W_{j_0-r}\geq u_{j_0}^{(N)} \right)   \ll_{d}    \sum_{ r \leq j_0 \leq N/3} \delta (1+ \upsilon)^3 N^{-3/2} (1 + j_0 - r)^{-1.4} \\ & 
	+ \sum_{ N/3 \leq j_0 \leq 2N/3} \delta(1+ \upsilon)^4 N^{-2.6}  + \sum_{ 2N/3 \leq j_0 \leq N}
	\delta (1+ \upsilon)^4 N^{-3/2} (1 + N-j_0)^{-1.4}
	\ll \delta (1+ \upsilon)^4 N^{-3/2}, 
\end{align*}
which proves \eqref{etoil1} (without dependence in $r$ if $N$ is large enough depending on $d$ and $r$). 

Now we shall prove \eqref{etoil2}. By the Markov property, and Corollaries \ref{corollary:barrier_estimate} and 
\ref{corollary:shifted_barrier_estimate}, when $N$ is large enough depending on $r$, and  $r < j_0\leq \frac{N}{3}$, the probability we want to estimate is at most: 
\begin{align*}
	&\P\left( d W_{j_0 - r} \in [ - (j_0-r)^{9/10} - \upsilon - \delta, - (j_0-r)^{9/10} - \upsilon] \right) \\ & \times 
 \P\left(  \max_{ j \leq (N/2)-j_0} d W_j \leq 
(j_0 - r)^{9/10} + 2 \upsilon + 2 \delta, 
\right. \\ & \left.
\max_{(N/2)-j_0 < j \leq N-j_0} d W_j 
\leq  (j_0 - r)^{9/10} + 2 \upsilon + 2 \delta
- \frac{3}{4} \log N, 
dW_{N-j_0}  \geq (j_0 - r)^{9/10} - \delta 
- \frac{3}{4} \log N
\right)
\\ & \ll_d \frac{\delta}{d \sqrt{j_0-r}}e^{- \frac{(j_0 - r)^{1.8}}{2 d^2(j_0-r)}} \, 
\frac{(\upsilon + \delta)(1+ \upsilon + \delta)( 1 + \upsilon + \delta 
+ (j_0 - r)^{9/10})}{(N - j_0)^{3/2}} 
\\ & \ll_{d}(1+ \upsilon)^3 \delta e^{- \frac{(j_0 - r)^{4/5}}{2 d^2}} (j_0 - r)^{9/10} N^{-3/2}.
\end{align*}
When $j_0 \in [\frac{2N}{3}, N]$, we can use the time-reversal property of the Brownian motion. 
We get the probability of the intersection of the following events: 
\begin{align*}
&\left\{\forall j < \frac{N}{2},\,    -\upsilon +j^{1/10}-\delta \leq d W_j -\left(\frac{3}{4}\log N + dW_{N-r} \right) \leq \upsilon +j^{9/10} + \delta \right\},
\\
&\{d W_{N-j_0} - \left(  \frac{3}{4}\log N + dW_{N-r}  \right)\geq \upsilon  +(N-j_0)^{9/10}  \},
\\
&\left\{ \forall j\in \left[\frac{N}{2},\, N-r \right],\,  -\upsilon +(N-r-j)^{1/10}-\delta\leq dW_j-d W_{N-r} \leq \upsilon +(N-r-j)^{9/10}+\delta   \right\}.
\end{align*}
If we restrict this intersection with the event 
$a \leq \frac{3}{4} \log N + d W_{N-r} \leq a + \delta$ for $a \in [-\upsilon-\delta, \upsilon + \delta]$, we get something included in the intersection of: 
\begin{align*}
&\left\{\forall j < \frac{N}{2},\,   a -\upsilon +j^{1/10}-\delta \leq d W_j  \leq a + \upsilon +j^{9/10} +  2 \delta \right\},
\\
&\{ d W_{N-j_0} \geq a  + \upsilon  +(N-j_0)^{9/10}  \},
\\
&\left\{ \forall j\in \left[\frac{N}{2},\, N-r \right],\,- \frac{3}{4} \log N + a -\upsilon +(N-r-j)^{1/10}-\delta\leq d W_j, \right. \\ & \left.  dW_j \leq  - \frac{3}{4} \log N + a + \upsilon +(N-r-j)^{9/10}+ 2 \delta, \; 
a- \frac{3}{4} \log N \leq dW_{N-r} \leq 
a+ \delta -   \frac{3}{4} \log N \right\}.
\end{align*}
The probability of this intersection is at most: 
\begin{align*}
	&\P\left( d W_{N-j_0} \in [  (N-j_0)^{9/10}+ a +  \upsilon , (N-j_0)^{9/10}+ a +  \upsilon + 2 \delta] \right) \\ & \times 
 \P\left(  \min_{ j < j_0 - (N/2)} d W_j \geq 
-(N- j_0)^{9/10} - 2 \upsilon  - 3 \delta, 
\right. \\ & \left.
\max_{j_0 - (N/2) < j \leq j_0-r} d W_j 
\geq  -(N-j_0)^{9/10} -2 \upsilon - 3 \delta
- \frac{3}{4} \log N,  \right.  \\ & 
\left.
\qquad  -(N-j_0)^{9/10} - \upsilon 
- \frac{3}{4} \log N - 2 \delta \leq  dW_{j_0 - r}  \leq -(N-j_0)^{9/10} - \upsilon 
- \frac{3}{4} \log N + \delta 
\right).
\end{align*}
If $2N/3 \leq j_0 \leq N-1$, we get a quantity dominated by 
\begin{align*}
& \frac{\delta}{d \sqrt{N-j_0}}e^{- \frac{((N - j_0)^{9/10} - \delta)_+^2 }{2 d^2(N-j_0)}} \, 
\frac{\delta(1+ \upsilon + \delta)( 1 + \upsilon + \delta 
+ (N - j_0)^{9/10})}{(j_0 - r)^{3/2}} 
\\ & \ll_{d} \delta^2 (1+ \upsilon)^2 e^{- \frac{(N-j_0)^{4/5}}{8 d^2}} (N-j_0)^{9/10} N^{-3/2}.
\end{align*}
If $j_0 = N$, we get a probability equal to zero for 
$a \in (- \upsilon, \upsilon + \delta]$ (since $W_0 = 0$) and
dominated by 
\begin{align*} 
\frac{\delta(1+ \upsilon + \delta)( 1 + \upsilon + \delta)}{(N-r)^{3/2}} \ll_{d} \delta (1+ \upsilon)^2 N^{-3/2}
\end{align*}
for $a \in [-\upsilon - \delta, - \upsilon]$. 
Hence, by adding the estimates for $a = - \upsilon - \delta +  k \delta$, $0 \leq k \leq \lfloor 2 (\upsilon + \delta)/\delta \rfloor$, 
we get, for all $j_0 \in [2N/3, N]$, a probability which is 
$$\ll_{d} \delta (1+ \upsilon)^3 e^{- \frac{(N-j_0)^{4/5}}{8 d^2}} (1 + N-j_0)^{9/10} N^{-3/2}.$$
Finally when $j_0\in [\frac{N}{3}, \frac{2N}{3}]$ we simply observe that for $N \geq 12 r$, 
which implies that $l_{j_0}^{(N)} \leq - (N/3 - r)^{9/10} 
\leq - (N/4)^{9/10}$, 
\begin{align*}
&	\P\left(d W_{j_0-r}\leq l_{j_0}^{(N)},\,   G_0v(d,\delta,0) \right)  \leq \P 
	\left(l_{j_0}^N   - \delta  \leq d W_{j_0 - r} \leq l_{j_0}^{N} 
	\right)
	\\ & \ll \frac{\delta}{d \sqrt{j_0 - r}} e^{- \frac{(N/4)^{1.8}}{2(j_0 - r) d^2}} \ll_d \delta e^{- \frac{(N/4)^{1.8}}{2N d^2} }
	\leq \delta e^{-\frac{N^{4/5}}{25 d^2}}. 
\end{align*}
Adding the previous estimates, we deduces that for $N$ large enough depending only on $r$,  the sum in \eqref{etoil2} is 
\begin{align*}
 & \ll_{d} 
\sum_{r < j_0 \leq N/3} \delta (1+ \upsilon)^3 e^{- \frac{(j_0-r)^{4/5}}{2 d^2}}(j_0 - r)^{9/10} N^{-3/2} + \sum_{N/3 \leq j_0 \leq 2N/3}\delta  e^{-\frac{N^{4/5}}{25 d^2}}
\\ &  \qquad + \sum_{2N/3 \leq j_0 \leq N}
\delta (1+ \upsilon)^3 e^{- \frac{(N-j_0)^{4/5}}{8 d^2}} (1 + N -j_0)^{9/10} N^{-3/2} \ll_{d,\upsilon} \delta (1+ \upsilon)^3 N^{-3/2},
\end{align*}
which concludes the proof of \eqref{etoil2}, and thus the proof of the first limit in \eqref{ProClas0} and  \eqref{ProClas1}.

\paragraph{}It remains to prove the second limit in \eqref{ProClas1}. By using the first limit, it suffices to prove that for any 
 $N\geq r$, we have
\begin{align*}
	\P\left(  G_0v(d,0,0)   \right)(1 - \theta_1)\leq \P\left(   G_0v(d,\delta ,E)   \right) \leq\P\left(   G_0v(d, \delta + \theta_2,0) \right)(1+ \theta_3),
\end{align*}
where $\theta_1, \theta_2, \theta_3 \geq 0$ can depend on $\upsilon, d, \delta, ||E||_1$ but not on $r$ and $N$, as soon as $N$ is large enough depending on 
$d$ and $r$, and tend to zero for $\upsilon, d$ fixed and $\delta, ||E||_1 \rightarrow 0$. 

We first prove the right-hand side inequality. Let us first assume that $e_j\geq0$ for all $j \geq 1$. Let $\widetilde{W}$ be a standard Brownian motion, independent of $W$. Observe that $(W_{j+E_j})_{j\leq N} \eqlaw \left( W_j + \widetilde{W}_{E_j}  \right)_{j \leq N}$. It follows that for any $\delta, \theta_2 >0$, 
\begin{align}
\P\left( G_0v(d,\delta,E)\right) & = \P\left( \forall j\in [r,N],\, l_j^{(N)} -\delta \leq  d W_{j-r} + d \widetilde{W}_{E_{j-r}}  \leq u_j^{(N)} +\delta \right) \nonumber \\
&\leq  \sum_{k = 1}^{\infty}  \P\left(  G_0v(d,\delta + k \theta_2,0) \right) \P\left( \sup_{j\geq 0} d|\widetilde{W}_{E_j}| \geq \theta_2 (k-1)\right)
\nonumber \\
&\leq  \P\left( G_0v(d,\delta+ \theta_2,0)  \right) + 
\mathcal{O}_{\upsilon,d} \left(  N^{-\frac{3}{2}} \sum_{k\geq 1} (1+(\delta +\theta_2) k)^3 e^{-\frac{(\theta_2 k)^2}{2 d^2 ||E||_1}} \right), \label{Gov}
\end{align}
where in the last inequality we used Corollaries \ref{corollary:barrier_estimate} and 
\ref{corollary:shifted_barrier_estimate} to bound $\P\left(  G_0v(1,\delta+ k \theta_2,0) \right) $: 
note that in these corollaries, we use the fact that $N$ is large enough depending on $r$, since 
the length of the two parts of the trajectory are $ \lfloor N/2 \rfloor - r$ and $N - \lfloor N/2 \rfloor$. 
If we take $\theta_2 = ||E||_1^{1/3}$, we deduce 
$$\P\left( G_0v(d,\delta,E)\right)
\leq \P\left( G_0v(d,\delta+ \theta_2,0)  \right) ( 1+ \theta_3),$$
where
$$\theta_3 \ll_{\upsilon,d}
N^{-3/2} \left[ \P\left( G_0v(d,\delta+ \theta_2,0)  \right) \right]^{-1} \sum_{k \geq 1}(1+(\delta +||E||_1^{1/3}) k)^3 e^{-\frac{k^2}{2 d^2 ||E||_1^{1/3}}}.$$
By dominated convergence, the last sum tends to zero when $\delta, ||E||_1$ go to zero, whereas  by \eqref{eq:barrier_estimate},
 
$$ N^{-3/2} \left[ \P\left( G_0v(d,\delta+ \theta_2,0)  \right) \right]^{-1}
\leq  N^{-3/2} \left[ \P\left( G_0v(d,0,0)  \right) \right]^{-1} \ll_{\upsilon, d} 1.$$
Hence, we can majorize $\theta_3$ by a quantity depending only on  $\upsilon, d, \delta, ||E||_1$ and
tending to zero with $\delta$ and $||E||_1$. Now, let us extend the majorization 
to the general case, for which $e_j$ can be negative. 
As for Lemma \ref{lemma:shifted_barrier_estimate}, we  introduce $|E|_j:= \sum_{k=1}^{j}|e_k|$. 
We know that 
$(W_{j+E_j+|E|_j})_{j \geq 0}$ has the same law as $(W_{j+E_j} + \widetilde{W}_{|E|_j})_{j \geq 0}$. Hence, 
\begin{align*}\P (G_0v(d,\delta + ||E||_1^{1/3} ,E + |E|) ) & \geq \P( G_0v(d,\delta ,E))
\P \left( \sup_{j \geq 0}  d|\widetilde{W}_{|E|_j}| \leq ||E||_1^{1/3} \right),
\end{align*}
and then 
$$ \P( G_0v(d,\delta ,E)) 
\leq \P (G_0v(d,\delta + ||E||_1^{1/3} ,E + |E|) ) (1+ \theta_4),$$
where 
$$\theta_4 := \frac{\P \left( \sup_{s \in [0,1]}|W_s|
\geq d^{-1}||E||_1^{-1/6} \right)} 
{\P \left( \sup_{s \in [0,1]}|W_s|
\leq d^{-1}||E||_1^{-1/6} \right)}.$$
depends only on $d$ and $||E||_1$ and tends to zero with $||E||_1$. 
We then deduce the majorization we want from the case where $\delta$ and $E$ are replaced by $\delta + ||E||_1^{1/3}$ and $E + |E|$. 


It remains to prove the left-hand side inequality, which is deduced from the case $\delta = 0$. By taking $\theta_1 = 1$ for 
$||E||_1 \geq ((1\wedge\upsilon)/4)^3$, we can assume 
$||E||_1 < ((1\wedge\upsilon)/4)^3$. The inequality we want to prove can be rewritten as follows: 
$$\P (G_0v (d,0,0) )
(1-\theta_1) \leq \P (G_0v_{||E||^{1/3}_1}(d,||E||^{1/3}_1,E) ),$$
$G_0v_{\delta}(d,\delta',E)$  denotes the event obtained from $G_0v(d,\delta',E)$ by changing the implicit value of $\upsilon$ to $\upsilon - \delta$. We first assume that $e_j \geq 0$ for all $j \geq 1$. In this case, let us prove the 
 slightly stronger estimate: 
$$\P (G_0v (d,0,0) )(1-\theta_1)
\leq  \P (G_0v_{2 ||E||^{1/3}_1}(d,||E||^{1/3}_1,E) ) = \P (G_0v_{||E||^{1/3}_1}(d,0,E) ).$$
We have
\begin{align*}
\P (G_0v_{2||E||^{1/3}_1}(d,||E||^{1/3}_1,E) )
& \geq  \P (G_0v_{2||E||^{1/3}_1}(d,0,0) )
\P \left( \sup_{j \geq 0} d |\widetilde{W}_{E_j}| \leq ||E||_1^{1/3} \right)
\\ & \geq \P (G_0v_{2||E||^{1/3}_1}(d,0,0) \, \P \left( \sup_{s \in [0,1]}|W_s| \leq d^{-1} ||E||_1^{-1/6} \right), 
\end{align*}
which shows that it is sufficient to prove an equality of the form
$$\P (G_0v_{2||E||^{1/3}_1}(d,0,0))
\geq \P (G_0v(d,0,0)) (1 - \theta_1)
= \P (G_0v_{2||E||^{1/3}_1}(d,2||E||_1^{1/3},0))(1-\theta_1).$$
From the equations \eqref{etoil1} and \eqref{etoil2}, valid since $2||E||^{1/3}_1 \leq 1/2$,  we get for $N$ large 
enough depending on $d$ and $r$, 
\begin{align*}\P (G_0v_{2||E||^{1/3}_1}(d,2||E||_1^{1/3},0)) & = \P (G_0v_{2||E||^{1/3}_1}(d,0,0)) + \mathcal{O}_{d} (||E||_1^{1/3} N^{-3/2} (1+ \upsilon  - 2||E||^{1/3})^4)
\\ & =  \P (G_0v_{2||E||^{1/3}_1}(d,0,0))
+ \mathcal{O}_{d,\upsilon} (||E||_1^{1/3} N^{-3/2}).
\end{align*}
Note that in this estimate, we use the control of the dependence in $\upsilon$  in \eqref{etoil1} and \eqref{etoil2}.  On the other hand, since 
$2||E||_1^{1/3} \leq \upsilon/2$,  we have by 
\eqref{eq:barrier_estimate} (with $\upsilon$ changed to $\upsilon/2$),  
$$\P (G_0v_{2||E||^{1/3}_1}(d,0,0))
\geq \P (G_0v_{\upsilon/2}(d,0,0)) \gg_{d,\upsilon} N^{-3/2},$$
and then 
$$\P (G_0v_{2||E||^{1/3}_1}(d,2||E||_1^{1/3},0)) \leq \P (G_0v_{2||E||^{1/3}_1}(d,0,0)) ( 1 + \mathcal{O}_{d, \upsilon} 
(||E||_1^{1/3})),$$
which gives the desired bound. 

Finally, we remove the assumption that $e_j \geq 0$ for all $j \geq 1$. In this case, we can take $\theta_1 = 1$ for  $3||E||_1 \leq   ((1 \wedge \upsilon)/4)^3$ and then assume that 
$3||E||_1 <   ((1 \wedge \upsilon)/4)^3$.
We have
\begin{align*} 
&\P (G_0v_{||E||^{1/3}} (d, 0,
 E + 2 |E|))  \\ &  \leq 
 \sum_{k=1}^{\infty}
 \P (G_0v_{||E||_1^{1/3}} (d, k ||E||_1^{1/3}, E))  \, \P \left( \sup_{j \geq 0} 
  d |\widetilde{W}_{2 |E|_j}| \geq ||E||_1^{1/3} (k-1)  \right).
 \end{align*}
 Moreover, the computation given in \eqref{Gov}
 implies 
\begin{align*} \P (G_0v_{||E||_1^{1/3}} & (d, k ||E||_1^{1/3}, E))
 =  \P (G_0v (d, (k-1) ||E||_1^{1/3}, E)) \\ & \leq  \P (G_0v (d, k ||E||_1^{1/3}, 0) + \mathcal{O}_{\upsilon,d} 
\left( N^{-3/2} \sum_{\ell = 1}^{\infty} 
(1 + k \ell ||E||^{1/3}_1)^{3} e^{- \frac{\ell^2}{2 d^2 ||E||_1^{1/3}}} \right).
\end{align*}
By using Corollaries \ref{corollary:barrier_estimate} and \ref{corollary:shifted_barrier_estimate}, we get 
(for $N$ large enough depending on $r$): 
$$ \P (G_0v (d, k ||E||_1^{1/3}, 0) )
\ll_{d, \upsilon} N^{-3/2} ( 1+ k ||E||_1^{1/3})^3,
$$
and then, since we assume $||E||_1 < 1/192$, 
 $$\P (G_0v_{||E||_1^{1/3}}  (d, k ||E||_1^{1/3}, E)) \ll_{d, \upsilon}
 N^{-3/2} (1+ k^3),$$
which gives 
\begin{align*} 
&\P (G_0v_{||E||^{1/3}} (d, 0,
 E + 2 |E|))  \\ &  \leq 
 \P (G_0v_{||E||_1^{1/3}} (d, ||E||_1^{1/3}, E))
 + \mathcal{O}_{d, \upsilon} \left(
 \sum_{k=2}^{\infty}  N^{-3/2}(1+ k^3)
 e^{- \frac{(k-1)^2}{4 d^2 ||E||_1^{1/3}}} \right).
 \end{align*}
On the other hand, since 
$||E||_1 \leq ||E + 2 |E| \, ||_1 \leq 3 ||E||_1 
\leq [(1 \wedge \upsilon)/4]^3$, we deduce from the 
particular case $e_j \geq 0$ previously studied (indeed, $e_j + 2 |e_j| \geq 0$):
$$\P (G_0v_{||E||^{1/3}} (d, 0,
 E + 2 |E|)) \geq
 \P (G_0v_{||E + 2|E|\,||^{1/3}} (d, 0,
 E + 2 |E|)) \geq \P(G_0v (d, 0, 0)) (1 - \theta_1),
  $$
where $\theta_1$ satisfies the same conditions as above. 

We then get 
\begin{align*} & \P(G_0v (d, 0, 0)) (1 - \theta_1)
\\ & \leq \P(G_0v (d, 0, E)) + 
\mathcal{O}_{d, \upsilon} \left(
 \sum_{k=2}^{\infty}  N^{-3/2}(1+ k^3)
 e^{- \frac{(k-1)^2}{4 d^2 ||E||_1^{1/3}}} \right).
 \end{align*}
and then 
$$\P(G_0v (d, 0, 0)) (1 - \theta_1 - \theta_4)
\leq \P(G_0v (d, 0, E)),$$ 
where 
$$\theta_4 \ll_{d, \upsilon} 
[\P(G_0v (d, 0, 0))]^{-1} 
 \sum_{k=2}^{\infty}  N^{-3/2}(1+ k^3)
 e^{- \frac{(k-1)^2}{4 d^2 ||E||_1^{1/3}}}$$
 tends to zero with $||E||_1$ by dominated convergence and \eqref{eq:barrier_estimate}. 
 This gives the desired bound.

\end{proof}

\bibliographystyle{halpha}
\bibliography{Bib_Max_CBE}

\begin{thebibliography}{BHNY08}

\bibitem[ABB16]{bib:ABB}
{Louis Pierre} Arguin, David Belius, and Paul Bourgade.
\newblock Maximum of the characteristic polynomial of random unitary matrices.
\newblock {\em Communications in Mathematical Physics}, pages 1--49, 9 2016.

\bibitem[ABBS13]{ABBS11}
Elie A{\"\i}d{\'e}kon, Julien Berestycki, {\'E}ric Brunet, and Zhan Shi.
\newblock {Branching Brownian motion seen from its tip}.
\newblock {\em Probability Theory and Related Fields}, 157(1-2):405--451, 2013.

\bibitem[ABH15]{ABH15}
L.-P. {Arguin}, D.~{Belius}, and A.~J. {Harper}.
\newblock {Maxima of a randomized {R}iemann zeta function, and branching random
  walks}.
\newblock {\em ArXiv e-prints}, June 2015, 1506.00629.

\bibitem[ABK13]{ABK11}
Louis-Pierre Arguin, Anton Bovier, and Nicola Kistler.
\newblock {The extremal process of branching Brownian motion}.
\newblock {\em Probability Theory and related fields}, 157(3-4):535--574, 2013.

\bibitem[A{\"{\i}}d13]{Aid11}
Elie A{\"{\i}}d{\'e}kon.
\newblock Convergence in law of the minimum of a branching random walk.
\newblock {\em Ann. Probab.}, 41(3A):1362--1426, 2013.

\bibitem[AS10]{AShi10}
E.~A{\"{\i}}d{\'e}kon and Z.~Shi.
\newblock Weak convergence for the minimal position in a branching random walk:
  a simple proof.
\newblock {\em Period. Math. Hungar.}, 61(1-2):43--54, 2010.

\bibitem[AS14]{AShi11}
Elie Aidekon and Zhan Shi.
\newblock The {S}eneta-{H}eyde scaling for the branching random walk.
\newblock {\em Ann. Probab.}, 42(3):959--993, 2014.

\bibitem[AZ14]{bib:AZ14}
Louis-Pierre Arguin and Olivier Zindy.
\newblock Poisson-{D}irichlet statistics for the extremes of a log-correlated
  {G}aussian field.
\newblock {\em Ann. Appl. Probab.}, 24(4):1446--1481, 2014.

\bibitem[BDZ16]{BDZ13}
Maury Bramson, Jian Ding, and Ofer Zeitouni.
\newblock Convergence in law of the maximum of the two-dimensional discrete
  {G}aussian free field.
\newblock {\em Communications on Pure and Applied Mathematics}, 69(1):62--123,
  2016.

\bibitem[BHNY08]{bib:BHNY}
P.~Bourgade, C.-P. Hughes, A.~Nikeghbali, and M.~Yor.
\newblock {The characteristic polynomial of a random unitary matrix: a
  probabilistic approach}.
\newblock {\em Duke Math. J.}, 145(1):45--69, 2008.

\bibitem[BK14]{BK14}
David Belius and Nicola Kistler.
\newblock The subleading order of two dimensional cover times.
\newblock {\em Probability Theory and Related Fields}, pages 1--92, 2014.

\bibitem[BL]{BL16}
Marek Biskup and Oren Louidor.
\newblock Extreme local extrema of two-dimensional discrete {G}aussian free
  field.
\newblock {\em Communications in Mathematical Physics}, pages 1--34.

\bibitem[BNR09]{bib:BNR09}
Paul Bourgade, Ashkan Nikeghbali, and Alain Rouault.
\newblock Circular {J}acobi ensembles and deformed {V}erblunsky coefficients.
\newblock {\em Int. Math. Res. Not. IMRN}, (23):4357--4394, 2009.

\bibitem[Bra78]{Bra78}
M.~D. Bramson.
\newblock Minimal displacement of branching random walk.
\newblock {\em Z. Wahrsch. Verw. Gebiete}, 45(2):89--108, 1978.

\bibitem[BZ12]{BZe10}
Maury Bramson and Ofer Zeitouni.
\newblock Tightness of the recentered maximum of the two-dimensional discrete
  {G}aussian free field.
\newblock {\em Communications on Pure and Applied Mathematics}, 65(1):1--20,
  2012.

\bibitem[CNN16]{CNN}
Reda Chhaibi, Joseph Najnudel, and Ashkan Nikeghbali.
\newblock The circular unitary ensemble and the {R}iemann zeta function: the
  microscopic landscape and a new approach to ratios.
\newblock {\em Inventiones mathematicae}, pages 1--91, 2016.

\bibitem[CPH01]{bib:HKO}
N.~O'Connell C.-P.~Hughes, J.-P.~Keating.
\newblock {On the characteristic polynomial of a random unitary matrix}.
\newblock {\em Comm. Math. Phys.}, 2001.

\bibitem[DRZ15]{RDZ15}
J.~{Ding}, R.~{Roy}, and O.~{Zeitouni}.
\newblock {Convergence of the centered maximum of log-correlated Gaussian
  fields}.
\newblock {\em ArXiv e-prints}, March 2015, 1503.04588.

\bibitem[DS94]{bib:DS94}
Persi Diaconis and Mehrdad Shahshahani.
\newblock On the eigenvalues of random matrices.
\newblock {\em J. Appl. Probab.}, 31A:49--62, 1994.
\newblock Studies in applied probability.

\bibitem[DS11]{DS2}
Bertrand Duplantier and Scott Sheffield.
\newblock Liouville quantum gravity and {KPZ}.
\newblock {\em Invent. Math.}, 185(2):333--393, 2011.

\bibitem[FHK12]{bib:FHK}
Y.~V. {Fyodorov}, G.~A. {Hiary}, and J.~P. {Keating}.
\newblock {Freezing Transition, Characteristic Polynomials of Random Matrices,
  and the Riemann Zeta Function}.
\newblock {\em Physical Review Letters}, 108(17):170601, April 2012.

\bibitem[FK14]{bib:FyKe}
Y.-V Fyodorov and J.-P. Keating.
\newblock Freezing transitions and extreme values: random matrix theory and
  disordered landscapes.
\newblock {\em Philos. Trans. R. Soc. Lond. Ser. A Math. Phys. Eng. Sci.},
  372(2007), 20120503, 2014.

\bibitem[HS09]{HSh09}
Yueyun Hu and Zhan Shi.
\newblock Minimal position and critical martingale convergence in branching
  random walks, and directed polymers on disordered trees.
\newblock {\em Ann. Probab.}, 37(2):742--789, 03 2009.

\bibitem[JM15]{bib:JM15}
Tiefeng Jiang and Sho Matsumoto.
\newblock Moments of traces of {C}ircular $\beta$ {E}nsembles.
\newblock {\em Ann. Probab.}, 43(6):3279--3336, 11 2015.

\bibitem[Joh97]{bib:J97}
Kurt Johansson.
\newblock On random matrices from the compact classical groups.
\newblock {\em Ann. of Math. (2)}, 145(3):519--545, 1997.

\bibitem[Kah85]{K}
Jean-Pierre Kahane.
\newblock Sur le chaos multiplicatif.
\newblock {\em Ann. Sci. Math. Qu\'ebec}, 9(2):105--150, 1985.

\bibitem[Kis15]{bib:Kistler}
Nicola Kistler.
\newblock Derrida's random energy models. {F}rom spin glasses to the extremes
  of correlated random fields.
\newblock In {\em Correlated random systems: five different methods}, volume
  2143 of {\em Lecture Notes in Math.}, pages 71--120. Springer, Cham, 2015.

\bibitem[KN04]{bib:KN04}
Rowan Killip and Irina Nenciu.
\newblock Matrix models for circular ensembles.
\newblock {\em Int. Math. Res. Not.}, (50):2665--2701, 2004.

\bibitem[Koz76]{Koz76}
M.~V. Kozlov.
\newblock The asymptotic behavior of the probability of non-extinction of
  critical branching processes in a random environment.
\newblock {\em Teor. Verojatnost. i Primenen.}, 21(4):813--825, 1976.

\bibitem[KPZ88]{KPZ}
V.~G. Knizhnik, A.~M. Polyakov, and A.~B. Zamolodchikov.
\newblock Fractal structure of {$2$}{D}-quantum gravity.
\newblock {\em Modern Phys. Lett. A}, 3(8):819--826, 1988.

\bibitem[KS00]{bib:KSn}
J.-P. Keating and N.~Snaith.
\newblock {Random Matrix Theory and $\zeta(1/2 + it)$ }.
\newblock {\em Commun. Math. Physics}, 214:57--89, 2000.

\bibitem[KS09]{bib:KSt09}
Rowan Killip and Mihai Stoiciu.
\newblock Eigenvalue statistics for {CMV} matrices: from {P}oisson to clock via
  random matrix ensembles.
\newblock {\em Duke Math. J.}, 146(3):361--399, 2009.

\bibitem[Luk55]{bib:Luk55}
Eugene Lukacs.
\newblock A characterization of the gamma distribution.
\newblock {\em Ann. Math. Statist.}, 26:319--324, 1955.

\bibitem[Mad15a]{bib:Mad13}
T.~Madaule.
\newblock {Maximum of a log-correlated Gaussian field}.
\newblock {\em Ann. Inst. H. Poincar\'e Probab. Statist.}, 51(4):1369--1431, 11
  2015.

\bibitem[Mad15b]{Mad15}
Thomas Madaule.
\newblock Convergence in law for the branching random walk seen from its tip.
\newblock {\em Journal of Theoretical Probability}, pages 1--37, 2015.

\bibitem[PZ16]{bib:PZ}
E.~{Paquette} and O.~{Zeitouni}.
\newblock {The maximum of the CUE field}.
\newblock {\em ArXiv e-prints}, February 2016, 1602.08875.

\bibitem[RV14]{RV}
R{\'e}mi Rhodes and Vincent Vargas.
\newblock Gaussian multiplicative chaos and applications: a review.
\newblock {\em Probab. Surv.}, 11:315--392, 2014.

\bibitem[RY99]{bib:RY}
Daniel Revuz and Marc Yor.
\newblock {\em Continuous martingales and {B}rownian motion}, volume 293 of
  {\em Grundlehren der Mathematischen Wissenschaften [Fundamental Principles of
  Mathematical Sciences]}.
\newblock Springer-Verlag, Berlin, third edition, 1999.

\bibitem[Sim05]{bib:Sim}
Barry Simon.
\newblock {\em Orthogonal polynomials on the unit circle. {P}art 1}, volume~54
  of {\em American Mathematical Society Colloquium Publications}.
\newblock American Mathematical Society, Providence, RI, 2005.
\newblock Classical theory.

\bibitem[Web16]{W}
Christian Webb.
\newblock Linear statistics of the {C}ircular $\beta$-{E}nsemble, {S}tein’s
  method and circular {D}yson {B}rownian motion.
\newblock {\em Electron. J. Probab.}, 21:16 pp., 2016.

\end{thebibliography}

\end{document}